\documentclass[a4paper,reqno,10pt]{amsart}

\usepackage{amsmath}
\usepackage{amsthm}
\usepackage{amssymb}
\usepackage{amscd} 

\usepackage{bbm} 
\usepackage{mathrsfs} 

\usepackage{lmodern} 
\usepackage[T1]{fontenc} 
\usepackage{newtxtext,newtxmath}

\usepackage{manfnt} 

\usepackage{graphicx} 

\usepackage{verbatim} 

\usepackage[Lenny]{fncychap} 

\usepackage{sqrcaps} 

\usepackage{comment}

\usepackage[svgnames]{xcolor}
\usepackage{pdfcolmk}

\usepackage{tikz-cd}
\usepackage{pifont}

\swapnumbers 

\newtheoremstyle{exercise} 
  {3pt} 
  {3pt} 
  {\small\rmfamily} 
  {
} 
  {\rmfamily\scshape} 
  {.} 
  {.5em} 
  {} 

\newtheoremstyle{newplain}
  {5pt}
  {5pt}
  {\itshape}
  {}
  {\rmfamily\scshape}
  {. ---}
  {.5em}
  {}

\newtheoremstyle{newremark}
  {5pt}
  {5pt}
  {\rmfamily}
  {}
  {\rmfamily\scshape}
  {. ---}
  {.5em}
  {}

\theoremstyle{newplain}
\newtheorem*{Theorem*}{Theorem} 

\theoremstyle{newplain}
\newtheorem{Theorem}{Theorem}
\newtheorem{Lemma}[Theorem]{Lemma}
\newtheorem{Corollary}[Theorem]{Corollary}
\newtheorem{Proposition}[Theorem]{Proposition}
\newtheorem{Conjecture}[Theorem]{Conjecture}
\newtheorem{Definition}[Theorem]{Definition}

\theoremstyle{newremark}
\newtheorem{Empty}[Theorem]{}
\newtheorem{Remark}[Theorem]{Remark}

\newtheorem{Claim}[Theorem]{Claim}

\theoremstyle{exercise}

\numberwithin{Theorem}{section}
\numberwithin{Exercise}{section}

\newcommand{\N}{\mathbb{N}} 

\newcommand{\R}{\mathbb{R}} 
\newcommand{\Rm}{\R^m}
\newcommand{\Rn}{\R^n}

\newcommand{\ind}{\mathbbm{1}} 

\newcommand{\calA}{\mathscr{A}}
\newcommand{\calB}{\mathscr{B}}
\newcommand{\calC}{\mathscr{C}}

\newcommand{\calE}{\mathscr{E}}
\newcommand{\calF}{\mathscr{F}}
\newcommand{\calG}{\mathscr{G}}
\newcommand{\calH}{\mathscr{H}}
\newcommand{\calI}{\mathscr{I}}

\newcommand{\calK}{\mathscr{K}}
\newcommand{\calL}{\mathscr{L}}
\newcommand{\calM}{\mathscr{M}}
\newcommand{\calN}{\mathscr{N}}

\newcommand{\calP}{\mathscr{P}}

\newcommand{\calU}{\mathscr{U}}

\newcommand{\calZ}{\mathscr{Z}}

\newcommand{\frA}{\frak A}
\newcommand{\frB}{\frak B}

\newcommand{\frE}{\frak E}

\newcommand{\sfForget}{\mathsf{Forget}}

\newcommand{\sfLOC}{\mathsf{LOC}}

\newcommand{\sfMSN}{\mathsf{MSN}}

\newcommand{\sfLOCsp}{\mathsf{LOC_{sp}}}
\newcommand{\sfLLDsp}{\mathsf{LLD_{sp}}}

\newcommand{\balpha}{\boldsymbol{\alpha}}

\newcommand{\biota}{\boldsymbol{\iota}}

\newcommand{\bB}{\mathbf{B}}

\newcommand{\bE}{\mathbf{E}}

\newcommand{\bL}{\mathbf{L}}
\newcommand{\bM}{\mathbf{M}}

\newcommand{\bZ}{\mathbf{Z}}

\newcommand{\bg}{\mathbf{g}}
\newcommand{\bh}{\mathbf{h}}

\newcommand{\bk}{\mathbf{k}}

\newcommand{\bp}{\mathbf{p}}
\newcommand{\bq}{\mathbf{q}}
\newcommand{\br}{\mathbf{r}}
\newcommand{\bs}{\mathbf{s}}

\DeclareMathOperator{\rmcard}{\mathrm{card}} 

\DeclareMathOperator{\rmesssup}{\mathrm{ess\,sup}} 

\DeclareMathOperator{\rmHom}{\mathrm{Hom}} 
\DeclareMathOperator{\rmid}{\mathrm{id}} 

\DeclareMathOperator{\rmloc}{\mathrm{loc}}

\newcommand{\rmpu}{\mathrm{pu}}
\DeclareMathOperator{\rmTan}{\mathrm{Tan}} 

\newcommand{\rmccc}{\mathrm{ccc}}

\newcommand{\hel} {
\hskip2.5pt{\vrule height7pt width.5pt depth0pt}
\hskip-.2pt\vbox{\hrule height.5pt width7pt depth0pt}
\, }

\newcommand{\shel} {
\hskip2.5pt{\vrule height5pt width.4pt depth0pt}
\hskip-.2pt\vbox{\hrule height.4pt width5pt depth0pt}
\, }

\def\XXint#1#2#3{{%
\setbox0=\hbox{$#1{#2#3}{\int}$}
\vcenter{\hbox{$#2#3$}}\kern-.5\wd0}}

\newcommand{\vphi}{\varphi}

\newcommand{\la}{\langle}
\newcommand{\ra}{\rangle}

\renewcommand{\em}{\bf}

\renewcommand{\leq}{\leqslant}
\renewcommand{\geq}{\geqslant}
\renewcommand{\subset}{\subseteq}
\renewcommand{\supset}{\supseteq}

\newlength{\drop}

\usepackage{trajan}

\newcommand{\defeq}{\mathrel{\mathop:}=}
\newcommand{\sfMSNsp}{\mathsf{MSN_{sp}}}

\begin{document}



\title[Localizable locally determined measurable spaces with
  negligibles]{Localizable locally determined\\ measurable spaces with
  negligibles}

\def\curraddrname{{\itshape On leave of absence from}}

\author[Ph. Bouafia]{Philippe Bouafia}

\address{F\'ed\'eration de Math\'ematiques FR3487 \\
  CentraleSup\'elec \\
  3 rue Joliot Curie \\
  91190 Gif-sur-Yvette
}

\email{philippe.bouafia@centralesupelec.fr}

\author[Th. De Pauw]{Thierry De Pauw}

\address{School of Mathematical Sciences\\
Shanghai Key Laboratory of PMMP\\ 
East China Normal University\\
500 Dongchuang Road\\
Shanghai 200062\\
P.R. of China\\
and NYU-ECNU Institute of Mathematical Sciences at NYU Shanghai\\
3663 Zhongshan Road North\\
Shanghai 200062\\
China}
\curraddr{Universit\'e Paris Diderot\\ 
Sorbonne Universit\'e\\
CNRS\\ 
Institut de Math\'ematiques de Jussieu -- Paris Rive Gauche, IMJ-PRG\\
F-75013, Paris\\
France}
\email{thdepauw@math.ecnu.edu.cn,thierry.de-pauw@imj-prg.fr}



\thanks{The first author was partially supported by the Science and Technology Commission of Shanghai (No. 18dz2271000).}



\begin{abstract}
  We study measurable spaces equipped with a
  $\sigma$-ideal of negligible sets. We find conditions under which
  they admit a localizable locally determined version -- a kind of fiber
  space that describes locally their directions -- defined by a universal property in an appropriate category that we introduce. These methods allow to promote each measure space $(X, \calA,
  \mu)$ to a strictly localizable version $(\hat{X}, \hat{\calA},
  \hat{\mu})$, so that the dual of $\bL_1(X, \calA, \mu)$ is
  $\bL_\infty(\hat{X}, \hat{\calA}, \hat{\mu})$. Corresponding to this
  duality is a generalized Radon-Nikod\'ym theorem. We also provide a
  characterization of the strictly localizable version in special
  cases that include integral geometric measures, when the negligibles are the purely unrectifiable sets in a given dimension.
\end{abstract}

\maketitle

\tableofcontents


\section{Foreword}

The Radon-Nikod\'ym Theorem does not hold for every measure space
$(X,\calA,\mu)$.
One way to phrase this precisely is to consider the canonical
embedding
\begin{equation*}
  \Upsilon : \bL_\infty(X,\calA,\mu) \to \bL_1(X,\calA,\mu)^*.
\end{equation*}
The following hold.
\begin{enumerate}
\item[(A)] $\Upsilon$ is injective (this corresponds to the uniqueness
  almost everywhere of Radon-Nikod\'ym derivatives) if and only if
  $(X,\calA,\mu)$ is semi-finite.
\item[(B)] $\Upsilon$ is surjective (this corresponds to the existence
  of Radon-Nikod\'ym derivatives) if and only if the Boolean algebra
  $\calA / \calN_{\mu,\rmloc}$ is Dedekind complete (i.e. order
  complete as a lattice).
\end{enumerate}
While (A) is classical, see e.g. \cite[243G(a)]{FREMLIN.II}, (B) is
recent and due to the second author, see \cite[4.6]{DEP.19.b}.
Let us recall the relevant definitions.
Given a measure space $(X,\calA,\mu)$, we abbreviate $\calA^f \defeq
\calA \cap \{ A : \mu(A) < \infty \}$ and $\calN_\mu \defeq \calA \cap
\{ N : \mu(N) = 0 \}$.
We say that $(X,\calA,\mu)$ is {\em semi-finite} if every $A \in
\calA$ of infinite measure contains some $F \in \calA^f \setminus
\calN_\mu$.
Equivalently, $\mu(A) = \sup \{ \mu(F) : A \supset F \in \calA^f \}$.
We further define the $\sigma$-ideal of {\em locally $\mu$-null} sets
as follows: $\calN_{\mu,\rmloc} \defeq \calA \cap \{ A : A \cap F \in
\calN_\mu \text{ for all } F \in \calA^f \}$.
It is easy to see \cite[4.4]{DEP.19.b} that $(X,\calA,\mu)$ is
semi-finite if and only if $\calN_{\mu,\rmloc}=\calN_\mu$.
Thus we obtain the following classical criterion, \cite[243G(b)]{FREMLIN.II}.
\begin{enumerate}
\item[(C)] $\Upsilon$ is an isometric isomorphism if and only if
  $(X,\calA,\mu)$ is semi-finite and the Boolean algebra $\calA /
  \calN_\mu$ is Dedekind complete.
\end{enumerate}
Though semi-finiteness is a natural property, Caratheodory's method
does not always provide it.
For instance, the measure spaces $(\R^2,\calA_{\calH^1},\calH^1)$ and
$(\R^2,\calB(\R^2),\calI^1_\infty)$ are not semi-finite -- see
\cite[439H]{FREMLIN.IV} and \cite[3.3.20]{GMT}.
Here, $\calH^1$ is the 1-dimensional Hausdorff measure in the
Euclidean plane \cite[2.10.2]{GMT}, $\calA_{\calH^1}$ is the
$\sigma$-algebra consisting of $\calH^1$-measurable sets in
Caratheodory's sense, $\calI^1_\infty$ is a 1-dimensional integral
geometric measure \cite[2.10.5(1)]{GMT}, and $\calB(\R^2)$ is the
$\sigma$-algebra whose members are the Borel subsets of $\R^2$.
Both $\calH^1(\Gamma)$ and $\calI^1_\infty(\Gamma)$ coincide with the
usual Euclidean length of $\Gamma$ when this is a Lipschitz curve.
\par It is natural to want to associate, with an arbitrary
$(X,\calA,\mu)$, an improved version of itself -- in a universal way
-- ideally one for which the Radon-Nikod\'ym Theorem holds.
This is one of our several achievements in this paper.
It is not difficult to modify slightly the measure $\mu$, keeping the
underlying measurable space $(X,\calA)$ untouched, in order to make it
semi-finite.
Specifically, letting $\mu_{\mathrm{sf}}(A) = \sup \{ \mu(A \cap F) :
F \in \calA^f \}$, for $A \in \calA$, one checks that
$(X,\calA,\mu_{\mathrm{sf}})$ is semi-finite and that
$\calN_{\mu_{\mathrm{sf}}} = \calN_{\mu,\rmloc}$.
However, it appears to be a more delicate task to modify
$(X,\calA,\mu)$ in a canonical way in order for $\Upsilon$ to become
surjective.
\par An idea for testing whether $\Upsilon$ is surjective is as
follows.
Given $\alpha \in \bL_1(X,\calA,\mu)^*$ we apply the Radon-Nikod\'ym
Theorem ``locally'', as it is valid on each finite measure subspace
$(F,\calA_F,\mu_F)$, $F \in \calA^f$, i.e. we represent by integration
the functional $\alpha \circ \iota_F \in \bL_1(F,\calA_F,\mu_F)^*$,
where $\iota_F \colon \bL_1(F, \calA_F, \mu_F) \to \bL_1(X, \calA,
\mu)$ is the obvious map.
This produces a family of Radon-Nikod\'ym derivatives $\la f_F \ra_{F
  \in \calA^f}$.
By the almost everywhere uniqueness of Radon-Nikod\'ym derivatives in
finite measure spaces, this is a {\em compatible family} in the sense
that $F \cap F' \cap \{ f_F \neq f_{F'} \} \in \calN_\mu$, for every
$F,F' \in \calA^f$.
In order to obtain a globally defined Radon-Nikod\'ym derivative, one
ought to be able to ``glue'' together the functions of this family.
A {\em gluing} of $\la f_F \ra_{F \in \calA^f}$ is, by definition, an
$\calA$-measurable function $f \colon X \to \R$ such that $F \cap \{ f
\neq f_F \} \in \calN_\mu$ for every $F \in \calA^f$.
\par The question whether such a gluing exists takes us away from the
realm of measure spaces, as it rather pertains to {\em measurable
  spaces with negligibles}, abbreviated MSNs, i.e. triples
$(X,\calA,\calN)$ where $(X,\calA)$ is a measurable space and $\calN
\subset \calA$ is a $\sigma$-ideal.
The notion of compatible family $\la f_E \ra_{E \in \calE}$ of
$\calA_E$-measurable functions $E \to \R$ subordinated to an arbitrary
collection $\calE \subset \calA$ readily makes sense in this more
general setting, as does the notion of gluing of such a compatible
family.
That each compatible family of partially defined measurable functions
admits a gluing is equivalent to the Boolean algebra $\calA / \calN$
being Dedekind complete.
In this case we say that $(X,\calA,\calN)$ is {\em localizable}.
Equivalently, $(X,\calA,\calN)$ is localizable if and only if every
collection $\calE \subset \calA$ admits an $\calN$-essential supremum
(see \ref{essentialSup} for a definition), which corresponds to taking
an actual supremum in the Boolean algebra $\calA / \calN$.
For a proof of these classical equivalences, see
e.g. \cite[3.13]{DEP.19.b}.
It will be convenient to call {\em $\calN$-generating} a collection
$\calE \subset \calA$ that admits $X$ as an $\calN$-essential
supremum.
For instance, one easily checks that if $(X,\calA,\mu)$ is a
semi-finite measure space, then $\calA^f$ is $\calN_\mu$-generating.
\par Let $(X,\calA,\calN)$ be a localizable MSN, $\calE \subset \calA$
be $\calN$-generating, and $\la f_E \ra_{E \in \calE}$ be a compatible
family of partially defined measurable functions.
The problem of gluing this compatible family in our setting is
reminiscent of the fact that, for a topological space $X$, the functor
of continuous functions on open sets is a sheaf.
However, unlike in the case of continuous functions, in order to
define $f$ globally, we ought to make choices on the domains $E \cap
E'$, for $E,E' \in \calE$, because $f_E$ and $f_{E'}$ do not coincide
everywhere there, but merely {\it almost everywhere}.
In an attempt to avoid the issue, one can replace $\calE$ with an
almost disjointed refinement of itself, say $\calF$.
By this we mean that each member of $\calF$ is contained in a member
of $\calE$, that $\calF$ is $\calN$-generating, and that $F \cap F'
\in \calN$ whenever $F,F' \in \calF$ are distinct.
The existence of $\calF$ follows from Zorn's Lemma, see \ref{zorn}.
Still, $F \cap F'$ may not be empty whenever $F,F' \in \calF$ are
distinct and we are again in a position to make choices.
A step further along the road would be to produce from $\calF$ a
disjointed family $\calG$ whose union is conegligible.
From classical measure theory, we learn of two situations when this is
doable.
First, in the presence of a lower density of $(X,\calA,\calN)$ (see
\ref{ldens} for a definition), and second when $\rmcard \calF \leq
\mathfrak{c}$ (see the proof of \ref{existslld}).
In those cases, a gluing exists.
In fact, in the context of measure spaces, the existence of a lower
density yields a somewhat stronger structure than localizability.
In order to state this, we need one more definition.
We say that a measure space $(X,\calA,\mu)$ is {\em locally
  determined} if it is semi-finite and if the following holds:
\begin{equation*}
  \forall A \subset X : \big[ \forall F \in \calA^f : A \cap F \in
    \calA \big] \Rightarrow A \in \calA.
\end{equation*}
A complete locally determined measure space $(X,\calA,\mu)$ admits a
lower density if and only if it is strictly localizable, which means,
by definition, that there exists a partition $\calG \subset \calA$ of
$X$ such that $\calA = \calP(X) \cap \{ A : A \cap G \in \calA \text{
  for all } G \in \calG \}$ and $\mu(A) = \sum_{G \in \calG} \mu(A
\cap G)$, for $A \in \calA$.
See \cite[341M]{FREMLIN.III} for a proof.
The existence of a lower density for a strictly localizable measure
space follows from the case of finite measure spaces by gluing, and
the case of finite measure spaces is a consequence of a martingale
convergence theorem.
Even though the notion of a lower density makes sense for MSNs, their
existence does not hold for even the most natural generalization of
finite measure spaces, namely ccc MSNs (satisfying the countable chain
condition, \ref{ccc.def} and \ref{finite=>loc}), see \cite{SHE.98}.
\par Both notions of localizability (of an MSN) and local
determination (of a measure space) seem to express in different ways
the fact that ``there are enough measurable sets''.
For instance, one easily checks that an MSN $(X,\calA,\{\emptyset\})$,
such that $\calA$ contains all singletons, is localizable if and only
if $\calA = \calP(X)$.
Thus, given an arbitrary MSN $(X,\calA,\calN)$, one may naively
attempt to ``add measurable sets'' in a smart way in order to obtain a
localizable MSN $(X,\hat{\calA},\hat{\calN})$, just as many as needed,
and that would be a ``localizable version'' of $(X,\calA,\calN)$.
Unfortunately, within ZFC this cannot always be done while ``sticking
in the base space $X$'', as shown by the following, quoted from
\cite{DEP.19.b}.
\begin{Theorem*}
  Assume that:
  \begin{enumerate}
  \item[(1)] $C \subset [0,1]$ is some Cantor set of Hausdorff dimension 0;
  \item[(2)] $X = C \times [0,1]$;
  \item[(3)] $\calA$ is a $\sigma$-algebra such that $\calB(X) \subset
    \calA \subset \calP(X)$;
  \item[(4)] $\calN = \calN_{\calH^1}$ or $\calN = \calN_{\rmpu}$.
  \end{enumerate}
  Then $(X,\calA,\calN)$ is consistently not localizable.
\end{Theorem*}
Here, $\calN_{\rmpu}$ consists of those subsets $S$ of $X$ that are {\em
  purely 1-unrectifiable}, i.e. $\calH^1(S \cap \Gamma) = 0$ for every
Lipschitz (or, for that matter, $C^1$) curve $\Gamma \subset \R^2$.
Thus, one may need to also add points to the base space $X$ and, in
particular cases such as the one above, we give a very specific way of
doing so, in the last section of this paper.
In the case of a general measure space $(X,\calA,\mu)$, we can get a
feeling of what needs to be done, when trying to define the gluing of
a compatible family $\la f_F \ra_{F \in \calA^f}$.
Indeed, each $x \in X$ may belong to several $F \in \calA^f$ and this
calls for considering an appropriate quotient of the fiber bundle
$\{(x,F) : x \in F \in \calA^f\}$.  \par One of the tasks that we
assign ourselves in this paper is to define a general notion of
``localization'' of an MSN and to prove existence results in some
cases.
Since a definition of ``localization'' will involve a universal
property, it is critical to determine which category is appropriate
for our purposes.
As this offers unexpected surprises, we describe the several steps in
some detail.
The objects of our first category $\sfMSN$ are the {\em saturated} MSN
$(X,\calA,\calN)$, by what we mean that for every $N,N' \subset X$, if
$N \subset N'$ and $N' \in \calN$, then $N \in \calN$.
This is in analogy with the notion of a complete measure space. 
In order to define the morphisms between two objects $(X,\calA,\calN)$
and $(Y,\calB,\calM)$, we say that a map $f \colon X \to Y$ is
$[(\calA,\calN),(\calB,\calM)]$-measurable if $f^{-1}(B) \in \calA$
for every $B \in \calB$ and $f^{-1}(M) \in \calN$ for every $M \in
\calM$.
For instance, if $X$ is a Polish space and $\mu$ is a diffuse
probability measure on $X$, there exists \cite[3.4.23]{SRIVASTAVA} a
Borel isomorphism $f \colon X \to [0,1]$ such that $f_\# \mu = \calL^1$,
where $\calL^1$ is the Lebesgue measure, thus $f$ is
$[(\calB(X),\calN_\mu),(\calB([0,1]),\calN_{\calL^1})]$-measurable.
We define an equivalence relation for such measurable maps $f,f' \colon X
\to Y$ by saying that $f \sim f'$ if and only if $\{ f \neq f' \} \in
\calN$.
The morphisms in the category $\sfMSN$ between the objects
$(X,\calA,\calN)$ and $(Y,\calB,\calM)$ are the equivalence classes of
$[(\calA,\calN),(\calB,\calM)]$-measurable maps.
At this stage, we need to suppose that $(X, \calA, \calN)$ is
saturated for the relation of equality almost everywhere to be
transitive~\ref{MSNMorphism}. With this assumption, the composition of
measurable maps is also compatible with $\sim$, see~\ref{lemma28}.
\par We let $\sfLOC$ be the full subcategory of $\sfMSN$ whose objects
are the localizable MSNs.
We may be tempted to define the localization of a saturated MSN
$(X,\calA,\calN)$ as its coreflection (if it exists) along the
forgetful functor $\sfForget \colon \sfLOC \to \sfMSN$, and the question of
existence in general becomes that of the existence of a right adjoint
to $\sfForget$.
Specifically, we may want to say that a pair
$[(\hat{X},\hat{\calA},\hat{\calN}),\bp]$, where
$(\hat{X},\hat{\calA},\hat{\calN})$ is saturated localizable MSN and
$\bp$ is a morphism $\hat{X} \to X$, is a localization of
$(X,\calA,\calN)$ whenever the following universal property holds.
For every pair $[(Y,\calB,\calM),\bq]$, where $(Y,\calB,\calM)$ is a
saturated localizable MSN and $\bq$ is a morphism $Y \to X$, there
exists a unique morphism $\br \colon Y \to \hat{X}$ such that $\bq = \bp
\circ \br$.
\begin{equation}
\label{diagram}
  \begin{tikzcd}
    (Y, \calB, \calM) \arrow[rd, swap, "\bq"] \arrow[rr, dotted,
      "\exists!\br"] & & (\hat{X}, \hat{\calA}, \hat{\calN})
    \arrow[ld, "\bp"] \\ & (X, \calA, \calN) &
  \end{tikzcd} \tag{$\bigtriangledown$}
\end{equation}
\par However, we now illustrate that the notion of morphism defined so
far is not yet the appropriate one that we are after.
We consider the MSN $(X,\calA,\{\emptyset\})$ where $X = \R$ and
$\calA$ is the $\sigma$-algebra of Lebesgue measurable subsets of
$\R$.
We recall that we want the localization of $(X,\calA,\{\emptyset\})$
to be $[(X,\calP(X),\{\emptyset\}),\bp]$ with $\bp$ induced by the
identity $\rmid_X$.
Assume if possible that this is the case.
In the diagram above we consider $(Y, \calB,\calM) =
(X,\calA,\calN_{\calL^1})$ and $\bq$ induced by the identity.
Note that this is, indeed, a localizable MSN since it is associated
with a $\sigma$-finite measure space (see \ref{finite=>loc} and
\ref{ccc=>loc}).
Thus, there would exist a morphism $\br$ in $\sfMSN$ such that $\bp
\circ \br = \bq$.
Picking $r \in \br$, this implies that $X \cap \{ x: r(x) \neq x \}$
is Lebesgue negligible.
The measurability of $r$ would then imply that $r^{-1}(S) \in \calA$
for every $S \in \calP(X)$, contradicting the existence of non
Lebesgue measurable subsets of $\R$.
\par The problem with the example above is that the objects
$(X,\calA,\calN_{\calL^1})$ and $(X,\calA,\{\emptyset\})$ should not
be compared, in other words that $\bq$ should not be a legitimate
morphism.
We say that a morphism $\mathbf{f} \colon (X,\calA,\calN) \to
(Y,\calB,\calM)$ of the category $\sfMSN$ is {\em supremum preserving}
if the following holds for (one, and therefore every) $f
\in\mathbf{f}$.
If $\calF \subset \calB$ admits an $\calM$-essential supremum $S \in
\calB$, then $f^{-1}(S)$ is an $\calN$-essential supremum of
$f^{-1}(\calF)$.
It is easy to see that adding this condition to the definition of
morphism rules out the $\bq$ considered in the preceding paragraph.
We define the category $\sfMSNsp$ to be that whose objects are the
saturated MSNs and whose morphisms are those morphisms of $\sfMSN$
that are supremum preserving.
We define similarly $\sfLOCsp$.
We now define the {\em localizable version} (if it exists) of a
saturated MSN with the similar universal property illustrated in
\eqref{diagram}, except for we now require all morphisms to be in
$\sfMSNsp$, i.e. supremum preserving.
In other words, it is a coreflection of an object of $\sfMSNsp$ along
$\sfForget \colon \sfLOCsp \to \sfMSNsp$.
Unfortunately, this is not quite yet the right setting.
Indeed, we show in \ref{counterexamplelv} that if $X$ is uncountable
and $\calC(X)$ is the countable-cocountable $\sigma$-algebra of $X$,
then $[(X,\calP(X),\{\emptyset\}),\biota]$ (with $\biota$ induced by
$\rmid_X$) is not the localizable version of
$(X,\calC(X),\{\emptyset\})$.
This prompts us to introduce a new category.
\par We say that an object $(X,\calA,\calN)$ of $\sfMSN$ is {\em
  locally determined} if for every $\calN$-generating collection
$\calE \subset \calA$ the following holds:
\begin{equation*}
  \forall A \subset X : \big[ \forall E \in \calE : A \cap E \in \calA
    \big] \Rightarrow A \in \calA.
\end{equation*}
In case $(X,\calA,\calN)$ is the MSN associated with some complete
semi-finite measure space $(X,\calA,\mu)$, then it is locally
determined (in the sense of MSNs) if and only if $(X,\calA,\mu)$ is
locally determined (in the sense of measure spaces) -- see
\ref{elemld}(F) -- even though the latter sounds stronger because we
test with any generating family $\calE$.
We say that an object of $\sfMSN$ is {\em lld} if it is both
localizable and locally determined, and we let $\sfLLDsp$ be the
corresponding full subcategory of $\sfLOCsp$.
We now define the {\em lld version} of an object of $\sfMSNsp$ to be
its coreflection (if it exists) along $\sfForget \colon \sfLLDsp \to
\sfMSNsp$, i.e. it satisfies the corresponding universal property
illustrated in \eqref{diagram} with $Y$ and $\hat{X}$ being lld, and
the morphisms being supremum preserving.
This definition is satisfactory in at least the simplest case,
\ref{atomic.lld} : If $(X,\calA,\{\emptyset\})$ is so that $\calA$
contains all singletons, then it admits
$[(X,\calP(X),\{\emptyset\}),\biota]$ as its lld version.
\par Our general question has now become whether $\sfForget : \sfLLDsp
\to \sfMSNsp$ admits a right adjoint.
Freyd's Adjoint Functor Theorem \cite[3.3.3]{BORCEUX.1} could prove
useful, however do not know whether it applies, mostly because we do
not know whether coequalizers exist in $\sfMSNsp$.
We gather in Table \ref{table} the information that we know about
limits and colimits in the three categories we introduced.

\begin{table}[h]
\label{table}
\begin{tabular}{|c|c|c|c|}
\hline
& $\sfMSNsp$ & $\sfLOCsp$ & $\sfLLDsp$  \\ \hline
\hline
equalizers &  exist if $\{f=g\}$ is meas. \ref{propsp}(C) & $\boldsymbol{?}$ & exist \ref{eq.lld} \\ \hline
products & (countable) exist \ref{prodMSN} & $\boldsymbol{?}$ & $\boldsymbol{?}$ \\ \hline
coequalizers & $\boldsymbol{?}$ & $\boldsymbol{?}$ see \ref{coeqlld} & $\boldsymbol{?}$ see \ref{coeqlld} \\ \hline
coproducts & exist \ref{propsp}(D) & exist \ref{coprodLoc} & exist \ref{coprodLoc} and \ref{elemld}(D)\\ \hline
\end{tabular}\hskip 1cm
\vskip .3cm
\caption{Limits and colimits in the three categories of MSNs.}
\end{table}
\par In view of proving some partial existence result for lld
versions, we introduce the intermediary notion of a {\em cccc}
saturated MSN, short for coproduct (in $\sfMSNsp$) of ccc saturated
MSNs.
It is easy to see that cccc MSNs are lld, \ref{coprodLoc} and \ref{ccc=>loc}.
The {\em cccc version} of an object of $\sfMSNsp$ is likewise defined
by its universal property in diagram \eqref{diagram}, using supremum
preserving morphisms.
Our main results are about {\em locally ccc} MSNs, i.e. those
saturated MSNs $(X,\calA,\calN)$ such that $\calE_{\rmccc} = \calA
\cap \{ Z : \text{ the subMSN } (Z,\calA_Z,\calN_Z) \text{ is ccc}\}$
is $\calN$-generating.
A complete semi-finite measure space $(X,\calA,\mu)$ is clearly
locally ccc, since $\calA^f$ is $\calN_\mu$-generating.
Similarly, one can define the more general {\em locally localizable}
objects in $\sfMSNsp$.
In \ref{4.14}, we give an example of an MSN which is not even locally
localizable.
\begin{Theorem*}
  Let $(X,\calA,\calN)$ be a saturated locally ccc MSN. The following
  hold.
  \begin{enumerate}
  \item[(1)] $(X,\calA,\calN)$ admits a cccc version,
    \ref{existsccccv}.
  \item[(2)] If furthermore $\calE_{\rmccc}$ contains an
    $\calN$-generating subcollection $\calE$ such that $\rmcard \calE
    \leq \mathfrak{c}$ and each $(Z,\calA_Z)$ is countably separated,
    for $Z \in \calE$, then $(X,\calA,\calN)$ admits an lld version
    which is also its cccc version.
  \end{enumerate}
\end{Theorem*}
\par By saying that a measurable space $(Z,\calA_Z)$ is countably
separated we mean that $\calA_Z$ contains a countable subcollection
that separates points in $Z$.
The cccc version $(\hat{X},\hat{\calA},\hat{\calN})$ is obtained as a
coproduct $\coprod_{Z \in \calE} (Z,\calA_Z,\calN_Z)$ where $\calE$ is
an $\calN$-generating almost disjointed refinement of
$\calE_{\rmccc}$, whose existence ensues from Zorn's Lemma.
In order to establish that this, in fact, is also the lld version
under the extra assumptions in (2), we need to build an appropriate
morphism $\br$ in diagram \eqref{diagram}, associated with an lld pair
$[(Y,\calB,\calM),\bq]$.
It is obtained as a gluing of $\la q_Z \ra_{Z \in \calE}$ where $q_Z
\colon q^{-1}(Z) \to \hat{X}$ is the obvious map.
Since $\calE$ is almost disjointed, $\la q_Z \ra_{Z \in \calE}$ is
compatible and, since $(Y,\calB,\calM)$ in diagram \eqref{diagram} is
localizable, the only obstruction to gluing is that $\hat{X}$ is not
$\R$.
Notwithstanding, $(\hat{X},\hat{\calA}) = \coprod_{Z \in \calE}
(Z,\calA_Z)$ is itself countably separated because $\rmcard \calE \leq
\mathfrak{c}$, \ref{prop68} so that the local determinacy of
$(Y,\calB,\calM)$ and the fact that $q^{-1}(\calE)$ is
$\calM$-generating (because $\calE$ is $\calN$-generating and $q$ is
supremum preserving) provides a gluing $r$, \ref{gluelld}.
\par We now explain how this applies to associating, in a canonical
way, a strictly localizable measure space with {\it any} measure space
$(X,\calA,\mu)$.
First, we recall that without changing the base space $X$ we can
render the measure space complete and semi-finite.
In that case, $\calA^f$ is $\calN_\mu$-generating and witnesses the
fact that the saturated MSN $(X,\calA,\calN_\mu)$ is locally ccc.
By the theorem above, it admits a cccc version
$[(\hat{X},\hat{\calA},\hat{\calN}),\bp]$.
\begin{Theorem*}
  Let $(X,\calA,\mu)$ be a complete semi-finite measure space and
  $[(\hat{X},\hat{\calA},\hat{\calN}),\bp]$ its corresponding cccc
  version. Let $p \in \bp$. There exists a unique (and independent of
  the choice of $p$) measure $\hat{\mu}$ defined on $\hat{\calA}$ such
  that $p_\# \hat{\mu} = \mu$ and $\calN_{\hat{\mu}}=\hat{\calN}$.
  Furthermore $(\hat{X},\hat{\calA},\hat{\mu})$ is a strictly
  localizable measure space, and the Banach spaces
  $\bL_1(X,\calA,\calN)$ and $\bL_1(\hat{X},\hat{\calA},\hat{\mu})$
  are isometrically isomorphic.
\end{Theorem*}
\par Of course, the general process for constructing $\hat{X}$ is non
constructive, as it involves the axiom of choice to turn $\calA^f$
into an almost disjointed generating family.
This is why, in the last two sections of this paper, we explore a
particular case where we are able to describe explicitly $\hat{X}$ as
a quotient of a fiber bundle, all ``hands on''.
We start with the measure space $(\Rm,\calB(\Rm),\calI^k_\infty)$
where $1 \leq k \leq m-1$ are integers, $\calB(\Rm)$ is the
$\sigma$-algebra of Borel subsets of $\Rm$, and $\calI^k_\infty$ is
the integral geometric measure described in \cite[2.10.5(1)]{GMT} and
\cite[5.14]{MATTILA}.
Note that it is not semi-finite, \cite[3.3.20]{GMT}.
Thus, we replace it with its complete semi-finite version
$(\Rm,\widetilde{\calB(\Rm)},\tilde{\calI}^k_\infty)$.
We let $\calE$ be the collection of $k$-dimensional submanifolds $M
\subset \Rm$ of class $C^1$ such that $\phi_M = \calH^k \hel M$ is
locally finite.
It follows from the Besicovitch Structure Theorem \cite[3.3.14]{GMT}
that $\calE$ is $\calN_{\tilde{\calI}^k_\infty}$-generating,
\ref{igm}(ii).
Now, for each $x \in \Rm$ we define $\calE_x = \calE \cap \{ M : x \in
M \}$ and we define on $\calE_x$ an equivalence relation as follows.
We declare that $M \sim_x M'$ if and only if
\begin{equation*}
  \lim_{r \to 0^+} \frac{\calH^k(M \cap M' \cap \bB(x,r))}{\balpha(k)r^k} 
= 1.
\end{equation*}
Letting $[M]_x$ denote the equivalence class of $M \in \calE_x$, we
prove \ref{igm} that underlying set of the cccc, lld, and strictly
localizable version of the MSN
$(\Rm,\widetilde{\calB(\Rm)},\calN_{\tilde{\calI}^k_\infty})$ can be
taken to be
\begin{equation*}
  \hat{X} = \{ (x,[M]_x ) : x \in \Rm \text{ and } M \in \calE_x \} \,.
\end{equation*}
This leads to an explicit description of the dual of
$\bL_1(\Rm,\widetilde{\calB(\Rm)},\tilde{\calI}^k_\infty)$ as
$\bL_\infty( \hat{X},\hat{\calA},\hat{\calN})$.
\par We are indebted to David Fremlin whose point of view on measure
theory -- generously shared in his immense treatise \cite{FREMLIN.I,
  FREMLIN.II, FREMLIN.III, FREMLIN.IV, FREMLIN.V.1,FREMLIN.V.2} -- influenced our
work in this paper. It is the second author's pleasure to record
useful conversations with Francis Borceux.

\section{Measurable spaces with negligibles}

\begin{Empty}[$\sigma$-algebra]
  Let $X$ be a set.  A {\em $\sigma$-algebra} on $X$ is a set $\calA
  \subset \calP(X)$ such that
  \begin{enumerate}
  \item[(1)] $\emptyset \in \calA$;
  \item[(2)] If $A \in \calA$ then $X \setminus A \in \calA$;
  \item[(3)] If $\la A_n \ra_{n \in \N}$ is a sequence in $\calA$ then
    $\bigcup_{n \in \N} A_n \in \calA$.
  \end{enumerate}
  If $\calA$ is a $\sigma$-algebra on $X$ then $X \in \calA$ and
  $\bigcap_{n \in \N} A_n \in \calA$ whenever $\la A_n \ra_{n \in \N}$ is
  a sequence in $\calA$.  Clearly $\{\emptyset,X\}$ and $\calP(X)$ are
  $\sigma$-algebras on $X$, respectively the coarsest and the finest.
  If $\la \calA_i \ra_{i \in I}$ is a nonempty family of
  $\sigma$-algebras on $X$ then $\bigcap_{i \in I} \calA_i$ is a
  $\sigma$-algebra on $X$.  Thus each $\calE \subset \calP(X)$ is
  contained in a coarsest $\sigma$-algebra on $X$ which we will denote
  by $\sigma(\calE)$.  If $\calE,\calA \subset \calP(X)$ and $\calA =
  \sigma(\calE)$ we say that the $\sigma$-algebra $\calA$ {\em is
    generated} by $\calE$.  Clearly, if $\calE_1 \subset \calE_2
  \subset \calP(X)$ then $\sigma(\calE_1) \subset \sigma(\calE_2)$.  A
  {\em measurable space} is a couple $(X,\calA)$ where $X$ is a set
  and $\calA$ is a $\sigma$-algebra on $X$.  In this case, if no
  confusion is possible we call {\em measurable} the members of
  $\calA$.
\end{Empty}

\begin{Empty}[Measurable maps]
  Let $(X,\calA)$ and $(Y,\calB)$ be measurable spaces and $f \colon X
  \to Y$.  We say that $f$ is {\em $(\calA,\calB)$-measurable} (or
  simply {\em measurable} if no confusion can occur) if $f^{-1}(B) \in
  \calA$ whenever $B \in \calB$. Measurable spaces, together with
  measurable maps, form a well-defined category, as one can check that
  the composition of two measurable maps is measurable.
\end{Empty}

\begin{Empty}[$\sigma$-ideal]
  Let $(X, \calA)$ be a measurable space. A {\em $\sigma$-ideal}
  $\calN$ of $\calA$ is a subset of $\calA$ that satisfies the
  following requirements:
  \begin{enumerate}
  \item[(1)] $\emptyset \in \calN$;
  \item[(2)] If $A \in \calA$, $N \in \calN$ and $A \subset N$ then $A \in \calN$;
  \item[(3)] If $\la N_n\ra_{n \in \N}$ is a sequence in $\calN$, then
    $\bigcup_{n \in\N} N_n \in \calN$.
  \end{enumerate}
\end{Empty}

\begin{Empty}[Measurable space with negligibles]
  A {\em measurable space with negligibles} (abbreviated {\em MSN}) is
  a triple $(X, \calA, \calN)$ where $(X, \calA)$ is a measurable
  space and $\calN$ is a $\sigma$-ideal of $\calA$.  Given an MSN $(X,
  \calA, \calN)$, elements belonging to $\calN$ are referred to as
       {\em $\calN$-negligible} sets (simply negligible sets if no
       confusion can occur). Complements of $\calN$-negligible sets
       are called {\em $\calN$-conegligible} sets (or simply
       conegligible sets).

  We can associate to any measure space $(X, \calA, \mu)$ the MSN $(X,
  \calA, \calN_\mu)$ where $\calN_\mu$ is the $\sigma$-ideal
  $\calN_{\mu} = \calA \cap \{N : \mu(N) = 0\}$.  Conversely, any MSN
  $(X, \calA, \calN)$ derives from a measure space: it suffices to
  consider the measure $\mu \colon\calA \to [0,
    \infty]$ that sends negligible sets to $0$ and the remaining sets
  to $\infty$.
\end{Empty}

\begin{Empty}[Saturated MSNs]
  An MSN $(X, \calA, \calN)$ is called {\em saturated} whenever the
  following property holds: For all $N \in \calN$, any subset $N'
  \subset N$ is $\calA$-measurable -- therefore, in fact, $N' \in \calN$. 
This
  property is of purely technical nature, as an MSN $(X, \calA,
  \calN)$ that does not have it can be turned into a saturated MSN
  ($X, \bar{\calA}, \bar{\calN})$, by setting:
  \begin{equation*}
  \begin{split}
    \bar{\calN} & = \calP(X) \cap \{\bar{N} : \bar{N}
    \subset N \text{ for some }  N \in \calN\} \\ 
    \bar{\calA}&  = \calP(X) \cap \{\bar{A} : A \ominus \bar{A} \in \bar{\calN} \text{ for some } A \in \calA\} \\
    & = \calP(X) \cap \{ A \ominus \bar{N} : A \in \calA \text{ and } \bar{N} \in \bar{\calN} \} .
    \end{split}
  \end{equation*}
  Here, $\ominus$ denotes the symmetric difference of sets.  We call
  $(X,\bar{\calA},\bar{\calN})$ the {\em saturation} of
  $(X,\calA,\calN)$. In case the original MSN corresponds to a measure
  space $(X,\calA,\mu)$, its saturation corresponds to the measure
  space usually referred to as the completion of $(X,\calA,\mu)$. We
  will denote the latter by $(X,\bar{\calA},\bar{\mu})$.
\end{Empty}

\begin{Empty}
  \label{ANBMMeasurability}
  Let $(X, \calA, \calN)$ and $(Y, \calB, \calM)$ be two MSNs. We say
  that a map $f \colon X \to Y$ is {\em $[(\calA, \calN), (\calB,
      \calM)]$-measurable} (or simply {\em measurable}) if
  \begin{enumerate}
  \item[(1)] $f$ is $(\calA, \calB)$-measurable;
  \item[(2)] $f^{-1}(M) \in \calN$ for every $M \in \calM$.
  \end{enumerate}
  It is easy to check that measurability in the above sense is
  preserved by composition.
\end{Empty}

\begin{Empty}[Morphisms of saturated MSNs]
  \label{MSNMorphism}
  Let $(X, \calA, \calN)$ and $(Y, \calB, \calM)$ be two saturated
  MSNs. A {\em morphism} from $(X, \calA, \calN)$ to $(Y, \calB,
  \calM)$ is an equivalence class of $[(\calA, \calN), (\calB,
    \calM)]$-measurable maps under the relation $\sim$ of equality
  almost everywhere: $f \sim f'$ whenever $\{f \neq f'\} \in \calN$.
  In order to check that this relation is, indeed, transitive, it is
  important to assume that $(X,\calA,\calN)$ is saturated for
  otherwise we would not know that $\{f \neq f'' \} \in \calA$ when
  $f,f'' : X \to Y$ are both
  $[(\calA,\calN),(\calB,\calM)]$-measurable.  Also, in the special
  case where $X$ is $\calN$-negligible and $Y = \emptyset$, we follow
  the convention that there is unique morphism from $(X, \calA,
  \calN)$ to $(\emptyset, \{\emptyset\}, \{\emptyset\})$.
\end{Empty}

\begin{Lemma}
  \label{lemma28}
  Let $(X, \calA, \calN)$, $(Y, \calB, \calM)$ and $(Z, \calC, \calP)$
  be MSNs and let $f, f' \colon X \to Y$ and $g, g' \colon Y
  \to Z$ be maps. If
  \begin{enumerate}
  \item[(A)] $f, f'$ are $[(\calA, \calN), (\calB, \calM)]$-measurable;
  \item[(B)] $g, g'$ are $[(\calB, \calM), (\calC, \calP)]$-measurable;
  \item[(C)] $(X,\calA,\calN)$ is saturated;
  \item[(D)] $f \sim f'$, $g \sim g'$,
  \end{enumerate}
  then $g \circ f$, $g' \circ f'$ are $[(\calA, \calN), (\calC,
    \calP)]$-measurable and $g \circ f \sim g' \circ f'$.
\end{Lemma}

\begin{proof}
  The first conclusion follows from hypotheses (A) and (B) and
  Paragraph~\ref{ANBMMeasurability}.  The second conclusion is a
  consequence of $\{g \circ f \neq g' \circ f'\} \subset \{f \neq f'\}
  \cup f^{-1}(\{g \neq g'\})$ and hypotheses (A), (C) and (D).
\end{proof}

\begin{Empty}[Category $\sfMSN$]
  Thanks to the preceding result, there is a notion of composition for
  morphisms between saturated MSNs: if $\mathbf{f} \colon (X, \calA,
  \calN) \to (Y, \calB, \calM)$ and $\bg \colon (Y, \calB, \calM) \to
  (Z, \calC, \calP)$ are morphisms, we let $\bg \circ \mathbf{f}
  \colon (X, \calA, \calN) \to (Z, \calC, \calP)$ be the equivalence
  class of $g \circ f$ where $f \in \mathbf{f}$ and $g \in
  \mathbf{g}$.

  This allows to define the category $\sfMSN$ whose objects are
  saturated MSNs and whose morphisms are described in the
  paragraph~\ref{MSNMorphism}. Additionally, we add the convention
  that, for a negligible saturated MSN, i.e. an MSN of the form $(X,
  \calP(X), \calP(X))$, there is a unique morphism from $(X, \calP(X),
  \calP(X))$ to $(\emptyset, \{\emptyset\}, \{\emptyset\})$. This way,
  negligible saturated MSNs are isomorphic to one another in the
  category $\sfMSN$.

  The categorical point of view is rarely considered in measure theory,
  mainly due to the lack of a well-behaved notion of morphism between
  measure spaces. The category
  $\sfMSN$ also appears in the work \cite{PAV.20} under the
  name $\mathsf{StrictEMS}$. We start to investigate the existence of
  limits and colimits in this category.
\end{Empty}

\begin{Empty}[subMSN]
  Let $(X, \calA, \calN)$ be an MSN and $Z \in \calP(X)$. We define
  the {\em subMSN} $(Z, \calA_Z, \calN_Z)$, where $\calA_Z \defeq \{A
  \cap Z: A \in \calA\}$ and $\calN_Z \defeq \{N \cap Z : N \in
  \calN\}$. Note that in the special case where $Z$ is
  $\calA$-measurable, we have $\calA_Z = \calA \cap \{A : A \subset
  Z\}$ and likewise $\calN_Z = \calN \cap \{N : N \subset Z\}$.

  The inclusion map $\iota_Z \colon Z \to X$ is $[(\calA_Z, \calN_Z),
    (\calA, \calN)]$-measurable and induces a morphism
  $\boldsymbol{\iota}_Z$ between $(Z, \calA_Z, \calN_Z)$ and $(X,
  \calA, \calN)$.
\end{Empty}

\begin{Proposition}
  \label{equalizerMSN}
  Let $\mathbf{f}, \bg$ be a pair of
  morphisms $(X, \calA, \calN) \to (Y, \calB, \calM)$ in the category $\sfMSN$, represented by
  the maps $f \in \mathbf{f}$ and $g \in \bg$, and set $Z \defeq \{f =
  g\}$. Then the equalizer of $\mathbf{f}, \bg$ is $[(Z, \calA_Z,
  \calN_Z), \boldsymbol{\iota}_Z]$.
\end{Proposition}

\begin{proof}
  As $f \circ \iota_Z = g \circ \iota_Z$, we have clearly $\mathbf{f}
  \circ \boldsymbol{\iota}_Z = \bg \circ \boldsymbol{\iota}_Z$. Let
  $\bh \colon (T, \calC, \calP) \to (X, \calA, \calN)$ be any other
  morphism that satisfies the relation $\mathbf{f} \circ \bh = \bg
  \circ \bh$ and let $h \in \bh$. Then $h^{-1}(Z)$ is conegligible in
  $T$. Up to modifying $h$, we can suppose that it has values in
  $Z$. The restriction $h' \colon T \to Z$ of $h$ is $[(\calC, \calP),
    (\calA_Z, \calN_Z)]$-measurable and we have $h = \iota_Z \circ
  h'$, leading to a factorization $\bh = \boldsymbol{\iota}_Z \circ
  \bh'$. This factorization is unique, as any morphism $\bh'$
  satisfying $\bh = \boldsymbol{\iota}_Z \circ \bh'$ must derive from
  a map $h' \colon T \to Z$ that coincides almost everywhere with $h$.
\end{proof}

\begin{Proposition}
  The category $\sfMSN$ has coproducts. Let $\la (X_i, \calA_i,
  \calN_i)\ra_{i \in I}$ be a family of saturated MSNs. Its coproduct
  is the MSN $(X, \calA, \calN)$ whose underlying set is $X =
  \coprod_{i \in I} X_i$, and whose $\sigma$-algebra and
  $\sigma$-ideal are defined by
  \begin{gather*}
    \calA = \calP(X) \cap \{A : A \cap X_i \in \calA_i \text{ for all } 
i \in I\}, \\
    \calN = \calP(X) \cap \{N : N \cap X_i \in \calN_i \text{ for all } 
i \in I\}.
  \end{gather*}
  For $i \in I$, the canonical morphism $\boldsymbol{\iota}_i \colon
  (X_i, \calA_i, \calN_i) \to (X, \calA, \calN)$ is the morphism
  induced by the inclusion map $\iota_i \colon X_i \to X$.
\end{Proposition}

\begin{proof}
  Notice that, indeed, $(X,\calA,\calN)$ is a saturated MSN.  Let $(Y,
  \calB, \calM)$ be a saturated MSN and $\la \mathbf{f}_i \ra_{i \in
    I}$ be a collection of morphisms from $(X_i, \calA_i, \calN_i)$ to
  $(Y, \calB, \calM)$, each $\mathbf{f}_i$ being represented by a
  measurable map $f_i$. We set $f = \coprod_{i \in I} f_i$, the map
  such that $f \circ \iota_i = f_i$ for any $i \in I$. It is clear
  that $f$ is $[(\calA, \calN), (\calB, \calM)]$-measurable and
  $\mathbf{f} \circ \boldsymbol{\iota}_i = \mathbf{f}_i$ holds for all
  $i \in I$. We need to show that $\mathbf{f}$ is the unique morphism
  $(X, \calA, \calN) \to (Y, \calB, \calM)$ with this property.

  Suppose $\bg \colon (X, \calA, \calN) \to (Y, \calB, \calM)$ is
  another morphism, represented by a measurable map $g \colon X \to
  Y$, for which $\bg \circ \boldsymbol{\iota}_i = \mathbf{f}_i$ for
  all $i \in I$. Then $f$ and $g$ coincide almost everywhere on each
  $X_i$, which implies, due to the choice of $\calN$, that $f$ and $g$
  are equal almost everywhere and $\mathbf{f} = \bg$.
\end{proof}

\begin{Proposition}
  \label{prodMSN}
  The category $\sfMSN$ has countable products. Let $\la (X_i,
  \calA_i, \calN_i) \ra_{i \in I}$ be a countable family of saturated
  MSNs. Its product is the MSN $(X, \calA, \calN)$, whose underlying
  set is the product $X = \prod_{i \in I} X_i$, whose $\sigma$-ideal is
  \[
  \calN = \calP(X) \cap \left\{ N : \exists \la N_i \ra_{i \in I} \in
  \prod_{i \in I} \calN_i, \quad N \subset \bigcup_{i \in I} \pi_i^{-1}(N_i)
  \right\},
  \]
  where $\pi_i \colon X \to X_i$ denotes the projection map, and whose
  $\sigma$-algebra is the saturation of $\bigotimes_{i \in I}
  \calA_i$:
  \[
  \calA = \left\{ A \ominus N : A \in \bigotimes_{i \in I} \calA_i
  \text{ and } N \in \calN\right\}.
  \]
  For $i \in I$, the projection morphism $\boldsymbol{\pi}_i \colon
  (X, \calA, \calN) \to (X_i, \calA_i, \calN_i)$ is the map induced by
  the projection map $\pi_i$.
\end{Proposition}

\begin{proof}
  First note that the projection maps $\pi_i$ are $[(\calA, \calN),
    (\calA_i, \calN_i)]$-measurable by construction of $\calA$ and
  $\calN$. Let $(Y, \calB, \calM)$ be a saturated MSN and $\la
  \mathbf{f}_i \ra_{i \in I}$ be a collection of morphisms from $(Y,
  \calB, \calM)$ to $(X_i, \calA_i, \calN_i)$, each $\mathbf{f}_i$
  being represented by a measurable map $f_i \colon Y \to X_i$. We
  define $f = \prod_{i \in I} f_i \colon Y \to \prod_{i \in I} X_i$
  that assigns $y \in Y$ to $\la f_i(y) \ra_{i \in I}$. Clearly, $f$
  is $(\calB, \bigotimes_{i \in I} \calA_i)$-measurable. Moreover, for
  any negligible set $N \in \calN$, we can find a sequence $\la N_i
  \ra_{i \in I}$ such that $N \subset \bigcup_{i \in I}
  \pi_i^{-1}(N_i)$. Thus $f^{-1}(N) \subset \bigcup_{i \in I} (\pi_i
  \circ f)^{-1}(N_i) = \bigcup_{i \in I} f_i^{-1}(N_i)$. As $I$ is
  countable and $f_i^{-1}(N_i) \in \calM$ for all $i \in I$, we find
  that $f^{-1}(N) \in \calM$, which entails that the map $f$ is
  $[(\calB, \calM), (\calA, \calN)]$-measurable.

  Let $\bg \colon (Y, \calB, \calM) \to (X, \calA, \calN)$ be another
  morphism satisfying the identities $\boldsymbol{\pi}_i \circ \bg =
  \mathbf{f}_i$ for $i \in I$. Let $g \colon Y \to X$ be a
  representative of $\bg$. The coordinate functions $\pi_i \circ g$
  must coincide with $f_i$ almost everywhere. As there are only
  countably many of them, we conclude that $f$ and $g$ are equal
  almost everywhere, that is, $\mathbf{f} = \bg$.
\end{proof}

\begin{Remark}
In case $(X_i,\calA_i,\calN_i)$ are associated with measure spaces $(X_i,\calA_i,\mu_i)$, $i=1,2$, the $\sigma$-ideal $\calN$ considered in the above proposition may not coincide with $\calN_{\mu_1 \otimes \mu_2}$. This is the case, for instance, when $(X_i,\calA_i,\mu_i)=(\R,\calB(\R),\calL^1)$, $i=1,2$, since the diagonal $D = \R^2 \cap \{ (x,x) : x \in 
\R \} \in \calB(\R) \otimes \calB(\R)$ is $\calL^2$-negligible but does not belong to $\calN$.
\end{Remark}

\section{Supremum preserving morphisms}

\begin{Empty}[Motivation]
  One of the reasons we were led to introduce MSNs is that the
  category of measure spaces and (equivalence classes of) measure
  preserving measurable maps does not have good properties at
  all. Roughly speaking, this can be attributed to the fact that it
  has very few arrows. One way to increase their number is to define
  as morphisms $(X, \calA, \mu) \to (Y, \calB, \nu)$ the $(\calA,
  \calB)$-measurable maps $\varphi \colon X \to Y$ such that the
  pushforward measure $\varphi_\# \mu$ is absolutely continuous with
  respect to $\nu$. If we drop the measures and retain only which sets
  have measure zero, we get the notion of $[(\calA, \calN_\mu),
    (\calB, \calN_{\nu})]$-measurability
  of~\ref{ANBMMeasurability}. However, doing so, we may introduce some
  ``irregular'' maps. For example, if $\calA$ is the $\sigma$-algebra
  of Lebesgue measurable sets of the real line, $\calL^1$ the Lebesgue
  measure and $\nu$ the counting measure on $(\R, \calA)$ then the
  identity map induces a morphism $(\R, \calA, \calL^1) \to (\R,
  \calA, \nu)$. But, $\calL^1$ does not really compare to $\nu$,
  although it is absolutely continuous with respect to $\nu$. For
  instance, $\calL^1$ has no Radon-Nikod\'{y}m density with respect to
  $\nu$, not even in the generalized sense of
  Section~\ref{sec:RN}. Forgetting the measures, the morphism of MSNs
  $(\R,\calA, \calN_{\calL^1}) \to (\R, \calA, \{\emptyset\})$ is
  still somehow inappropriate. To avoid this, we restrict our
  attention to the supremum preserving morphisms introduced
  below. This will allow us to define a new category $\sfMSNsp$ of
  saturated MSNs with supremum preserving morphisms. Later, we will be
  able to define localizable versions of MSNs and similar notions by
  means of universal properties to be satisfied in $\sfMSNsp$.
\end{Empty}

\begin{Empty}[Boolean algebras]
\label{boole}
  Many of the properties that we will introduce underneath for MSNs
  are related to their {\em Boolean algebra}, defined in the
  following way: given an MSN $(X, \calA, \calN)$, we observe that the
  $\sigma$-algebra $\calA$ is a Boolean algebra and $\calN$ is an
  ideal of $\calA$ in the ring-theoretic sense; we then associate to
  $(X, \calA, \calN)$ the quotient Boolean algebra $\calA/\calN$.

  When we restrict our attention to saturated MSNs, this construction
  becomes functorial. Call $\mathsf{Bool}(X, \calA, \calN) =
  \calA/\calN$ the Boolean algebra of a saturated MSN. Given a
  morphism $\mathbf{f} \colon (X, \calA, \calN) \to (Y, \calB, \calM)$
  represented by a measurable map $f \colon X \to Y$, we define
  $\mathsf{Bool}(f) \colon \mathsf{Bool}(Y, \calB,\calM) \to
  \mathsf{Bool}(X, \calA, \calN)$ that maps the equivalence class of
  $B \in \calB$ to the equivalence class of $f^{-1}(B)$. This map is
  well-defined because of the $[(\calA,\calN),
    (\calB,\calM)]$-measurability of $f$, it is a morphism of Boolean
  algebras, and it does not depend on the representative of
  $\mathbf{f}$, as one can easily check.
\end{Empty}

\begin{Empty}
  \label{essentialSup}
  Let $(X, \calA, \calN)$ be an MSN and $\calE$ be a subcollection of
  $\calA$. We say that $U \in \calA$ is an {\em $\calN$-essential
    upper bound} of $\calE$ whenever $E \setminus U \in \calN$ for all
  $E \in \calE$. Furthermore, a measurable set $S \in \calA$ is an
  {\em $\calN$-essential supremum} of $\calE$ whenever
  \begin{enumerate}
  \item[(1)] $S$ is an $\calN$-essential upper bound of $\calE$;
  \item[(2)] If $S'$ is an $\calN$-essential upper bound of
    $\calE$, then $S \setminus S' \in \calN$.
  \end{enumerate}
  In particular, if $S$, $S'$ are both $\calN$-essential suprema of
  $\calE$, their symmetric difference $S \ominus S'$ is negligible. In
  other words, an essential supremum, when it exists, is unique up to
  negligible sets. In fact, it corresponds to a (unique) supremum in
  $\mathsf{Bool}(X,\calA,\calN)$. A collection $\calE \subset \calA$
  that admits $X$ as an $\calN$-essential supremum is called {\em
    $\calN$-generating}. We will use the following repeatedly. If
  $\calE \subset \calA$ and $S \in \calA$ is an $\calN$-essential
  supremum of $\calE$, then $\calE \cup \{ X \setminus S \}$ is
  $\calN$-generating. 

  The next ubiquitous lemma expresses that $\cap$ is distributive over
  the (partially defined) operation of taking essential suprema. It
  implies the following fact, which we will use frequently: If $\calE$
  is $\calN$-generating and $A \in \calA \setminus \calN$, then $E
  \cap A \not \in \calN$ for some $E \in \calE$.
\end{Empty}

\begin{Lemma}[Distributivity Lemma]
  \label{supDistrib}
  Let $(X, \calA, \calN)$ be an MSN, $\calE \subset \calA$ be a
  collection that has an $\calN$-essential supremum $S$, and $C \in
  \calA$. Then $C \cap S$ is an $\calN$-essential supremum of $\{C
  \cap E : E \in \calE\}$.
\end{Lemma}

\begin{proof}
  Condition (1) in Definition~\ref{essentialSup} is met because $C
  \cap E \setminus C \cap S = C \cap (E \setminus S) \in \calN$ for
  all $E \in \calE$. As for (2), we let $S'$ be an $\calN$-essential
  upper bound for $\{C \cap E: E \in \calE\}$. We claim that $S''
  \defeq S' \cup (X \setminus C)$ is an $\calN$-essential upper bound
  for $\calE$. Indeed, for any $E \in \calE$, we have $E \setminus S''
  = (C \cap E) \setminus S' \in \calN$. It follows that $(C \cap S)
  \setminus S' = S \setminus S'' \in \calN$.
\end{proof}

\begin{Empty}
  Note that if $(X,\calA,\calN)$ and $(Y,\calB,\calM)$ are MSNs, $f :
  X \to Y$ is $[(\calA,\calN),(\calB,\calM)]$-measurable, $\calE
  \subset \calB$, and $S \in \calB$ is an $\calM$-essential upper
  bound of $\calE$, then $f^{-1}(S)$ is an $\calN$-essential upper
  bound of $f^{-1}(\calE)$. However, if $S$ is an $\calM$-essential
  supremum of $\calE$ then $f^{-1}(S)$ may not be an $\calN$-essential
  supremum of $f^{-1}(\calE)$. Consider, for instance,
  $(X,\calA,\calN)=(\R,\calB(\R),\calN_{\calL^1})$,
  $(Y,\calB,\calM)=(\R,\calB(\R),\{\emptyset\})$, $f=\rmid_{\R}$, and
  $\calE = \{ \{x\} : x \in \R \}$. Then $\R$ is an
  $\{\emptyset\}$-essential supremum of $\calE$, $\emptyset$ is an
  $\calN_{\calL^1}$-essential supremum of $\calE = f^{-1}(\calE)$, and
  $\R \setminus \emptyset \not\in \calN_{\calL^1}$.
\end{Empty}

\begin{Empty}
  There are several objects that we can call {\em supremum
    preserving}. For saturated MSNs $(X, \calA, \calN)$ and $(Y,
  \calB, \calM)$, we define
  \begin{itemize}
  \item A morphism of Boolean algebras $\varphi \colon \frA \to \frB$
    is called supremum preserving if, for any family $\frE \subset \frA$ that admits a
    supremum, the family $\varphi(\frE)$ admits a supremum and $\varphi(\sup \frE) = \sup
    \varphi(\frE)$.
  \item A morphism $\mathbf{f} \colon (X, \calA, \calN) \to (Y, \calB,
    \calM)$ is called supremum preserving whenever
    $\mathsf{Bool}(\mathbf{f})$ is.
  \item An $[(\calA, \calN), (\calB, \calM)]$-measurable map $f \colon
    X \to Y$ is called supremum preserving if, for any collection $\calE
    \subset \calB$ possessing an $\calM$-essential supremum $S$, then
    $f^{-1}(S)$ is an $\calN$-essential supremum of $f^{-1}(\calE)
    \defeq \{f^{-1}(E) : E \in \calE\}$.
  \end{itemize}
  For a morphism $\mathbf{f}$ represented by $f \in \mathbf{f}$, the
  supremum preserving characters of $f$, $\mathbf{f}$ and
  $\mathsf{Bool}(\mathbf{f})$ are all equivalent. Also, the
  composition of two supremum preserving morphisms is supremum
  preserving. We call $\sfMSNsp$ the subcategory of $\sfMSN$ that
  consists of saturated MSNs and supremum preserving morphisms. In the
  next proposition, we gather some basic facts about supremum
  preserving morphisms and the category $\sfMSNsp$.
\end{Empty}

\begin{Proposition}
  \label{propsp}
  The following hold. 
  \begin{enumerate}
  \item[(A)] Two saturated MSNs are isomorphic in $\sfMSN$ if and only
    if they are isomorphic in $\sfMSNsp$.
  \item[(B)] Let $(X, \calA, \calN)$ be a saturated MSN and $Z \in
    \calA$. The morphism $\boldsymbol{\iota}_Z \colon (Z, \calA_Z,
    \calN_Z) \to (X, \calA, \calN)$ induced by the inclusion map
    $\iota_Z \colon Z \to X$ is supremum preserving.
  \item[(C)] Let $\mathbf{f}, \bg \colon (X, \calA, \calN) \to (Y,
    \calB, \calM)$ be a pair of morphisms in $\sfMSNsp$, represented
    by $f \in \mathbf{f}$ and $g \in \mathbf{g}$. If $Z \defeq \{f =
    g\}$ is $\calA$-measurable, then $((Z, \calA_Z, \calN_Z),
    \boldsymbol{\iota}_Z)$ is the equalizer of $\mathbf{f}, \bg$ in $\sfMSNsp$.
  \item[(D)] The category $\sfMSNsp$ has coproducts, which are
    preserved by the forgetful functor $\sfMSNsp \to \sfMSN$.
  \end{enumerate}
\end{Proposition}

\begin{proof}
  (A) Let $\mathbf{f}$ be an isomorphism in $\sfMSN$. Then
  $\mathsf{Bool}(\mathbf{f})$ is an isomorphism of Boolean
  algebras. More specifically, it is an isomorphism of posets and for
  this reason it preserves suprema.

  (B) This is the content of Lemma~\ref{supDistrib}.

  (C) By Proposition~\ref{equalizerMSN}, $((Z, \calA_Z, \calN_Z),
  \boldsymbol{\iota}_Z)$ is the equalizer of $\mathbf{f}, \bg$ in
  $\sfMSN$ and by (B) the morphism $\boldsymbol{\iota}_Z$ is a morphism of $\sfMSNsp$. Let $\bh \colon (T, \calC, \calP) \to (X, \calA, \calN)$
  be a supremum preserving morphism that satisfies $\mathbf{f} \circ \bh = \bg
  \circ \bh$. Recalling the proof of Proposition~\ref{equalizerMSN},
  there is a representative $h \in \bh$ with values in $Z$, and its
  restriction $h' \colon T \to Z$ induces the unique morphism $\bh'$
  such that $\bh = \boldsymbol{\iota}_Z \circ \bh'$. The results
  follows from the fact that $h'$ is easily checked to be supremum
  preserving.
  
  (D) Let $\la (X_i, \calA_i, \calN_i)\ra_{i \in I}$ be a family of
  saturated MSNs, $(X, \calA, \calN)$ be their coproduct in the
  category $\sfMSN$, and $\la \mathbf{f}_i \ra_{i \in I}$ be a family
  of supremum preserving morphisms from $(X_i, \calA_i, \calN_i)$ to a
  saturated $(Y, \calB, \calM)$, each represented by $f_i \colon X_i
  \to Y$. We need to show that $f \defeq \coprod_{i \in I} f_i \colon
  X \to Y$ is supremum preserving. For this, let $\calE \subset \calB$
  be a collection that has an $\calM$-essential supremum $S$. We
  observe that $f^{-1}(S) = \coprod_{i \in I} f^{-1}_i(S)$ is an
  $\calM$-essential upper bound of $f^{-1}(\calE)$. Let $U$ be a
  second $\calN$-essential upper bound of $f^{-1}(\calE)$. Then $X_i
  \cap U$ is an $\calN_i$-essential upper bound of $\{X_i \cap
  f^{-1}(E) : E \in \calE\} = f_i^{-1}(\calE)$. It follows that
  $f_i^{-1}(S) \setminus (X_i \cap U) \in \calN_i$. As this happens
  for all $i \in I$, we get that $f^{-1}(S) \setminus U \in \calN$.
\end{proof}

\section{Localizable, cccc and strictly localizable MSNs}

\begin{Empty}[Localizable MSN]
  An MSN $(X, \calA, \calN)$ is {\em localizable} whenever each
  collection $\calE \subset \calA$ admits an $\calN$-essential
  supremum. Equivalently, $(X, \calA, \calN)$ is localizable whenever
  its Boolean algebra $\calA / \calN$ is {\em Dedekind complete}, that
  is, each subset of $\calA/ \calN$ has a supremum.
  
  Originally, localizability was introduced by Segal \cite{SEG.51} in
  the context of measure spaces. Since then, many minor variations
  over the definition in that context have been proposed (see \cite{OKA.RIC.19} for
  an overview). We will follow the definition
  in~\cite[Chapter~2]{FREMLIN.II}. A measure space $(X, \calA, \mu)$ is
  called {\em localizable} whenever
  \begin{itemize}
  \item[(1)] it is {\em semi-finite}, i.e. for all $A \in
    \calA$ with $\mu(A) > 0$, there is a measurable set $A' \subset A$
    such that $0 < \mu(A') < \infty$;
  \item[(2)] the underlying MSN $(X, \calA, \calN_\mu)$ is
    localizable.
  \end{itemize}
  
\begin{Empty}[Semi-finite measure space] 
\label{semi-finite} 
  In the definition of localizable measure space, semi-finiteness
  plays on important r\^ole. Let us rephrase it. Given $(X,\calA,\mu)$
  a measure space, we abbreviate $\calA^f \defeq \calA \cap \{E :
  \mu(E) < \infty\}$. We say that $N \in \calA$ is {\em locally
    $\mu$-negligible} whenever $N \cap E \in \calN_\mu$ for every $E
  \in \calA^f$. We let $\calN_{\mu,\rmloc}$ be the $\sigma$-ideal
  consisting of locally $\mu$-negligible measurable sets. The
  following are equivalent:
  \begin{enumerate}
  \item[(1)] $\calN_\mu = \calN_{\mu,\rmloc}$.
  \item[(2)] $(X,\calA,\mu)$ is semi-finite.
  \item[(3)] $\calA^f$ is $\calN_\mu$-generating.
  \end{enumerate}
  
  The only non trivial part is $(3) \Rightarrow (1)$. If $N \in \calA$ then $N$ is an $\calN_\mu$-essential supremum of $\{ N \cap F : F \in \calA^f \}$, according to the Distributivity Lemma \ref{supDistrib}. If also $N \in \calN_{\mu,\rmloc}$, then it follows that $N \in \calN_\mu$. The notion of locally $\mu$-negligible sets will appear again in \ref{def.ld}.
\end{Empty}  
  
  Next we introduce some classes of localizable MSNs that will appear
  throughout the paper.
\end{Empty}

\begin{Empty}[Countable chain condition]
\label{ccc.def}
  Let $(X, \calA, \calN)$ be an MSN. A family $\calE \subset \calA
  \setminus \calN$ is called {\em almost disjointed} whenever $E \cap
  E' \in \calN$ for any pair of distinct $E, E' \in \calE$. The MSN $(X,
  \calA, \calN)$ is said to have the {\em countable chain condition}
  (in short: is {\em ccc}) whenever an almost disjointed family in
  $\calA \setminus \calN$ is at most countable.

  The previous notions have counterparts in the realm of Boolean
  algebras. Given a Boolean algebra $\frA$, a subset $\frE \subset
  \frA$ is called {\em disjointed} whenever $x \wedge y = 0$ for any
  pair of distinct elements $x, y \in \frE$. The Boolean algebra
  $\frA$ has the {\em countable chain condition} (or: is {\em ccc})
  whenever each of its disjointed families is at most countable. Of
  course, an MSN $(X, \calA, \calN)$ is ccc if and only if its Boolean
  algebra $\calA/\calN$ is. In the following proposition, we show that
  being ccc is stronger than localizability. It is related to the
  fact, first established in~\cite{TAR.37}, that a Dedekind
  $\sigma$-complete Boolean algebra (that is, a Boolean algebra where
  countable collections have suprema) having the countable chain
  condition is Dedekind complete.
\end{Empty}

\begin{Proposition}
  \label{ccc=>loc}
  If an MSN $(X, \calA, \calN)$ is ccc and $\calE \subset \calA$ is a
  collection, then there is a countable subcollection $\calE' \subset
  \calE$ such that $\bigcup \calE'$ is an $\calN$-essential supremum
  of $\calE$. In particular, $(X, \calA, \calN)$ is localizable.
\end{Proposition}

\begin{proof}
  Suppose the existence of a collection $\calE \subset \calA$ for
  which one cannot find a countable subcollection $\calE' \subset
  \calE$ whose union is an $\calN$-essential supremum of $\calE$. This
  assumption allows us to construct transfinitely a sequence $\la
  E_\alpha \ra_{\alpha < \omega_1}$ with values in $\calE$ such that
  for every $\alpha < \omega_1$, one has $F_\alpha \defeq E_\alpha
  \setminus \bigcup_{\beta < \alpha} E_\beta \not\in \calN$. But the
  disjointed family $\{F_\alpha : \alpha < \omega_1\}$ contradicts the
  fact that $(X, \calA, \calN)$ is ccc.
\end{proof}

\begin{Proposition}
  \label{finite=>loc}
  Let $(X, \calA, \mu)$ be a finite measure space. Then $(X, \calA,
  \calN_\mu)$ is ccc.
\end{Proposition}

\begin{proof}
  Let $\calE \subset \calA \setminus \calN$ be an almost disjointed
  family. For each positive integer $n$, set $\calE_n = \calE \cap \{E
  : \mu(E) > n^{-1}\}$. As $\mu(X) \geq \mu\left(\bigcup
  \calE_n\right) \geq n^{-1}\rmcard \calE_n$, we have that $\calE_n$
  is finite. Consequently, $\calE$ is at most countable.
\end{proof}

\begin{Proposition}
  \label{coprodLoc}
  A coproduct $\coprod_{i \in I} (X_i, \calA_i, \calN_i)$ of saturated
  localizable MSNs is localizable.
\end{Proposition}

\begin{proof}
  Let $\calE$ be a collection of measurable sets of a coproduct
  $\coprod_{i \in I} (X_i ,\calA_i, \calN_i)$. For each $i \in I$, the
  collection $\calE_i \defeq \{X_i \cap E : E \in \calE\}$ has an
  $\calN_i$-essential supremum $S_i \subset X_i$. We then routinely
  check that $S \defeq \coprod_{i \in I} S_i \in \calA$ is an
  $\calN$-essential supremum of $\calE$.
\end{proof}

\begin{Empty}[Stronger notions of localizability]
  An MSN is called {\em strictly localizable} if it is isomorphic to a
  coproduct of the form $\coprod_{i \in I} (X_i, \calA_i,
  \calN_{\mu_i})$, where $(X_i, \calA_i, \mu_i)$ are complete finite
  measure spaces.  Examples of strictly localizable MSNs are provided
  by MSNs associated to complete $\sigma$-finite measure spaces $(X,
  \calA, \mu)$. Indeed, denoting $\la X_i \ra_{i \in I}$ a countable
  partition of $X$ into measurable subsets of finite $\mu$ measure,
  one can verify that $(X, \calA, \calN_\mu)$ is isomorphic to
  $\coprod_{i \in I} (X_i, \calA_{X_i}, \calN_{\mu \shel X_i})$.

  Likewise, we say that an MSN is {\em cccc} whenever it is isomorphic
  to a coproduct of saturated ccc MSNs. We have the chain of
  implications
  \[
  \text{strictly localizable} \implies \text{cccc} \implies \text{localizable}.
  \]
  The first implication comes from Proposition~\ref{finite=>loc}, the
  second one from Propositions~\ref{ccc=>loc}
  and~\ref{coprodLoc}. Examples of non localizable spaces are provided
  by the next results.
\end{Empty}

\begin{Lemma}
  \label{47}
  Let $(X, \calA, \calN)$ be a localizable MSN, $\calE \subset \calA
  \setminus \calN$ an almost disjointed family. Then $\rmcard (\calA /
  \calN) \geq 2^{\rmcard \calE}$.
\end{Lemma}

\begin{proof}
  Consider the application $\calP(\calE) \to \calA / \calN$ which maps
  each subcollection $\calE' \subset \calE$ to the equivalence class
  of its $\calN$ essential supremum. We claim that this map is
  injective. Indeed, suppose $\calE', \calE'' \subset \calE$ are
  distinct. Call $S'$ (resp. $S''$) an $\calN$ essential supremum of
  $\calE'$ (resp. $\calE''$). Without loss of generality, there is $F
  \in \calE' \setminus \calE''$. By Lemma~\ref{supDistrib}, $F \cap
  S'$ (resp. $F \cap S''$) is an essential supremum of $\{F\cap E: E
  \in \calE'\}$ (resp. $\{F \cap E : E \in \calE''\}$). We deduce that $F 
\setminus S' \in \calN$ and, taking the
  almost disjointed character of $\calE$ into account, that $F \cap S'' \in \calN$. This
  implies that $S'$ and $S''$ do not induce the same equivalence class
  in $\calA/\calN$.
\end{proof}

We will use the following many times.

\begin{Lemma}
\label{zorn}
Let $(X,\calA,\calN)$ be an MSN and let $\calC \subset \calA$ be $\calN$-generating. There exists $\calE \subset \calA \setminus \calN$ with the following properties.
\begin{enumerate}
\item[(A)] $\calE$ is almost disjointed.
\item[(B)] For each $E \in \calE$, there exists $C \in \calC$ such that $E \subset C$.
\item[(C)] $\calE$ is $\calN$-generating. 
\end{enumerate}
\end{Lemma}

\begin{proof}
There is no restriction to assume that $\calN \neq \calA$; in
particular, $\calC \neq \emptyset$. Consider the set $\bE$ consisting
of those $\calE \subset \calA \setminus \calN$ that satisfy conditions
(A) and (B) above, ordered by inclusion. Thus, $\bE$ is nonempty and
one readily checks that every chain in $\bE$ possesses a maximal
element. Therefore, $\bE$ admits a maximal element $\calE$, according
to Zorn's Lemma. We ought to show that $\calE$ is
$\calN$-generating. If this were not the case, there would exist an
$\calN$-essential upper bound $U \in \calA$ of $\calE$ such that $X
\setminus U \not \in \calN$. The latter, together with the fact that
$\calC$ is $\calN$-generating, implies the existence of $C \in \calC$
such that $C \cap (X \setminus U) \not \in \calN$. Then, $\calE \cup
\{ C \cap (X \setminus U) \}$ contradicts the maximality of $\calE$.
\end{proof}

\begin{Proposition}[$\mathsf{ZFC} + \mathsf{CH}$]
  Let $X$ be a Polish space endowed with its Borel $\sigma$-algebra
  $\calB(X)$ and $\mu \colon \calB(X) \to [0, \infty]$ be a
  semi-finite Borel measure. Under the Continuum Hypothesis, one has
  the following dichotomy: either $\mu$ is $\sigma$-finite, or the
  MSN $(X, \calB(X), \calN_{\mu})$ is not localizable.
\end{Proposition}

\begin{proof}
  Let $\calE$ be associated with $\calC \defeq \calB(X) \cap \{ A :
  \mu(A) < \infty \}$ in Lemma \ref{zorn}. Recall \ref{semi-finite}
  that $\calC$ is $\calN_\mu$-generating.  If $\calE$ is countable,
  then $\bigcup \calE$ is measurable and, accordingly, an
  $\calN_\mu$-essential upper bound of $\calE$. Thus $X \setminus \bigcup
  \calE \in \calN_\mu$, since $\calE$ is $\calN_\mu$-generating. We
  have proven that $\mu$ is $\sigma$-finite.

  On the other hand, if $\calE$ is uncountable, the Continuum
  Hypothesis guarantees that it has cardinal greater or equal to
  $\mathfrak{c}$. Assume if possible that $(X, \calB(X), \mu)$ is
  localizable. As the map $\calB(X) \to \calB(X) / \calN_{\mu}$ is
  onto, we deduce from Lemma~\ref{47} that $\rmcard \calB(X) \geq
  2^{\mathfrak{c}} > \mathfrak{c}$. However, Borel sets are Suslin,
  and Suslin sets are continuous images of closed subsets of a
  particular Polish space, the Baire space, see e.g
  \cite[3.3.18]{SRIVASTAVA}. This gives the upper bound $\rmcard \calB(X)
  \leq \mathfrak{c}$, contradicting the preceding inequality.
\end{proof}

\begin{Empty}[$\calP$-version of an MSN]
  \label{pversion}
  Let $\calP$ be a property associated to MSNs. We suppose that the
  property $\calP$ is {\em hereditary}: if $(X, \calA, \calN)$ has
  $\calP$ then the MSNs $(Z, \calA_Z, \calN_Z)$ also has $\calP$ for
  all $Z \in \calA$. ``Being strictly localizable'', ``being cccc'' or
  ``being localizable'' are examples of hereditary properties.

  Let $(X, \calA, \calN)$ be a saturated MSN. We define a {\em
    $\calP$-version} of $(X, \calA, \calN)$ to be a couple $[(\hat{X},
    \hat{\calA}, \hat{\calN}), \bp]$ consisting of a saturated MSN
  $(\hat{X}, \hat{\calA}, \hat{\calN})$ with the property $\calP$ and
  a supremum preserving morphism $\bp \colon (\hat{X}, \hat{\calA},
  \hat{\calN}) \to (X, \calA, \calN)$ satisfying the following
  property: For any saturated MSN $(Y, \calB, \calM)$ with the
  property $\calP$ and any supremum preserving morphism $\bq \colon
  (Y,\calB, \calM) \to (X, \calA, \calN)$, there is a unique supremum
  preserving morphism $\br\colon (Y, \calB, \calM) \to (\hat{X},
  \hat\calA, \hat{\calN})$ such that $\bq =\bp \circ \br$.
  \[
  \begin{tikzcd}
    (Y, \calB, \calM) \arrow[rd, swap, "\bq"] \arrow[rr, dotted,
      "\exists!\br"] & & (\hat{X}, \hat{\calA}, \hat{\calN})
    \arrow[ld, "\bp"] \\ & (X, \calA, \calN) &
  \end{tikzcd}
  \]
  By this definition, a $\calP$-version must satisfy a universal
  property, and as such it is unique up to a unique isomorphism of the
  category $\sfMSNsp$. More specifically, if $[(\hat{X}, \hat{\calA},
  \hat{\calN}), \bp]$ and $[(\hat{X}', \hat{\calA}', \hat{\calN}'),
  \bp']$ are two $\calP$-versions, then we easily check that there is
  a unique isomorphism $\br \colon (\hat{X}, \hat{\calA}, \hat{\calN})
  \to (\hat{X}', \hat{\calA}', \hat{\calN}')$ such that $\bp' \circ
  \br = \bp$.
\end{Empty}

\begin{Empty}[Atomic MSNs]
  \label{atomic}
  One of our motivations in this article is to find a universal
  construction that transforms an MSN into something with better
  localizability properties. As such, it is wise to first have a look
  at the not so easy case of MSNs $(X, \calA, \calN)$ such that all
  singletons are $\calA$-measurable and $ \calN = \{\emptyset\}$. We
  call such MSNs {\em atomic}.

  In an atomic MSN $(X, \calA, \{\emptyset\})$, it is easy to see that
  a subset $\calE \subset \calA$ has an $\{\emptyset\}$-essential
  supremum if and only if $\bigcup \calE \in \calA$, in which case
  $\bigcup \calE$ is the $\{\emptyset\}$-essential supremum. Therefore
  the MSN $(X, \calA, \{\emptyset\})$ is localizable if and only if
  $\calA = \calP(X)$. In other words, the non localizability of $(X,
  \calA, \{\emptyset\})$ can only be due to the lack of measurable
  sets; therefore it seems sensible to ask for $(X, \calP(X),
  \{\emptyset\})$ to be the ``localization'' of $(X, \calA,
  \{\emptyset\})$.

  Unfortunately, Proposition~\ref{counterexamplelv} gives a negative
  result. It tells us that the localizable version of an MSN, as
  defined in~\ref{pversion}, is not the right notion of
  ``localization''. This issue will be addressed in
  Section~\ref{sec:ld} by introducing a notion of local determination
  for MSNs.
\end{Empty}

\begin{Proposition}
  \label{counterexamplelv}
  Let $X$ be an uncountable set, $\calC(X)$ be its
  countable-cocountable $\sigma$-algebra. Let $\boldsymbol{\iota}$ be
  the morphism $(X, \calP(X), \{\emptyset\}) \to (X, \calC(X),
  \{\emptyset\})$ induced by the identity map. Then $[(X, \calP(X),
  \{\emptyset\}), \boldsymbol{\iota}]$ is not a localizable version of
  $(X, \calC(X), \{\emptyset\})$.
\end{Proposition}

\begin{proof}
  That $\boldsymbol{\iota}$ is supremum preserving follows from the
  discussion in Paragraph~\ref{atomic}. Assume if possible that $((X,
  \calP(X), \{\emptyset\}), \boldsymbol{\iota})$ is a localizable
  version of $(X, \calA, \{\emptyset\})$. We will get a contradiction
  if we manage to build a localizable saturated MSN $(Y, \calB,
  \calM)$ and a function $q \colon Y \to X$ that is $[(\calB, \calM),
    (\calC(X), \{\emptyset\})]$-measurable, supremum preserving, but
  not $(\calB, \calP(X))$-measurable.

  We choose $Y = X^2 \times \{0, 1\}$. For any subset $B \subset Y$,
  we call $B[0]$ and $B[1]$ the subsets defined by
  \[
  B[i] \defeq X^2 \cap \{(x_1, x_2) : (x_1, x_2, i) \in B\} \quad
  \text{for } i \in \{0, 1\}
  \]
  We let $\calB = \calP(Y) \cap \{B : B[0] \ominus B[1] \text{ is
    countable}\}$. We claim that $\calB$ is a $\sigma$-algebra of
  $Y$. The stability of $\calB$ under countable unions is a
  consequence of the formula
  \[
  \left(\bigcup_{n \in \N} B_n\right)[0] \ominus \left(\bigcup_{n \in
  \N} B_n\right)[1] \subset \bigcup_{n \in \N} B_n[0] \ominus B_n[1]
  \]
  that holds for any sequence $\la B_n \ra_{n \in \N}$ of subsets in $Y$, 
and
  we leave the other points to the reader. Finally, we define the
  $\sigma$-ideal $\calM \defeq \calB \cap \{M : M[0] =
  \emptyset\}$. Clearly $(Y, \calB, \calM)$ is a saturated MSN.

  Let us show that $(Y, \calB, \calM)$ is localizable. Let $\calE
  \subset \calB$ be any collection. We set $A \defeq \bigcup_{E \in
    \calE} E[0]$ and $S \defeq A \times \{0, 1\}$. The set $S$ is
  $\calB$-measurable, because $S[0] = S[1] = A$. For any $E \in
  \calE$, we have $(E \setminus S)[0] = E[0] \setminus S[0] =
  \emptyset$, meaning that $S$ is an $\calM$-essential upper bound of
  $\calE$. Denoting by $U$ another essential upper bound of $\calE$,
  then $(E \setminus U)[0] = E[0] \setminus U[0] = \emptyset$ for all
  $E \in \calE$. It follows that $A \subset U[0]$ and $(S \setminus
  U)[0] = S[0] \setminus U[0] = \emptyset$. Thus, $S$ is an
  $\calM$-essential supremum, as we wanted.

  Now, let $\sigma \colon X \to X$ be a bijection of $X$ without
  fixed points.  For example, choose a partition $X = Z \cup Z'$ into
  subsets $Z, Z'$ that have the same cardinality as $X$, choose a
  bijection $f \colon Z \to Z'$ and set $\sigma$ so that $\sigma(x) =
  f(x)$ for all $x \in Z$ and $\sigma(x) = f^{-1}(x)$ for all $x \in
  Z'$. We define the map $q \colon Y \to X$ by
  \[
  \forall (x_1, x_2, i) \in Y, \quad q(x_1, x_2, i)
  = \begin{cases} x_2 & \text{if } i = 1 \text{ and } x_2 =
    \sigma(x_1) \\ x_1 & \text{otherwise}
  \end{cases}
  \]

  First we show that $q$ is $[(\calB, \calM), (\calC(X),
    \{\emptyset\})]$-measurable. It suffices to show that
  $q^{-1}(\{x\}) \in \calB$ for all $x \in X$. But we have
  \begin{align*}
    q^{-1}(\{x\})[0] & = \{x\} \times X \\
    q^{-1}(\{x\})[1] & =
  \big(\{x\} \times \left(X \setminus \{\sigma(x)\}\right)\big) \cup
  \{(\sigma^{-1}(x)\,, x)\}
  \end{align*}
  Consequently, $q^{-1}(\{x\})[0] \ominus q^{-1}(\{x\})[1]$ has only
  two elements. By the definition of $\calB$, this ensures the
  measurability of $q^{-1}(\{x\})$.

  However, we claim that $q$ is not $(\calB, \calP(X))$-measurable. To
  this end, we will show that $q^{-1}(Z) \not\in \calB$. We have
  \begin{align*}
    q^{-1}(Z)[0] & = Z \times X \\ q^{-1}(Z)[1] & = \{(x_1, x_2) : x_1
    \in Z, x_2 \neq \sigma(x_1)\} \cup \{(\sigma^{-1}(x), x) : x \in
    Z\}
  \end{align*}
  It follows that $q^{-1}(Z)[0] \ominus q^{-1}(Z)[1] = \{(x, \sigma(x)) 
:
  x \in Z\} \cup \{(\sigma^{-1}(x), x) : x \in Z\}$ is
  uncountable. Thus, $q^{-1}(Z) \not\in \calB$.

  It only remains to prove that $q$ is supremum preserving. Let $\calE
  \subset \calC(X)$ be a collection that has an
  $\{\emptyset\}$-essential supremum $S$. This implies that $S =
  \bigcup \calE$. We wish to prove that $q^{-1}(S)$ is an
  $\calM$-essential supremum of $q^{-1}(\calE) = \{q^{-1}\{x\} : x
  \in S\}$. First suppose that $\calE$ consists
  only of singletons. Of course, $q^{-1}(S)$ is an $\calM$-essential
  upper bound of $q^{-1}(\calE)$. Let $U$ an arbitrary $\calM$-essential upper bound of
  $q^{-1}(\calE)$.  For all $x \in S$, we have
  $q^{-1}\{x\} \setminus U \in \calM$, meaning that $\{x\} \times X =
  (q^{-1}\{x\})[0] \subset U[0]$. Thus $S \times X \subset U[0]$,
  which implies $(q^{-1}(S) \setminus U)[0] = q^{-1}(S)[0] \setminus
  U[0] = S\times X \setminus U[0] = \emptyset$. It means that
  $q^{-1}(S) \setminus U \in \calM$. Thus, we have shown that
  $q^{-1}(S)$ is an $\calM$-essential supremum of $\calE$.

  Now we turn to the general case, where $\calE$ need not consist only
  of singletons. Let $\calE' = \{ \{x\} : x \in E \in
  \calE\}$. Clearly, $\calE$ and $\calE'$ have the same
  $\{\emptyset\}$-essential supremum $S \defeq \bigcup \calE = \bigcup
  \calE'$. By what precedes, $q^{-1}(S)$ is an $\calM$-essential
  supremum of $q^{-1}(\calE')$ and it is an $\calM$-essential upper
  bound of $q^{-1}(\calE)$. An $\calM$-essential upper bound $U$ of
  $q^{-1}(\calE)$ is also an upper bound for $q^{-1}(\calE')$, as any
  member of $q^{-1}(\calE')$ is a subset of a member of
  $q^{-1}(\calE)$. Therefore, $q^{-1}(S) \setminus U \in \calM$,
  showing that $q^{-1}(S)$ is an $\calM$-essential supremum of
  $q^{-1}(\calE)$.
\end{proof}

\begin{Empty}[Example of an MSN with no localizable part]
\label{4.14}
   Consider an MSN of the form $(X, \calP(X), \calK(X))$, where $X$ is
   a set of cardinality $\aleph_1$ and $\calK(X)$ is the
   $\sigma$-ideal of countable subsets. There is a bijection $\varphi
   \colon X \to X \times X$ and we can use it to construct an
   uncountable family of ``horizontal lines'' $H_x \defeq
   \varphi^{-1}(X \times \{x\})$ indexed by $x \in X$ witnessing that
   $(X, \calP(X), \calK(X))$ is not ccc. Actually, we can do better
   and prove that it is not localizable. Suppose $\{H_x : x \in X\}$
   has a $\calK(X)$-essential supremum $S$. For each $x \in X$ choose
   a point $p_x \in S \cap H_x$. Then it is easy to see that $U \defeq
   S \setminus \{p_x : x \in X\}$ is an essential upper bound for the
   family of horizontal lines, however $S \setminus U = \{p_x : x \in
   X\}$ is not negligible, contradicting that $S$ is an essential
   supremum.

   Observe that the MSN $(X, \calP(X), \calK(X))$ is isomorphic to all
   its non negligible subMSNs. In particular, it has no nontrivial ccc
   or localizable part, an unpleasant situation that we will rule out
   in the next paragraph by introducing the notions of ``locally
   localizable'' and ``locally ccc'' MSN.

   We will prove nonetheless that $(X, \calP(X), \calK(X))$ has a cccc
   version, that is disappointingly the trivial MSN $(\emptyset,
   \{\emptyset\}, \{\emptyset\})$ (with the only morphism from there
   to $(X, \calP(X), \calK(X))$). To establish this fact, one needs to
   prove that if $(Y, \calB, \calM)$ is a cccc MSN and $\mathbf{f}
   \colon (Y, \calB, \calM) \to (X, \calP(X), \calK(X))$ is a supremum
   preserving morphism, then $Y \in \calM$ (actually, the supremum
   preserving character of $\mathbf{f}$ will not be used). We can
   reduce to the case where $(Y, \calB, \calM)$ is ccc.

   We reproduce an argument due to Ulam \cite{ULA.30}, showing that
   there is a family $\la A_{n,\alpha} \ra_{n \in \N, \alpha <
     \omega_1}$ of subsets of $X$ such that
   \begin{itemize}
   \item for all $n \in \N$, the family $\la A_{n, \alpha} \ra_{\alpha
     < \omega_1}$ is disjointed;
   \item for all $\alpha < \omega_1$, the union $\bigcup_{n \in \N}
     A_{n, \alpha}$ is conegligible (that is, cocountable).
   \end{itemize}
   Any ordinal $\beta < \omega_1$ is countable, so we can select a
   sequence $\la k_{\alpha, \beta} \ra_{\alpha < \beta}$ of distinct
   integers. Let $\la x_\beta \ra_{\beta < \omega_1}$ be an enumeration
   of all the elements in $X$. Set $A_{n, \alpha} \defeq \{x_\beta :
   \beta > \alpha \text{ and } k_{\alpha, \beta} = n\}$ for every $n
   \in \N$ and $\alpha < \omega_1$. For distinct $\alpha, \alpha' <
   \omega_1$, there cannot be some $x_\beta \in A_{n, \alpha} \cap
   A_{n, \alpha'}$, for otherwise we would have $k_{\alpha, \beta} =
   k_{\alpha', \beta}$. In addition, one has $\bigcup_{n \in \N} A_{n,
     \alpha} = \{x_\beta : \beta > \alpha\}$ whose complement in $X$ is
   the countable set $\{x_\beta : \beta \leq \alpha\}$.

   Now, fix a representative $f \in \mathbf{f}$. The family $\la
   f^{-1}(A_{n, \alpha}) \ra_{\alpha < \omega_1}$ being disjointed,
   the set $C_n \defeq \omega_1 \cap \{\alpha : f^{-1}(A_{n, \alpha})
   \not\in \calM\}$ is countable for all $n \in \N$. Hence the
   existence of some $\alpha < \omega_1 \setminus \bigcup_{n \in \N}
   C_n$. Now we see that $f^{-1}\left( \bigcup_{n \in \N} A_{n, \alpha}
   \right)$ is both negligible and conegligible in $Y$, which can
   happen only if $Y \in \calM$.
\end{Empty}

\begin{Empty}
  \label{locally}
  Let $\calP$ be a hereditary property associated to MSNs. We say that
  an MSN $(X, \calA, \calN)$ is {\em locally $\calP$} whenever one of
  the following equivalent statements holds:
  \begin{enumerate}
  \item[(A)] The collection $\calA \cap \{Z : (Z, \calA_Z, \calN_Z)
    \text{ has the property } \calP\}$ is $\calN$-generating;
  \item[(B)] for any $Y \in \calA \setminus \calN$ there is $Z \in
    \calA \setminus \calN$ such that $Z \subset Y$ and $(Z, \calA_Z,
    \calN_Z)$ has the property $\calP$.
  \end{enumerate}
  \begin{proof}[Proof of the equivalence]
    (A) $\implies$ (B) For $Y \in \calA \setminus \calN$, an
    application of Lemma~\ref{supDistrib} gives that $Y$ is an
    essential supremum of $\{Y \cap Z : Z \in \calA \text{ and } (Z,
    \calA_Z, \calN_Z) \text{ has the property } \calP\}$. Therefore,
    there must be some $Z \in \calA$ such that $Y \cap Z \not\in
    \calN$ and $(Z, \calA_Z, \calN_Z)$ has the property $\calP$. The
    subset $Y \cap Z \subset Y$ establishes (B).

    (B) $\implies$ (A) Clearly, $X$ is an $\calN$-essential upper
    bound of the collection $\calA \cap \{Z : (Z, \calA_Z, \calN_Z)
    \text{ has the property } \calP\}$. Let $S$ be another upper
    bound. If $X \setminus S$ were not negligible, (B) gives the
    existence of some measurable $Z \in \calA \setminus \calN$ such
    that $Z \subset X \setminus S$ and $(Z, \calA_Z, \calN_Z)$ has the
    property $\calP$. But $Z \setminus S = Z \not\in \calN$, which
    contradicts that $S$ is an essential upper bound.
  \end{proof}
  
  For instance, in a semi-finite measure space $(X, \calA, \mu)$, any
  non negligible set $A \in \calA \setminus \calN_\mu$ contains a
  measurable subset $Z$ of nonzero finite measure. By (B), this
  implies that the associated MSN $(X, \calA, \calN_\mu)$ is locally
  strictly localizable.

  We conclude this section with an important property of ``local
  isomorphism'' that holds for $\calP$-versions.
\end{Empty}

\begin{Proposition}
  \label{localiso}
  Let $(X, \calA, \calN)$ be a saturated MSN and $[(\hat{X},
  \hat{\calA}, \hat{\calN}), \bp]$ a $\calP$-version of it. Fix a
  representative map $p \in \bp$. For any $F \in \calA$, we set
  $\hat{F} \defeq p^{-1}(F)$ and we call $\bp_F \colon (\hat{F},
  \hat{\calA}_{\hat{F}}, \hat{\calN}_{\hat{F}}) \to (F, \calA_F,
  \calN_F)$ the morphism induced by the restriction $p_F \colon
  \hat{F} \to F$ of $p$.
  \begin{enumerate}
  \item[(A)] $[(\hat{F}, \hat{\calA}_{\hat{F}},
    \hat{\calN}_{\hat{F}}), \bp_F]$ is the $\calP$-version of $(F,
    \calA_F, \calN_F)$;
  \item[(B)] If $(F, \calA_F, \calN_F)$ has the property $\calP$, then
    $\bp_F$ is an isomorphism.
  \end{enumerate}
\end{Proposition}

\begin{proof}
  (A) Since the property $\calP$ is hereditary, we can assert
  $(\hat{F}, \hat{\calA}_{\hat{F}}, \hat{\calN}_{\hat{F}})$ has it. We
  also readily check that $p_F$ is supremum preserving. Let $\bq
  \colon (Y, \calB, \calM) \to (F, \calA_F, \calN_F)$ be a supremum
  preserving morphism starting from a saturated MSN with the property
  $\calP$. Then $\boldsymbol{\iota_F} \circ \bq$ is a supremum
  preserving morphism ending in $(X, \calA, \calN)$. It has a lifting
  $\br \colon (Y, \calB, \calM) \to (\hat{X}, \hat{\calA},
  \hat{\calN})$. Let $r \in \br$ and $q \in \bq$ be
  representatives. As $p(r(y)) = q(y)$ for $\calM$-almost all $y \in
  Y$, we lose no generality in supposing that $r$ has values in
  $\hat{F}$. Calling its restriction $r' \colon Y \to \hat{F}$, we see
  that $p_F \circ r'$ and $q$ coincide $\calM$-almost everywhere. The
  induced morphism $\br'$ provides a factorization of $\bq$ through
  $(\hat{F}, \hat{\calA}_{\hat{F}}, \hat{\calN}_{\hat{F}})$.

  To establish the uniqueness of this factorization, we proceed as
  follows. For any morphism $\br''$ such that $\bp_F \circ \br'' =
  \bq$ we notice that $\boldsymbol{\iota}_F \circ \bq =
  \boldsymbol{\iota}_F \circ \bp_F \circ \br'' = \bp \circ
  \boldsymbol{\iota}_{\hat{F}} \circ \br''$. Since this holds for
  $\br'$ we obtain $\bp \circ \boldsymbol{\iota}_{\hat{F}} \circ \br'
  = \bp \circ \boldsymbol{\iota}_{\hat{F}} \circ \br''$ and, by
  uniqueness of the factorization relative to the universal property
  of $(\hat{X},\hat{\calA},\hat{\calN})$,
  $\boldsymbol{\iota}_{\hat{F}} \circ \br' =
  \boldsymbol{\iota}_{\hat{F}} \circ \br''$. Thus, $r'$ and $r''$
  coincide $\calM$-almost everywhere.

  (B) If $(F, \calA_F, \calN_F)$ has property $\calP$, then obviously
  $[(F, \calA_F, \calN_F), \boldsymbol{\mathrm{id}}]$ is a second
  $\calP$-version. From the uniqueness of the $\calP$-version, we
  obtain a isomorphism $\br \colon (F, \calA_F, \calN_F) \to (\hat{F},
  \hat{\calA}_{\hat{F}}, \hat{\calN}_{\hat{F}})$ such that
  $\boldsymbol{\mathrm{id}} = \bp_F \circ \br$, whence $\bp_F =
  \br^{-1}$.
\end{proof}

\section{Localizable locally determined MSNs}
\label{sec:ld}

In order to motivate the main definition in this section, we start with the following result, of which we can think as a way of testing whether an 
MSN has a property $\calP$. For instance, if each $F \in \calF$ corresponds to a ccc subMSN, then $(X,\calA,\calN)$ is cccc. The difficulty in applying this proposition stems with both hypotheses: conditions (1) and (2) 
will be turned into a definition in \ref{def.ld}, whereas condition (3), that $\calF$ is {\it disjointed} rather than merely {\it almost disjointed}, calls for techniques that transform almost disjointed generating families (whose existence, in applications, follows from Lemma \ref{zorn}) into partitions -- see the proof of Theorem \ref{existslld} in case $\rmcard \calE \leq \mathfrak{c}$ and the notion of compatible family of densities introduced in Section \ref{sec:pu}.

\begin{Proposition}
\label{disjointed.coproduct}
Let $(X,\calA,\calN)$ be an MSN and $\calF \subset \calA$. Assume that
\begin{enumerate}
\item[(1)] For every $A \subset X$ the following holds:
\[
\big[ \forall F \in \calF : A \cap F \in \calA \big] \Rightarrow A \in \calA ;
\]
\item[(2)] For every $N \subset X$ the following holds:
\[
\big[ \forall F \in \calF : N \cap F \in \calN \big] \Rightarrow N \in \calN ;
\]
\item[(3)] $\calF$ is a partition of $X$.
\end{enumerate}
Then the MSNs $(X,\calA,\calN)$ and $\coprod_{F \in \calF} (F,\calA_F,\calN_F)$ are isomorphic in $\sfMSNsp$.
\end{Proposition}

\begin{proof}
  We abbreviate $(Y,\calB,\calM)$ for $\coprod_{F \in \calF}
  (F,\calA_F,\calN_F)$. Since $\calF$ is a partition of $X$, there is
  a canonical bijection $\vphi : X \to Y$. Its inverse $\vphi^{-1}$ is
  $[(\calB,\calM),(\calA,\calN)]$-measurable, by definition of
  coproduct of MSNs. We now show that $\vphi$ is
  $[(\calA,\calN),(\calB,\calM)]$-measurable. Given $B \in \calB$, we
  note that $\vphi^{-1}(B) = \bigcup_{F \in \calF} B \cap F$, whence
  $\vphi^{-1}(B) \cap F = B \cap F \in \calA$ for every $F \in \calF$,
  by definition of $\calB$. We infer from hypothesis (1) that
  $\vphi^{-1}(B) \in \calA$. Let $M \in \calM$. As above we infer from
  the definition of $\calM$ that $\vphi^{-1}(M) \cap F \in \calN$ for
  every $F \in \calF$, whence $\vphi^{-1}(M) \in \calN$, in view of
  hypothesis (2). In other words, $(X,\calA,\calN)$ and
  $(Y,\calB,\calM)$ are isomorphic in $\sfMSN$. The conclusion follows
  from Proposition \ref{propsp}(A).
\end{proof}

\begin{Empty}
\label{def.ld}
     We borrow the following definition from \cite[211H]{FREMLIN.II}. A measure space $(X, \calA, \mu)$
  is {\em locally determined} whenever it is semi-finite and, for every
  subset $A \subset X$,
  \[
  \big[ \forall E \in \calA^f : A \cap E \in \calA \big] \Rightarrow A \in \calA,
  \]
where, as usual, $\calA^f = \calA \cap \{ E : \mu(E) < \infty \}$.

  The definition relies on the particular collection $\calA^f$ (which 
  is $\calN_\mu$-generating, recall \ref{semi-finite}). This makes sense
  because we are dealing with a measure space. It is a rather
  good surprise that we can define an analogous notion of locally
  determined MSNs, by substituting for $\calA^f$ an arbitrary
  generating collection. Namely, a saturated MSN $(X, \calA, \calN)$
  is called {\em locally determined} whenever the following holds. For every $\calN$-generating collection $\calE \subset \calA$ and every $A \subset X$,
  \[
  \big[ \forall E \in \calE : A \cap E \in \calA\big] \Rightarrow A \in \calA .
  \]
  An MSN that
  is both localizable and locally determined is called {\em lld}.
  
  The following is useful as well. We say that a saturated MSN $(X,\calA,\calN)$ has {\em locally determined negligible sets} whenever the following holds. For every $\calN$-generating collection $\calE \subset \calA$ and every $N \subset X$,
  \[
  \big[ \forall E \in \calE : N \cap E \in \calN\big] \Rightarrow N \in \calN .
  \]
  We observe that if $(X,\calA,\calN)$ is locally determined, then it has 
locally determined negligible sets. Indeed, let $\calE \subset \calA$ and 
$N \subset X$ be as above, we first infer from the local determinacy of $(X,\calA,\calN)$ that $N \in \calA$ and, in turn from the Distributivity Lemma \ref{supDistrib}, that $N$ is an $\calN$-essential supremum of $\{ N \cap E : E \in \calE \}$. Therefore, $N \in \calN$.

  Next we prove some elementary properties concerning locally
  determined MSNs. In particular, the consistency between both notions
  of local determination (for complete semi-finite measure spaces and MSNs) is established in
  Proposition~\ref{elemld}(F). Here, the semi-finiteness property of a measure space is critical as the following example shows. We consider $\calH^1$, the 1-dimensional Hausdorff measure in $\R^2$ and $\calA$ the $\sigma$-algebra consisting of $\calH^1$-measurable subsets of $\R^2$ in the sense of Carath\'eodory. The following hold:
  \begin{enumerate}
  \item[(a)] $\forall A \subset \R^2 : \big[ \forall F \in \calA^f : A \cap F \in \calA \big] \Rightarrow A \in \calA$;
  \item[(b)] The measure space $(\R^2,\calA,\calH^1)$ is not semi-finite;
  \item[(c)] The (saturated) MSN $(\R^2,\calA,\calN_{\calH^1})$ does not have locally determined negligible sets and, in particular, is not locally determined. 
  \end{enumerate}
  
  For (a), see for instance \cite[6.2]{DEP.19.b}. For (b), see \cite[439H]{FREMLIN.IV}. Now (c) follows for example from \cite[4.4]{DEP.19.b}. It follows from \ref{semi-finite} that $\calA^f$ is not $\calN_{\calH^1}$-generating.
\end{Empty}

\begin{Proposition}
  \label{elemld}
  The following hold.
  \begin{enumerate}
  \item[(A)] Being locally determined is a hereditary property.
  \item[(B)] Being locally determined is a property invariant under
    isomorphisms in $\mathsf{MSN_{sp}}$.
  \item[(C)] A saturated ccc MSN is locally determined.
  \item[(D)] A coproduct of locally determined MSNs is locally determined.
  \item[(E)] A cccc MSN is locally determined.
  \item[(F)] A complete semi-finite measure space $(X, \calA, \mu)$ is
    locally determined (as a measure space) if and only if the MSN
    $(X, \calA, \calN_\mu)$ is locally determined.
  \end{enumerate}
\end{Proposition}

\begin{proof}
  (A) Let $(X, \calA, \calN)$ be a locally determined MSN and $Z \in
  \calA$. Let $\calE$ be an $\calN_Z$-generating family in the subMSN
  $(Z, \calA_Z, \calN_Z)$ and $A \subset Z$ be such that $E \cap A \in
  \calA_Z$ for any $E \in \calE$. The family $\calE \cup \{X \setminus
  Z\}$ is $\calN$-generating in $(X, \calA, \calN)$ and $E \cap A \in
  \calA$ for all $E \in \calE \cup \{X \setminus Z\}$. It follows that
  $A \in \calA$.

  (B) Let $(X, \calA, \calN)$ and $(Y, \calB, \calM)$ be two
  saturated MSNs, $f \colon X \to Y$ and $g \colon Y \to X$ be two
  measurable supremum preserving maps that induce reciprocal
  isomorphisms. Assume that $(X, \calA, \calN)$ is locally
  determined. Let $\calE \subset \calB$ be an $\calM$-generating
  collection and $B \subset Y$ be such that $E \cap B \in \calB$ for all
  $E \in \calE$. Then $f^{-1}(E) \cap f^{-1}(B) = f^{-1}(E \cap B) \in
  \calA$. As $f$ is supremum preserving, $f^{-1}(\calE)$ is
  $\calN$-generating. And as $(X, \calA, \calN)$ is locally
  determined, we infer that $f^{-1}(B) \in \calA$. Therefore
  $g^{-1}(f^{-1}(B)) \in \calB$. But $B \ominus g^{-1}(f^{-1}(B)) \in
  \calM$ and as $(Y, \calB, \calM)$ is saturated we conclude that $B
  \in \calB$.
  
  (C) Let $(X, \calA, \calN)$ be a saturated MSN, $\calE \subset
  \calA$ an $\calN$-generating family and $A \in \calP(X)$ be such
  that $E \cap A \in \calA$ for all $E \in \calE$. By
  Proposition~\ref{ccc=>loc}, there is a countable subset $\calE'
  \subset \calE$ that is $\calN$-generating. Then
  \[
  A = \left(\bigcup_{E \in \calE'} E \cap A\right) \cup \left(A \setminus \bigcup
  \calE' \right)
  \]
  Since $A \setminus
  \bigcup \calE'$ is $\calN$-negligible and $(X, \calA, \calN)$ is saturated, we infer that $A \setminus
  \bigcup \calE'$ is
  $\calA$-measurable. Therefore, $A \in \calA$.

  (D) Let $(X, \calA, \calN)$ be the coproduct of a family $\la(X_i,
  \calA_i, \calN_i)\ra_{i \in I}$ of locally determined MSNs. It is readily saturated. Let
  $\calE \subset \calA$ be an $\calN$-generating family and $A \subset
  X$ such that $E \cap A \in \calA$ for all $E \in \calE$. For all $i
  \in I$, the family $\calE_i \defeq \{E \cap X_i : E \in \calE\}$ is
  $\calN_i$-generating in $(X_i, \calA_i, \calN_i)$ by
  Lemma~\ref{supDistrib}. This observation leads to the fact that $A
  \cap X_i \in \calA_i$ for all $i \in I$, in other words, $A \in
  \calA$.

  (E) This obviously follows from (C) and (D).

  (F) Suppose that the measure space $(X, \calA, \mu)$ is locally
  determined. Let $\calE \subset \calA$ be an $\calN_\mu$-generating
  family and $A \subset X$ be such that $E \cap A \in \calA$ for all
  $E \in \calE$. Let $F \in \calA^f$. By Lemma~\ref{supDistrib}, the
  collection $\{F \cap E : E \in \calE\}$ is $\calN_F$-generating in
  $(F, \calA_F, (\calN_\mu)_F)$ and of course $F \cap E \cap A \in
  \calA_F$ for all $E \in \calE$. On top of that, $(F, \calA_F,
  (\calN_\mu)_F)$ is a ccc MSN by Proposition~\ref{finite=>loc} and it is saturated. We
  get from (B) above that $A \cap F$ is measurable. As this happens
  for all $F \in \calA^f$, we conclude that $A \in \calA$.

  Conversely, suppose the MSN $(X, \calA, \calN_\mu)$ is locally
  determined. Owing to the semi-finiteness of $(X, \calA, \mu)$, the
  collection $\calA^f$ is $\calN_\mu$-generating. Then $(X, \calA,
  \mu)$ is easily seen to be locally determined: if $A \in \calP(X)$
  satisfies $A \cap F \in \calA$ for all $F \in \calA^f$, then $A \in
  \calA$.
\end{proof}

\begin{Empty}[lld version of an atomic MSN]
\label{atomic.lld}
  As a first result, we mention that the lld version of an atomic MSN
  $(X, \calA, \{\emptyset\})$ is $[(X, \calP(X), \{\emptyset\}),
  \boldsymbol{\iota}]$, where $\boldsymbol{\iota}$ is the morphism
  induced by the identity map (that $\boldsymbol{\iota}$ is
  supremum preserving follows from~\ref{atomic}). This amounts to
  prove that, for any lld MSN $(Y, \calB, \calM)$, a $[(\calB, \calM),
    (\calA, \{\emptyset\})]$-measurable supremum preserving map $q
  \colon Y\to X$ is automatically $(\calB,
  \calP(X))$-measurable. Indeed, let $S \in \calP(X)$. Then $q^{-1}(S)
  \cap q^{-1}\{x\}$ is either $q^{-1}\{x\}$ or $\emptyset$, hence
  $q^{-1}(S) \cap q^{-1}\{x\}$ is $\calB$-measurable for every $x \in X$. 
Besides, $q$
  is supremum preserving, thus the collection $\{q^{-1}\{x\} : x \in
  X\}$ is $\calM$-generating. By local determination in $(Y, \calB,
  \calM)$ we conclude that $q^{-1}(S) \in \calB$.
\end{Empty}

\begin{Empty}
  Call $\mathsf{LLD_{sp}}$ the full subcategory of $\sfMSNsp$ that
  consists of lld MSNs, and consider the forgetful functor
  $\mathsf{Forget} \colon \mathsf{LLD_{sp}} \to \sfMSNsp$. In
  categorical terms, an lld version is the coreflection of an MSN $(X,
  \calA, \calN)$ along the functor $\mathsf{Forget}$, see
  \cite[Chapter~3]{BORCEUX.1}. In this paper, we do not answer the question whether there exists an lld version for
  each saturated MSN. This is equivalent
  to the existence of a right adjoint $\mathsf{R}$ of
  $\mathsf{Forget}$. As a matter of fact, if such an adjoint exists,
  there would be a natural transformation $\varepsilon \colon
  \mathsf{Forget} \circ \mathsf{R} \implies \rmid_{\sfMSNsp}$ such
  that the pair $[\mathsf{R}(X, \calA, \calN), \varepsilon_{(X, \calA,
    \calN)}]$ gives the lld version of any saturated MSN $(X, \calA,
  \calN)$.

  In search for an abstract proof of the existence of $\mathsf{R}$,
  one might think of using Freyd's Adjunction Theorem,
  \cite[Theorem~3.3.3]{BORCEUX.1}. Following this path, one needs to
  establish (setting aside the solution set condition) that:
  \begin{enumerate}
  \item[(A)] The category $\sfMSNsp$ is cocomplete;
  \item[(B)] The forgetful functor $\mathsf{Forget}$ preserves small
    colimits.
  \end{enumerate}
  
  Assertion (A) boils down to showing that $\sfMSNsp$ has two types of
  small colimits: coproducts and coequalizers. The existence of the
  former is shown in Proposition~\ref{propsp}(D). We do not know
  whether coequalizers exist in $\sfMSNsp$ and it is the main
  difficulty here.

  As for (B), which, regarding the existence of lld versions, is a
  necessary condition even if (A) were to be false, we have already
  proven in Propositions~\ref{coprodLoc} and~\ref{elemld}(D) that
  coproducts of lld MSNs are lld. Our next goal is
  Proposition~\ref{coeqlld} which states that coequalizers of lld MSNs
  are lld. Before that, we need to introduce some notation and a
  lemma.

  For a saturated MSN $(X, \calA, \calN)$ and an arbitrary $\calE \subset 
\calA$, we define
  \begin{align*}
    \calA_{\calE} & \defeq \calP(X) \cap \{A : E \cap A \in \calA \text{
      for all } E \in \calE\}, \\ \calN_{\calE} & \defeq \calP(X) \cap
    \{N : E \cap N \in \calN \text{ for all } E \in \calE\}.
  \end{align*}
  It is clear that $(X, \calA_{\calE}, \calN_{\calE})$ is a saturated
  MSN.
\end{Empty}

\begin{Lemma}
  \label{lldeq}
  Let $(X, \calA, \calN)$ be a localizable saturated MSN and $\calE
  \subset \calA$ an $\calN$-generating family. Let
  $\boldsymbol{\iota}$ be the morphism $(X, \calA_{\calE}, \calN_{\calE})
  \to (X, \calA, \calN)$ induced by the identity map on $X$. Then
  $\mathsf{Bool}(\boldsymbol{\iota})$ is an isomorphism. In
  particular, $\boldsymbol{\iota}$ is supremum preserving and $(X,
  \calA_{\calE}, \calN_{\calE})$ is localizable.
\end{Lemma}

\begin{proof}
  First we make the following observation, to be used later in the
  proof: $\calA \cap \calN_{\calE} = \calN$. Indeed, if $N \in \calA$ is
  such that $E \cap N \in\calN$ for all $E \in \calE$, then by the
  distributivity lemma~\ref{supDistrib}, we conclude that $N \in
  \calN$. This proves the inclusion $\calA \cap \calN_{\calE} \subset
  \calN$, the reciprocal being trivial.

  The identity map $\iota \colon X \to X$ is $[(\calA_{\calE},
    \calN_{\calE}), (\calA,\calN)]$-measurable because $\calA \subset
  \calA_{\calE}$ and $\calN \subset \calN_{\calE}$. Let us show that
  $\mathsf{Bool}(\boldsymbol{\iota}) \colon \calA /\calN \to
  \calA_{\calE} / \calN_{\calE}$ is injective by inspecting its
  kernel. Let $\boldsymbol{A} \in \calA / \calN$ be a class
  represented by $A \in \calA$ such that
  $\mathsf{Bool}(\boldsymbol{\iota})(\boldsymbol{A}) = 0$, in other
  words $A = \iota^{-1}(A) \in \calN_{\calE}$. Then $A \in \calA \cap
  \calN_{\calE} = \calN$. This means that $\boldsymbol{A} = 0$ in $\calA
  / \calN$. Therefore, $\mathsf{Bool}(\boldsymbol{\iota})$ is
  injective.

  Now let us show that $\mathsf{Bool}(\boldsymbol{\iota})$ is surjective. 
To
  this end, let $\boldsymbol{H} \in \calA_{\calE} / \calN_{\calE}$ be a
  class represented by $H \in \calA_{\calE}$. We ought to prove
  that $\boldsymbol{H}$ is in the range of
  $\mathsf{Bool}(\boldsymbol{\iota})$. Set $\calF \defeq \{E \cap H :
  E \in \calE\}$. Note that $\calF \subset \calA$. The localizability
  of $(X, \calA, \calN)$ guarantees that $\calF$ has an
  $\calN$-essential supremum $S \in \calA$. In particular, $E \cap H
  \setminus S \in \calN$ for all $E \in \calE$, meaning that $H
  \setminus S \in \calN_{\calE}$.

  We also claim that $S \setminus H \in \calN_{\calE}$. Indeed, let $E_0
  \in \calE$. Set $S' \defeq S \setminus (E_0 \cap S \setminus H)$. We
  note that $S' \in \calA$. For all $E \in \calE$, we have $E \cap H
  \setminus S' = E \cap H \setminus S \in \calN$, as $H \setminus S
  \in \calN_{\calE}$. This means that $S'$ is an $\calN$-essential upper
  bound of $\calF$. It follows that $S \setminus S' = E_0 \cap H
  \setminus S \in \calN$. As $E_0\in \calE$ is arbitrary, we obtain $S
  \setminus H \in \calN_{\calE}$, as required.

  We proved that $H \ominus S \in \calN_{\calE}$. Calling
  $\boldsymbol{S}$ the equivalence class of $S$ in $\calA / \calN$, we
  have that $\mathsf{Bool}(\boldsymbol{\iota})(\boldsymbol{S}) =
  \boldsymbol{H}$.
\end{proof}

\begin{Proposition}
  \label{coeqlld}
  Consider the following diagram in $\sfMSNsp$, where $((Z, \calC,
  \calP), \bh)$ is the coequalizer of $\mathbf{f}, \bg$.
  \[
  \begin{tikzcd}
    (X, \calA, \calN) \arrow[r, shift left, "\mathbf{f}"] \arrow{r}[shift
      right, below]{\bg} & (Y, \calB, \calM) \arrow{r}{\bh} &
    (Z, \calC, \calP)
  \end{tikzcd}
  \]
  \begin{enumerate}
  \item[(A)] If $(Y, \calB, \calM)$ is
    localizable, so is $(Z, \calC, \calP)$.  
  \item[(B)] If $(Y, \calB, \calM)$ is lld,
    so is $(Z, \calC, \calP)$. 
  \end{enumerate}
\end{Proposition}

\begin{proof}
  (A) Let us call $\boldsymbol{2}$ the special MSN $(\{0, 1\},
  \calP(\{0,1\}), \{\emptyset\})$.  First we show the following
  intermediate result: for any MSN $(X, \calA, \calN)$, there is a
  one-to-one correspondence $\Upsilon_X$ between the Boolean algebra
  $\mathsf{Bool}(X, \calA,\calN)$ and the set of morphisms $\rmHom((X,
  \calA, \calN), \boldsymbol{2})$ (those are automatically supremum
  preserving since the Boolean algebra of $\boldsymbol{2}$ is
  finite). Given a class $\boldsymbol{A} \in \mathsf{Bool}(X, \calA,
  \calN)$, represented by a set $A$, the characteristic function
  $\ind_A \colon X \to \{0, 1\}$ induces a morphism $\boldsymbol{1_A}$
  which only depends on the equivalence class
  $\boldsymbol{A}$. Indeed, if $A'$ is another representative of
  $\boldsymbol{A}$, then $\ind_A$ and $\ind_{A'}$ coincide
  $\calN$-almost everywhere. We set $\Upsilon_X(\boldsymbol{A}) \defeq
  \boldsymbol{1_A}$.

  This map is surjective because each morphism $\boldsymbol{\varphi} \colon
  (X, \calA, \calN) \to \boldsymbol{2}$ is represented by a map
  $\varphi \in \boldsymbol{\varphi}$ which has the form $\varphi =
  \ind_{\varphi^{-1}(\{1\})}$. It is injective because if $\ind_A$
  coincides with $\ind_B$ almost everywhere, for measurable sets $A$
  and $B$, then $A$ and $B$ yield the same equivalence class in
  $\mathsf{Bool}(X, \calA, \calN)$.

  Now we turn to the proof of conclusion (A). By naturality of
  $\Upsilon$, the following diagram is commutative, where $\rmHom(\bh,
  \mathbf{2})$, $\rmHom(\mathbf{f}, \mathbf{2})$ and $\rmHom(\bg,
  \mathbf{2})$ denote the right composition with $\bh$, $\mathbf{f}$,
  and $\bg$, respectively.
  \[
  \begin{tikzcd}
    \rmHom((Z, \calC, \calP), \boldsymbol{2})  \arrow[rr, "{\rmHom(\bh, \boldsymbol{2})}"] & &
    \rmHom((Y, \calB, \calM), \boldsymbol{2})  \arrow[rr, shift left, "{\rmHom(\mathbf{f},
        \boldsymbol{2})}"] \arrow{rr}[shift right,
      below]{\rmHom(\mathbf{g}, \boldsymbol{2})} & & \rmHom((X, \calA,
    \calN), \boldsymbol{2}) 
    \\ \mathsf{Bool}(Z, \calC, \calP)\arrow[u,
      "{\Upsilon_Z}"] \arrow[rr,
      "{\mathsf{Bool}(\bh)}"] & & \mathsf{Bool}(Y, \calB, \calM)\arrow[u,
      "{\Upsilon_Y}"]
    \arrow[rr, shift left, "{\mathsf{Bool}(\mathbf{f})}"]
    \arrow{rr}[shift right, below]{\mathsf{Bool}(\bg)}
    &&
    \mathsf{Bool}(X,\calA,\calN)\arrow[u, "{\Upsilon_X}"]
  \end{tikzcd}
  \]
  We show that $\rmHom(\bh,\mathbf{2})$ is injective. Indeed, if
  $\boldsymbol{\vphi}$ and $\boldsymbol{\psi}$ are such that
  $\rmHom(\bh,\mathbf{2})(\boldsymbol{\vphi})=\rmHom(\bh,\mathbf{2})(\boldsymbol{\psi})$
  then, upon letting $\bk = \boldsymbol{\vphi} \circ \bh =
  \boldsymbol{\psi} \circ \bh$, we infer that $\bk \circ \mathbf{f} =
  \bk \circ \bg$. By the universal property of $(Z,\calC,\calP)$,
  there exists a unique $\boldsymbol{\ell} \in
  \rmHom((Z,\calC,\calP),\mathbf{2})$ such that $\boldsymbol{\ell}
  \circ \bh = \bk$. Since $\boldsymbol{\vphi}$ and $\boldsymbol{\psi}$
  have the property of $\boldsymbol{\ell}$, we conclude that they
  coincide. Similarly, the universal property of coequalizers tells us
  that the range of $\rmHom(\bh,\mathbf{2})$ consists of those
  morphisms $\bk$ such that $\rmHom(\mathbf{f},
  \mathbf{2})(\bk) = \rmHom(\bg,
  \mathbf{2})(\bk)$. On the second line of the
  diagram, these two observations translate to the fact that
  $\mathsf{Bool}(\bh)$ induces an isomorphism of Boolean algebras from
  $\mathsf{Bool}(Z, \calC, \calP)$ onto the Boolean subalgebra
  \[
  \frA \defeq \mathsf{Bool}(Y, \calB, \calM) \cap \{ \xi :
  \mathsf{Bool}(\mathbf{f})(\xi) = \mathsf{Bool}(\bg)(\xi)\}.
  \]
  It remains to prove that $\frA$ is Dedekind complete. Let $\frE
  \subset \frA$ be a collection. It has a supremum $s$ in
  $\mathsf{Bool}(Y, \calB, \calM)$, as $(Y, \calB, \calM)$ is
  localizable. Since $\mathbf{f}$ and $\bg$ are supremum preserving,
  we have $\mathsf{Bool}(\mathbf{f})(s) = \sup
  \mathsf{Bool}(\mathbf{f})(\frE) = \sup
  \mathsf{Bool}(\mathbf{g})(\frE) = \mathsf{Bool}(\bg)(s)$. Hence $s
  \in \frA$ and $\frA$ is Dedekind complete.

  (B) That $(Z, \calC, \calP)$ is localizable follows from (A). Let
  $\calG \subset \calC$ be any $\calP$-generating family.  We wish to
  prove that $\calC_{\calG} \subset \calC$. If we manage to do so,
  then $(Z, \calC, \calP)$ is locally determined, as $\calG$ is
  arbitrary.

  Let $h$ be a representative of $\bh$. By definition, it is $[(\calB,
    \calM), (\calC, \calP)]$-measurable and supremum preserving. We
  claim that it is, in fact, $[(\calB, \calM), (\calC_{\calG},
    \calP_{\calG})]$-measurable. Indeed, let $C \in \calC_{\calG}$. For
  all $G \in \calG$, we have $G \cap C \in \calC$ which implies that
  $h^{-1}(G) \cap h^{-1}(C) = h^{-1}(G \cap C) \in \calB$. Moreover,
  $h^{-1}(\calG)$ is $\calM$-generating, as $\calG$ is $\calP$-generating 
and $h$ is
  supremum preserving. Thus, since $(Y, \calB, \calM)$ is locally
  determined, we have that $h^{-1}(C) \in \calB$. Next, if $P \in
  \calP_{\calG}$, then $h^{-1}(P) \in \calB$ by what precedes and $h^{-1}(P) \cap
  h^{-1}(G) = h^{-1}(P \cap G) \in \calM$ for all $G \in \calG$. By
  the distributivity lemma~\ref{supDistrib}, we obtain $h^{-1}(P) \in
  \calM$.

  Denote as $\bh' \colon (Y,\calB, \calM) \to (Z, \calC_{\calG},\calP_{\calG})$
  the morphism induced by $h$, and denote as $\boldsymbol{\iota} \colon (Z,
  \calC_{\calG},\calP_{\calG}) \to (Z, \calC, \calG)$ the morphism
  induced by the identity map $\rmid_Z$. By Lemma~\ref{lldeq}, we have
  $\mathsf{Bool}(\bh') = \mathsf{Bool}(\boldsymbol{\iota})^{-1} \circ
  \mathsf{Bool}(\bh)$, which is the composition of two
  supremum preserving morphisms of Boolean algebras. Thus, $\bh'$ is a
  supremum preserving as well. Also, we recall $\bh \circ \mathbf{f} =
  \bh \circ \bg$. As $\bh$ and $\bh'$ are induced by the same map, we
  deduce that $\bh' \circ \mathbf{f} = \bh' \circ \bg$. By the
  universal property of coequalizers, there is a morphism $\bk \colon
  (Z, \calC, \calP) \to (Z, \calC_{\calG}, \calP_{\calG})$ such that $\bh'
  = \bk \circ \bh$.
  \[
  \begin{tikzcd}
    & & (Z, \calC_{\calG}, \calP_{\calG}) \arrow[d, shift left, "\boldsymbol{\iota}"]
    \\ (X, \calA, \calN) \arrow[r, shift left, "\mathbf{f}"]
    \arrow{r}[shift right, below]{\bg} & (Y, \calB, \calM)
    \arrow{r}{\bh} \arrow[ru, "\bh'"] & (Z, \calC, \calP)
    \arrow[u, "\bk"]
  \end{tikzcd}
  \]
  Hence $\boldsymbol{\iota} \circ \bk \circ \bh = \boldsymbol{\iota}
  \circ \bh' = \bh = \rmid_{(Z,\calC,\calP)} \circ \bh$. The
  uniqueness in the universal property of equalizers implies that
  $\bh$ is an epimorphism. Thus $\boldsymbol{\iota} \circ \bk =
  \rmid_{(Z,\calC,\calP)}$. A representative $k \in \bk$ must satisfy
  $z = \rmid_Z(k(z)) = k(z)$ for $\calP$-almost all $z \in Z$, 
    i.e. $P = Z \cap \{ z : z \neq k(z) \} \in \calP$. Let $C \in
  \calC_{\calG}$. Since $k$ is $(\calC,\calC_{\calG})$-measurable, it
  follows that $k^{-1}(C) \in \calC$. Since $k^{-1}(C) \ominus C
  \subset P$, we deduce that $k^{-1}(C) \ominus C \in \calC$ and, in
  turn, $C \in \calC$.
\end{proof}

\begin{Empty}
  The last two results of this section show that the category $\mathsf{LLD_{sp}}$
  has better categorical properties than $\sfMSNsp$: it has equalizers
  in full generality. This result is reminiscent of
  \cite[214Ie]{FREMLIN.II}.
\end{Empty}

\begin{Proposition}
  \label{subMSNLLDsp}
  Let $(X, \calA, \calN)$ be an lld MSN and $Y \subset X$ any
  subset. Then the subMSN $(Y, \calA_Y, \calN_Y)$ is lld and the
  canonical morphism $\boldsymbol{\iota}_Y \colon (Y, \calA_Y,
  \calN_Y) \to (X, \calA, \calN)$ is supremum preserving.
\end{Proposition}

\begin{proof}
  First we show that the map $\iota_Y \colon Y \to X$ is
  supremum preserving. Let $\calE \subset \calA$ and assume $S \in \calA$ 
is an $\calN$-essential supremum of $\calE$. The set $S \cap Y = \iota_Y^{-1}(S)$ is an $\calN_Y$-essential
  upper bound of $\iota_Y^{-1}(\calE)$.
  Let $U \in \calA_Y$ be an arbitrary $\calN_Y$-essential upper bound of $\iota_Y^{-1}(\calE)$. We ought to show that
  $S \cap Y \setminus U \in \calN_Y$. For all $E \in \calE$, one has
  $E \cap S \cap Y \setminus U \subset E \cap Y \setminus U \in
  \calN_Y$. As $(X, \calA, \calN)$ is saturated, $\calN_Y \subset \calN$ and, also, $E \cap S \cap Y \setminus U \in \calN$. Of course $(X \setminus S) \cap
  S \cap Y \setminus U = \emptyset$ is also $\calN$-negligible. Since the
  family $\calE \cup \{X \setminus S\}$ is $\calN$-generating and $(X,
  \calA,\calN)$ is locally determined, we deduce that $S \cap Y
  \setminus U$ is $\calA$-measurable and, in turn, that it is
  $\calN$-negligible by the distributivity lemma \ref{supDistrib}. The proof that $\boldsymbol{\iota}_Y$ is
  supremum preserving is complete.

  Since the morphism $\mathsf{Bool}(\boldsymbol{\iota}_Y) \colon \calA
  / \calN \to \calA_Y / \calN_Y$ is onto, supremum preserving, and
  $\calA / \calN$ is Dedekind complete, so is $\calA_Y / \calN_Y$,
  meaning that $(Y, \calA_Y, \calN_Y)$ is localizable. It remains to show 
that $(X,\calA,\calN)$ is locally determined.

  We claim the following: {\it If $\calE \subset \calA_Y$ is
   $\calN_Y$-generating and $N \in \calP(Y)$ satisfies $E
  \cap N \in \calN_Y$ for all $E \in \calE$, then $N \in \calN_Y$.} By
  definition of $\calA_Y$, any set $E \in \calE$ can be written as $E
  = E' \cap Y$, for some $E' \in \calA$, so there is a subset $\calE'
  \subset \calA$ such that $\calE = \iota_Y^{-1}(\calE')$. The
  localizability of $(X, \calA, \calN)$ guarantees the existence of an
  $\calN$-essential supremum $S$ of $\calE'$. For all $E' \in \calE'$
  one has $E' \cap N = E' \cap Y \cap N \in \calN_Y \subset \calN$,
  because $E' \cap Y \in \calE$. Also $S \cap Y = \iota_Z^{-1}(S)$ is
  an $\calN_Y$-essential supremum of $\calE =
  \iota_Z^{-1}(\calE')$, by the first paragraph. Recalling that $\calE$ is
  $\calN_Y$-generating, we find that $Y \setminus S = Y \setminus (S \cap Y) \in \calN_Y$. Consequently, $(X \setminus S) \cap
  N \subset Y \setminus S \in \calN_Y \subset \calN$. As $(X, \calA, \calN)$ is saturated, we
  find that $(X \setminus S) \cap N \in \calN$. In conclusion, $E'
  \cap N \in \calN$ for any $E'$ that belongs to the
  $\calN$-generating family $\calE' \cup \{X \setminus S\}$. Since $(X,
  \calA, \calN)$ is locally determined, we infer that $N \in \calA$
  and then that $N \in \calN$ by the distributivity lemma \ref{supDistrib}. As $N
  \subset Y$, we conclude that $N \in \calN_Y$.

  Now let $\calE \subset \calA_Y$ be an $\calN_Y$-generating collection
  and $A \in \calP(Y)$ be such that $E \cap A \in \calA_Y$ for all $E
  \in \calE$. We want to prove that $A \in \calA_Y$. As $(Y, \calA_Y,
  \calN_Y)$ is localizable, $\{E \cap A : E \in \calE\}$ has an
  $\calN_Y$-essential supremum $S$. This implies that $E \cap A
  \setminus S \in \calN_Y$ for all $E \in \calE$. By the claim above, $A \setminus S \in \calN_Y$.

  Fix $E_0 \in \calE$. Note that $E_0 \cap (S \setminus A) = (E_0 \cap
  S) \setminus (E_0 \cap A) \in \calA_Y$. Also, 
  \begin{align*}
    E \cap A \setminus \left((S \setminus (E_0 \cap S \setminus A)
    \right) & = E \cap A \cap ((Y \setminus S) \cup (E_0 \cap S
    \setminus A))\\ & = E \cap A \setminus S \in \calN_Y,
  \end{align*}
  for all $E \in \calE$. In other words, $S \setminus (E_0
  \cap S \setminus A)$ is an $\calN_Y$-essential upper bound of $\{E
  \cap A : E \in \calE\}$. As $S$ is an $\calN_Y$-essential supremum
  of this family, $S \setminus (S \setminus (E_0 \cap S \setminus A))
  = E_0 \cap S \setminus A \in \calN_Y$. Applying again the claim above, we deduce that $S
  \setminus A \in \calN_Y$ from the arbitrariness of $E_0$. Summing up, $A \ominus S \in \calN_Y$. As $S \in
  \calA_Y$ we infer that $A
  \in \calA_Y$. The proof that $(Y, \calA_Y, \calN_Y)$
  is locally determined is now complete.
\end{proof}

\begin{Corollary}
\label{eq.lld}
  $\mathsf{LLD_{sp}}$ has equalizers preserved by the forgetful
  functor $\mathsf{LLD_{sp}} \to \mathsf{MSN}$.
\end{Corollary}

\begin{proof}
  Consider a pair of supremum preserving morphisms $\mathbf{f}, \bg
  \colon (X, \calA, \calN) \to(Y, \calB, \calM)$ in the category
  $\mathsf{LLD_{sp}}$, represented by maps $f \in \mathbf{f}$ and $g
  \in \mathbf{g}$. Let $\bh \colon (T, \calC, \calP) \to (X, \calA,
  \calN)$ be another supremum preserving morphism in
  $\mathsf{LLD_{sp}}$, such that $\mathbf{f} \circ \bh = \bg \circ
  \bh$.

  Set $Z = \{f = g\}$. We know since Proposition~\ref{equalizerMSN}
  that $((Z, \calA_Z, \calN_Z), \boldsymbol{\iota}_Z)$ is the
  equalizer of the pair $\mathbf{f}, \bg$ in the category
  $\mathsf{MSN}$, so there is a unique morphism $\bh' \colon (T,
  \calC, \calP) \to (X, \calA, \calN)$ such that $\bh =
  \boldsymbol{\iota}_Z \circ \bh'$. By the
  proposition~\ref{subMSNLLDsp}, $\boldsymbol{\iota}_Z$ is
  supremum-preserving and $(Z, \calA_Z, \calN_Z)$ is lld.  It remains
  to prove that $\bh'$ is supremum preserving. This follows from the
  fact that $\mathsf{Bool}(\bh) = \mathsf{Bool}(\bh') \circ
  \mathsf{Bool}(\boldsymbol{\iota}_Z)$, where $\mathsf{Bool}(\bh)$ is
  supremum preserving and $\mathsf{Bool}(\boldsymbol{\iota}_Z)$ is
  supremum preserving and surjective.
\end{proof}

\section{Gluing measurable functions}

\begin{Empty}
  \label{defGluing}
  Let $(X, \calA, \calN)$ be an MSN and $(Y, \calB)$ a measurable
  space. Let $\calE \subset \calA$ be a collection. A {\em family
    subordinated to $\calE$} is a family of functions $\la f_E \ra_{E
    \in \calE}$ such that:
  \begin{enumerate}
    \item[(1)] $f_E \colon E \to Y$ is $(\calA_E, \calB)$-measurable
      for every $E \in \calE$.
  \end{enumerate}
  We further say that $\la f_E \ra_{E \in \calE}$ is {\em compatible} whenever
  \begin{enumerate}
    \item[(2)] for all pairs $E, E' \in \calE$ one has $E \cap E' \cap
      \{f_E \neq f_{E'}\} \in \calN$.
  \end{enumerate}
  A {\em gluing} of a compatible family $\la f_E \ra_{E \in \calE}$
  subordinated to $\calE$ is a function $f \colon X \to Y$ such that
  \begin{enumerate}
  \item[(3)] $f$ is $(\calA, \calB)$-measurable;
  \item[(4)] $E \cap \{f \neq f_{E}\} \in \calN$ for every $E \in
    \calE$.
  \end{enumerate}
  
  In this section, we will be mainly concerned about the existence of
  gluings, as they will be of use in the construction of the cccc
  version of a locally ccc MSN in Section~\ref{sec:slv}. This turns
  out to depend both on the domain and the target space. In case where
  $(Y, \calB)$ is the real line equipped with its Borel
  $\sigma$-algebra $(\R, \calB(\R))$, we can glue measurable functions
  together if $(X, \calA, \calN)$ is localizable. In fact, this
  important property is a characterization of localizability. The
  interested reader may find a proof of this classical result
  expressed in the language of MSNs in \cite[Proposition
    3.13]{DEP.19.b}. Only the measurable structure of $(\R,
  \calB(\R))$ is involved, thus, the result holds in the more general case where
  $(Y, \calB)$ is a standard Borel space, see \cite[Chapter~3]{SRIVASTAVA}.
  
  Many questions arise when we remove the condition that $(Y, \calB)$
  is a standard Borel space. In this case, we need some additional
  assumptions on $(X, \calA, \calN)$. We will focus on two cases: $(X,
  \calA, \calN)$ is cccc or lld. But first, we prove that a gluing
  inherits some of the properties of the functions $f_E$.
\end{Empty}

\begin{Lemma}
  \label{propGluing}
  Let $(X, \calA, \calN)$ be a saturated MSN, $(Y, \calB, \calM)$
  an MSN, and $\calE \subset \calA$ an $\calN$-generating collection. We let $\la f_E \ra_{E \in \calE}$ be
  a compatible family of functions subordinated to $\calE$ and we
  assume that:
  \begin{enumerate}
  \item[(1)] for every $E \in \calE$, the map $f_E$ is $[(\calA_E,
    \calN_E), (\calB, \calM)]$-measurable;
  \item[(2)] the family $\la f_E \ra_{E \in \calE}$ has a gluing $f$.
  \end{enumerate}
  Then
  \begin{enumerate}
  \item[(A)] the gluing $f$ is
    $[(\calA, \calN), (\calB, \calM)]$-measurable;
  \item[(B)] if $f_E$ is supremum preserving, for every $E \in \calE$,
    then so is $f$.
  \end{enumerate}
\end{Lemma}

\begin{proof}
We start with the following easy observation. For each $E \in \calE$ and $B \in \calB$ one has $f_E^{-1}(B) \ominus (E \cap f^{-1}(B)) \subset E \cap \{ f_E \neq f \} \in \calN$.

  (A) As the gluing $f$ is $(\calA, \calB)$-measurable by definition,
  we need only show that $f^{-1}(M) \in \calN$ for $M \in \calM$.  Since $f_E^{-1}(M)$ is $\calN$-negligible, the above observation applied with $B=M$ ensures that $E \cap f^{-1}(M) \in \calN$ for any $E
  \in \calE$. We next use Lemma~\ref{supDistrib} to assert that
  $f^{-1}(M)$ is an $\calN$-essential supremum of $\{E \cap f^{-1}(M)
  : E \in \calE\}$. This forces $f^{-1}(M)$ to be $\calN$-negligible.

  (B) Let $\calF \subset \calB$ be a collection that admits an
  $\calM$-essential supremum $S$. Since $f_E$ is supremum preserving
  for every $E \in \calE$, $f_E^{-1}(S)$ is an $\calN$-essential supremum 

of $\{ f_E^{-1}(F) : F \in \calF \}$ and it ensues from the observation above, applied with $B \in \{S\} \cup \calF$, that $E \cap
   f^{-1}(S)$ is an $\calN$-essential supremum of $\{E \cap f^{-1}(F) : F 
\in
  \calF\}$. Therefore,
  \begin{equation*}
  \begin{split}
  \quad\quad f^{-1}(S) & = \calN\text{-}\rmesssup_{E \in \calE} E \cap f^{-1}(S)
  \intertext{(by Lemma \ref{supDistrib})}
  & = \calN\text{-}\rmesssup_{E \in \calE} \left( \calN\text{-}\rmesssup_{F \in \calF} E \cap f^{-1}(F)\right)
  \intertext{(from what precedes)}
  & = \calN\text{-}\rmesssup_{F \in \calF} \left( \calN\text{-}\rmesssup_{E \in \calE} E \cap f^{-1}(F)\right) \\
  & = \calN\text{-}\rmesssup_{F \in \calF} f^{-1}(F)
  \intertext{(by Lemma \ref{supDistrib}).}
  \end{split}
  \end{equation*}
\end{proof}

\begin{Proposition}
  \label{uniqueGluing}
  Let $(X, \calA, \calN)$ be a locally determined MSN and $(Y, \calB)$
  be any nonempty measurable space. Let $\calE \subset \calA$ be an
  $\calN$-generating collection. If a compatible family $\la f_E \ra_{E
    \in \calE}$ has a gluing, then it is unique up to equality almost
  everywhere.
\end{Proposition}

\begin{proof}
  Let $f, g \colon X \to Y$ be two gluings of $\la f_E \ra_{E \in
    \calE}$. We warn the reader that the measurability of $\{f \neq g\}$
  is not immediate, since the diagonal $\{(y, y) : y
  \in Y\}$ may not be measurable in $(Y^2, \calB \otimes \calB)$. Notwithstanding, for
  all $E \in \calE$, we have $E \cap \{f \neq g\} \subset (E \cap \{f
  \neq f_E\}) \cup (E \cap \{g \neq f_E\})$. Since $f, g$ are gluings
  and $(X, \calA, \calN)$ is saturated, it follows that $E \cap \{f
  \neq g\} \in \calN$. This happens for any $E$ in the
  $\calN$-generating set $\calE$. By local determination and the
  distributivity lemma~\ref{supDistrib}, $\{f \neq g\} \in \calN$.
\end{proof}

\begin{Proposition}
  \label{Gluingcccc}
  Let $(X, \calA, \calN)$ be a cccc MSN and $(Y, \calB)$ be any
  nonempty measurable space. Let $\calE \subset \calA$ be an
  $\calN$-generating collection. Any compatible family $\la f_E \ra_{E
    \in \calE}$ subordinated to $\calE$ admits a unique gluing $f$ up
  to equality $\calN$-almost everywhere.
\end{Proposition}

\begin{proof}
  First observe that the uniqueness of the gluing up to almost
  everywhere equality follows from Proposition~\ref{uniqueGluing}, as
  a cccc MSN is locally determined by Proposition~\ref{elemld}(E).
  
  Let us treat the special case where $(X, \calA, \calN)$ is a
  saturated ccc MSN. According to Proposition~\ref{ccc=>loc}, we can
  find a sequence of sets $\la E_i \ra_{i \in \N}$ in $\calE$ such
  that $\bigcup_{i \in \N} E_i$ provides an $\calN$-essential supremum
  of $\calE$. We then define the $(\calA, \calB)$-measurable map $f
  \colon X \to Y$ which, for all $i \in \N$, coincides with $f_i$ on
  the set $E_i \setminus \bigcup_{j < i} E_j$, and maps the negligible
  set $N \defeq X \setminus \bigcup_{i \in \N} E_i$ to some arbitrary
  point. Let $E \in \calE$. For every $i \in \N$, we have $N_i \defeq
  E \cap E_i \cap \{f_E \neq f_{E_i}\} \in \calN$, by hypothesis. Thus,
  $E \cap \{f \neq f_{E}\} \subset N \cup \bigcup_{i \in \N} N_i$ is
  negligible. 

  Suppose now that $(X, \calA, \calN)$ is a coproduct $\coprod_{i \in
    I}(X_i, \calA_i, \calN_i)$ of saturated ccc MSNs. For each $E \in
  \calE$ and $i \in I$, call $f_{E,i}$ the restriction of $f_E$ to $E
  \cap X_i$. By Lemma~\ref{supDistrib}, $X_i$ is an $\calN_i$-essential supremum of the collection $\{E \cap X_i : E \in
  \calE\}$. Also, $\la f_{E,i}\ra_{E\in \calE}$ is a compatible family
  of measurable maps subordinated to $\la E \cap X_i \ra_{E \in
    \calE}$. From what precedes, it admits a gluing $f_i \colon X_i
  \to Y$. Define $f \defeq \coprod_{i \in I} f_i$. The verification
  that $f$ is a gluing of $\la f_E \ra_{E\in \calE}$ is
  routine. 
\end{proof}

\begin{Empty}
  It would be interesting to know whether cccc MSNs are the only lld
  MSNs such that gluings can always be performed, with no restriction
  on the target space $(Y, \calB)$. This property will be used in the
  next section and justifies the special role played by cccc MSN.
\end{Empty}

\begin{Empty}[Countably separated measurable spaces]  
  A measurable space $(Y, \calB)$ is called {\em countably separated}
  whenever there is a countable set $\calC \subset \calB$ such that
  for any distinct $y_1, y_2 \in Y$ there exists $C \in \calC$ such
  that $y_1 \in C \not\ni y_2$ or $y_1 \not\in C \ni y_2$. Here is a
  well-known characterization of countably separated spaces.
\end{Empty}

\begin{Proposition}
  \label{caracCS}
  Let $(Y, \calB)$ be a measurable space. The following statements are
  equivalent:
  \begin{enumerate}
  \item[(A)] $(Y, \calB)$ is countably separated;
  \item[(B)] there is an injective measurable map $(Y, \calB) \to (\R,
    \calB(\R))$;
  \item[(C)] there is an injective measurable map $(Y, \calB) \to (X,
    \calB(X))$ to a Polish space $X$.
  \end{enumerate}
\end{Proposition}

\begin{proof}
  (A) $\implies$ (B) Let $\calC \subset \calB$ be a countable set
  that separates the points of $Y$. Let $\la C_n \ra_{n \in \N}$ be a
  enumeration of $\calC$ and $h \colon Y \to \R$ be the map $h =
  \sum_{n=0}^\infty 3^{-n} \ind_{C_n}$. The sets $C_n$ are measurable,
  therefore $h$ is measurable.

  Let $y_1, y_2$ be distinct points in $Y$ and $n_0 \defeq \min \{n
  \in \N : \ind_{C_n}(y_1) \neq \ind_{C_n}(y_2)\}$. Then
  \[
  |h(y_1) - h(y_2)| \geq 3^{-n_0} - \sum_{n=n_0+1}^\infty 3^{-n}
  |\ind_{C_n}(y_1) - \ind_{C_n}(y_2)| \geq 3^{-n_0} -
  \frac{3^{-n_0}}{2} > 0
  \]
  thus $h(y_1) \neq h(y_2)$, which shows that $f$ is injective.

  (B) $\implies$ (C) is obvious.

  (C) $\implies$ (A) Let $\calU$ be a countable basis for the
  topology of $X$. If there is an injective measurable map $h \colon
  (Y, \calB) \to (X, \calB(X))$, then $h^{-1}(\calU) \subset \calB$ is
  a countable set that separates points.
\end{proof}

\begin{Proposition}
  \label{prop68}
  Let $\la (Y_i, \calB_i) \ra_{i \in I}$ be a family of countably
  separated measurable spaces. If $\rmcard I \leq \mathfrak{c}$, then
  $\coprod_{i \in I} (Y_i, \calB_i)$ is countably separated.
\end{Proposition}

\begin{proof}
  For each $i \in I$, there is an injective $(\calB_i,\calB(\R))$-measurable map $h_i \colon
  Y_i \to \R$ by Proposition~\ref{caracCS}.
  Choose an arbitrary injective map $g \colon I \to \R$. Let $h \colon
  \coprod_{i \in I} Y_i \to \R^2$ be the map defined by $h(y_i) \defeq
  (h_i(y_i), g(i))$ for all $i \in I$ and $y_i \in Y_i$. Let $B
  \subset \R^2$ be a Borel set.
  Then, for any $i \in I$, we have $h^{-1}(B) \cap Y_i = h_i^{-1}(\R \cap \{x
   : (x, g(i)) \in B\}) = h_i^{-1}(B^{g(i)})$. This last set is $\calB_i$-measurable
  as $h_i$ is measurable and the horizontal section $B^{g(i)}$ is Borel. As $i$ is arbitrary, we conclude that
  $h^{-1}(B)$ is measurable. This means that $h$ is measurable. By
  Proposition~\ref{caracCS}, it follows that $\coprod_{i \in I} (Y_i,
  \calB_i)$ is countably separated.
\end{proof}

\begin{Remark}
  The restriction on the cardinal of $I$ is necessary, since a
  countably measurable space must have cardinal less or equal than
  $\mathfrak{c}$ by Proposition~\ref{caracCS}(B).
\end{Remark}

\begin{Proposition}
  \label{gluelld}
  Let $(X, \calA, \calN)$ be an lld MSN and $(Y, \calB)$ be a nonempty
  countably separated measurable space. Let $\calE \subset \calA$ be
  $\calN$-generating. Any compatible family $\la f_E \ra_{E \in \calE}$
  subordinated to $\calE$ admits a gluing, unique up to equality
  almost everywhere.
\end{Proposition}

\begin{proof}
  Let $h$ be a measurable injective map $(Y, \calB) \to (\R,
  \calB(\R))$, whose existence follows from Proposition~\ref{caracCS}.
  Now, $\la h \circ f_E \ra_{E \in \calE}$ is still a compatible
  family of measurable functions, this time with values in $(\R,
  \calB(\R))$. As $(X, \calA, \calN)$ is localizable, it admits a
  gluing $g \colon X \to \R$. For all $E \in \calE$, one has $E \cap
  g^{-1}(\R \setminus h(Y)) \subset E \cap \{g \neq h \circ
  f_E\}$. Therefore $E \cap g^{-1}(\R \setminus h(Y))$ is
  negligible. This holds for any $E$ in the $\calN$-generating set
  $\calE$. By local determination and Lemma \ref{supDistrib}, we deduce that $g^{-1}(\R \setminus
  h(Y)) \in \calN$. Thus, we lose no generality in supposing, from now on, that $g$
  takes values in $h(Y)$. Define $f \defeq h^{-1} \circ g$. We claim
  that $f$ is a gluing. For $E \in \calE$, we observe that $E \cap \{f
  \neq f_E\} \subset E \cap \{g \neq h \circ f_E\} \in \calN$, since
  $h$ is injective. Therefore, condition~(4) of~\ref{defGluing} is
  satisfied.

  Also, let $B \in \calB$, then $(E \cap f^{-1}(B)) \ominus
  f_E^{-1}(B) = E \cap \{f \neq f_E\} \in \calN$. Since $f_E$ is
  measurable, we have $f_E^{-1}(B) \in \calA$ and, in turn, $E \cap f^{-1}(B) \in \calA$. Since $E$ is arbitrary,
  we deduce that $f^{-1}(B) \in \calA$, by local determination, showing
  that $f$ is measurable. Of course, the uniqueness of the gluing is
  given by Proposition~\ref{uniqueGluing}.
\end{proof}

\begin{Empty}
  \label{NoGluing}
  In this paragraph, we exhibit an lld MSN $(X,\calA, \calN)$, a
  measurable space $(Y, \calB)$ and, within this setting, a compatible
  family of measurable maps that cannot be glued. With regards to
  Proposition~\ref{Gluingcccc}, it is natural to turn towards
  Fremlin's example in \cite[\S{}216E]{FREMLIN.II} of a localizable,
  locally determined but not strictly localizable\footnote{In the
  context of measure spaces, we follow the terminology
  of~\cite{FREMLIN.II}: $(X, \calA, \mu)$ is strictly localizable
  whenever there is a measurable partition $\la X_i \ra_{i \in I}$
  such that a set $A \subset X$ is measurable whenever the sets $A
  \cap X_i$ are, and in that case $\mu(A) = \sum_{i \in I} \mu(A \cap
  X_i)$.} measure space $(X, \calA, \mu)$. Let us recall its
  construction. Fix a set $Y$ with cardinal greater than
  $\mathfrak{c}$ and we set $X \defeq \{0, 1\}^{\calP(Y)}$. For any $y
  \in Y$, we define $x_y \in X$ by
  \[
  \forall Z \in \calP(Y), \qquad x_y(Z) = \begin{cases}
    1 & \text{if } y \in Z \\
    0 & \text{if } y \not\in Z
  \end{cases}
  \]
  Let $\calK \subset \calP(\calP(Y))$ be the family of countable
  subsets of $\calP(Y)$. For any $K \in \calK$ and $y \in Y$, we
  define $F_{y, K} \defeq X \cap \{x : x(Z) = x_y(Z) \text{ for all }
  Z \in K\}$. Then we define, for all $y \in Y$,
  \[
  \calA_y = \calP(X) \cap \{A : \text{there is } K \in \calK
  \text{ such that } F_{y, K} \subset A \text{ or } F_{y, K}
  \subset X \setminus A\}
  \]
  Let us prove that $\calA_y$ is a $\sigma$-algebra. Clearly
  $\emptyset \in \calA_y$ and $\calA_y$ is closed under
  complementations. Let $\la A_n \ra_{n \in \N}$ be a sequence in
  $\calA_y$. Suppose there is some $n_0 \in \N$ and $K \in \calK$
  such that $F_{y, K} \subset A_{n_0}$. Then $F_{y, K}
  \subset \bigcup_{n \in \N} A_n$, which implies $\bigcup_{n \in \N}
  A_n \in \calA_y$. Suppose on the contrary that for all $n \in
  \N$, there is $K_n \in \calK$ such that $F_{y, K_n} \subset X
  \setminus A_n$. Then $\bigcap_{n \in \N} F_{y, K_n} =
  F_{y, \bigcup_{n \in \N} K_n} \subset X \setminus \bigcup_{n
    \in \N} A_n$ which also gives that $\bigcup_{n \in \N} A_n \in
  \calA_y$.
  
  Finally set $\calA \defeq \bigcap_{y \in Y} \calA_y$ and
  define the measure $\mu \colon \calA \to [0, \infty]$ by
  \[
  \forall A \in \calA, \qquad \mu(A) = \rmcard(Y \cap \{ y :
  x_y \in A\})
  \]
  For the rest of the discussion, we admit that $(X, \calA, \mu)$ is
  complete, localizable, locally determined and not strictly
  localizable. The proof of the latter relies on a non trivial fact in
  infinitary combinatorics; we refer to \cite[216E(f)(g)]{FREMLIN.II}
  for more details. The associated MSN $(X, \calA, \calN_\mu)$ is
  saturated, localizable, and it is locally determined, by
  Proposition~\ref{elemld}(E).

  Define $E_y = X \cap \{x : x(\{y\}) = 1\}$ for all $y \in Y$. This
  set is $\calA$ measurable, because $F_{y, \{\{y\}\}} = E_y$ (hence
  $E_y \in \calA_y$), and for any $z \in Y \setminus \{y\}$, we have
  $F_{z, \{\{y\}\}} = X \setminus E_y$ (hence $E_y \in
  \calA_{z}$). Note that $Y \cap \{z : x_z \in E_y\} = \{y\}$.

  We now choose $\calB$ to be the countable cocountable
  $\sigma$-algebra of $Y$. For any $y \in Y$, we define the measurable
  map $f_y \colon (E_y, \calA_{E_y}) \to (Y, \calB)$ that is constant
  equal to $y$. We claim that $\la f_y \ra_{y \in Y}$ is a
  compatible family of measurable maps subordinated to $\la E_y \ra_{y
    \in Y}$. This ensues from the fact that $E_y \cap E_{z} \in
  \calN_{\mu}$ for any distinct $y, z \in Y$. Assume by contradiction
  that we can find a gluing $f \colon X \to Y$. We will use the
  decomposition $\la f^{-1}(\{y\})\ra_{y \in Y}$ to show $(X, \calA,
  \mu)$ is strictly localizable.

  Let $A \subset X$ such that $A \cap f^{-1}(\{y\}) \in \calA$
  for all $y$. We want to show that $A \in \calA$. For $y
  \in Y$, we have
  \begin{itemize}
  \item Case $x_y \in A$: as $A \cap f^{-1}(\{y\})\in
    \calA_y$ and $x_y \in A \cap f^{-1}(\{y\})$, there
    is $K \in \calK$ such that $F_{y,K} \subset A \cap
    f^{-1}(\{y\}) \subset A$ (because $x_y \in F_{y,K}$, this is the branch of the dichotomy, in the definition of $\calA_y$, that occurs). Therefore $A \in \calA_y$.
  \item Case $x_y \not\in A$: since $x_y \not\in A \cap
    f^{-1}(\{y\}) \in \calA_y$, we can find $K \in \calK$
    such that $F_{y, K} \subset X \setminus (A \cap
    f^{-1}(\{y\}))$.  We deduce that $F_{y, K} \cap
    f^{-1}(\{y\}) \subset X \setminus A$. But $x_y \in
    f^{-1}(\{y\}) \in \calA_y$, therefore there is $K' \in
    \calK$ such that $F_{y, K'} \subset
    f^{-1}(\{y\})$. Whence $F_{y, K \cup K'} = F_{y, K}
    \cap F_{y, K'} \subset F_{y, K} \cap f^{-1}(\{y\})
    \subset X \setminus A$. It follows that $A \in \calA_y$.
  \end{itemize}
  In any case, we have shown that $A \in \calA_y$. As $y \in Y$ is
  arbitrary, $A \in \calA$.

  Now, one observes that the only $z \in Y$ such that $x_z \in
  f^{-1}(\{y\})$ is $y$. Therefore $\mu(A \cap f^{-1}(\{y\}))$ equals
  $1$ if $x_y \in A$ and $0$ otherwise. In consequence, we have
  $\mu(A) = \sum_{y \in Y} \mu(A \cap f^{-1}(\{y\}))$ as desired.
\end{Empty}

\section{Existence of cccc and lld versions}
\label{sec:slv}

%

\begin{Theorem}
  \label{thm61}
  Let $(X, \calA, \calN)$ be a saturated MSN and $\calE \subset \calA
  \setminus \calN$. We suppose that
  \begin{enumerate}
  \item[(1)] $(Z, \calA_Z, \calN_Z)$ is cccc for every $Z \in \calC$;
  \item[(2)] $\calE$ is almost disjointed;
  \item[(3)] $\calE$ is $\calN$-generating.
  \end{enumerate}
  Then the pair consisting of the MSN
  \[
  (\hat{X}, \hat{\calA}, \hat{\calN}) = \coprod_{Z \in \calE} (Z,
  \calA_Z, \calN_Z)
  \]
  and the morphism $\bp = \coprod_{Z \in \calE} \boldsymbol{\iota}_Z$
  is the cccc version of $(X, \calA, \calN)$ (as usual
  $\boldsymbol{\iota}_Z$ is the morphism induced by the inclusion map
  $\iota_Z \colon Z \to X$).
\end{Theorem}

\begin{proof}
  %
  The MSN $(\hat{X}, \hat{\calA}, \hat{\calN})$ is cccc as a coproduct
  of cccc MSNs (this is a general fact, in any category, a coproduct
  of coproducts is a coproduct,
  see~\cite[Proposition~2.2.3]{BORCEUX.1}), and $\bp$ is supremum
  preserving, according to \ref{propsp}(B) and~(D). Observe that each
  $Z \in \calE$ is also a subset of $\hat{X}$ and we denote by
  $\hat{\iota}_Z \colon Z \to \hat{X}$ the corresponding inclusion
  map.

  Let $(Y, \calB, \calM)$ be a cccc MSN and $\bq \colon (Y, \calB,
  \calM) \to (X, \calA, \calN)$ be a supremum preserving morphism,
  represented by $q \in \bq$. For all $Z \in \calE$, call $q_Z \defeq
  \hat{\iota}_Z \circ (q \hel q^{-1}(Z)) \colon q^{-1}(Z) \to
  \hat{X}$. Because $\bq$ is supremum preserving, $Y = q^{-1}(X)$ is
  an $\calM$-essential supremum of the family $\la q^{-1}(Z)\ra_{Z \in
    \calE}$. The family $\calE$ being almost disjointed and $q$ being
  measurable, $q^{-1}(Z) \cap q^{-1}(Z') = q^{-1}(Z \cap Z') \in
  \calM$ for any distinct $Z, Z' \in \calE$. As a result, the family
  $\la q_Z \ra_{Z \in \calE}$ subordinated to $\la q^{-1}(Z) \ra_{Z
    \in \calE}$ is compatible. By Proposition~\ref{Gluingcccc}, this family
  has a gluing $r \colon Y \to \hat{X}$ and, by Lemma~\ref{propGluing},
  $r$ is $[(\calB, \calM), (\hat{\calA}, \hat{\calN})]$-measurable and
  supremum preserving.  Call $p \defeq \coprod_{Z \in \calE}
  \iota_Z$. For each $Z \in \calE$, we have
  \[
  q^{-1}(Z) \cap \{p \circ r \neq q\} \subset q^{-1}(Z) \cap \{r \neq
  \hat{\iota}_Z \circ (q \hel q^{-1}(Z))\} \in \calM.
  \]
  The family $\{q^{-1}(Z) : Z \in \calE\}$ is $\calM$-generating and
  $(Y, \calB, \calM)$ is locally determined, so we conclude that $\{p
  \circ r \neq q\} \in \calM$, that is, $\bp \circ \br = \bq$.

  As for uniqueness, let $\br$ be any morphism such that $\bp \circ
  \br = \bq$, and call $r \in \br$ one of its representatives. Fix $Z
  \in \calE$. Observe that $\hat{\iota}_Z (p(z)) = z$ for all $z \in
  Z$. For $\calM$-almost every $x \in q^{-1}(Z)$, we have $p(r(x)) =
  q(x)$ which implies that $r(x) \in Z \subset \hat{X}$. For such an
  $x$, we find that $r(x) = \hat{\iota}_Z (p(r(x))) =
  \hat{\iota}_Z(q(x)) = q_Z(x)$. Hence, $r$ is a gluing of the
  compatible family $\la q_Z \ra_{Z \in \calE}$ and we invoke the
  uniqueness part of Proposition~\ref{Gluingcccc} to conclude.
\end{proof}

\begin{Empty}
  Consider the following example, taken
  from~\cite[216D]{FREMLIN.II}. Let $X$ be a set of cardinality greater
  or equal than $\aleph_2$. For each $x,y \in X$, we define $H_y = X
  \times \{y\}$ and $V_x = \{x\} \times X$. Sets of this form are
  respectively called horizontal and vertical lines. We define a
  $\sigma$-algebra $\calA$ of $X^2$ by declaring that $A \in \calA$
  iff for all $x, y \in X$, the trace $A \cap H_y$ (resp. $A \cap
  V_x$) is either countable or cocountable in $H_y$
  (resp. $V_x$). Also, we define the $\sigma$-ideal $\calN$ of
  $\calA$ as follows: $N \in \calN$ if and only if the intersection of $N$ with any
  line is countable. Clearly, $(X^2, \calA, \calN)$ is saturated.

  We assert that it is not localizable. Suppose if possible that the
  family of horizontal lines $\{H_y : y \in X\}$ has an
  $\calN$-essential supremum $S$. Then for all $y \in X$, the
  intersection $S \cap H_y$ is cocountable in $H_y$, that is, $N_y
  \defeq X \cap \{x : (x, y) \not\in S\}$ is countable. Let $Z$ be a
  subset of $X$ of cardinality $\aleph_1$. Then $\rmcard \bigcup_{y
    \in X} N_y \leq \aleph_1$, hence the existence of $x \in X
  \setminus \bigcup_{y \in X} N_y$.  This means that $V_x \subset S$.
  However, $S \setminus V_x$ is easily checked to be an
  $\calN$-essential upper bound of $\{H_y : y \in X\}$, as $H_y \cap
  V_x$ is negligible for all $y$. Since $V_x = S \setminus (S
  \setminus V_x) \not\in \calN$, we get a contradiction.

  The family of all lines $\calE \defeq \{H_y : y \in X\} \cup \{V_x :
  x \in X\}$ satisfies the three hypotheses of
  Theorem~\ref{thm61}. Applying the theorem, we see that the cccc
  version of $(X^2, \calA, \calN)$ can be described as the coproduct
  of all lines. Doing so, we see that each point $(x, y)$ in the base
  MSN $(X^2, \calA, \calN)$ is duplicated in the cccc version: the
  ``fibers'' $p^{-1}(\{(x,y)\})$ contains two elements, which
  represent the horizontal and vertical directions emanating from the
  point $(x, y)$.

  If a given MSN has no obvious choice of a family satisfying the
  conditions of Theorem~\ref{thm61}, we can justify the existence of a
  cccc version in a non constructive way.
\end{Empty}

\begin{Lemma}
  \label{lemma72}
  Let $(X, \calA, \calN)$ be a saturated MSN and $\calC \subset A$ an
  $\calN$-generating collection such that $(Z, \calA_Z, \calN_Z)$ is
  ccc for all $Z \in \calC$. Then we can find a collection $\calE
  \subset \calA \setminus \calN$ that satisfies conditions~(1), (2)
  and~(3) of Theorem~\ref{thm61} and such that each of its members is
  a subset of some member of $\calC$. Moreover, we can suppose
  $\rmcard \calE \leq \max\{\aleph_0, \rmcard \calC\}$.
\end{Lemma}

\begin{proof}
Let $\calE$ be associated with $\calC$ in Lemma \ref{zorn}. It clearly satisfies conditions (1), (2), and (3)
  of~\ref{thm61}, since a subMSN of a ccc MSN is ccc as well. If $\calC$ is infinite, then for all $Z \in
  \calC$, call $\calE_Z \defeq \calE \cap \{Z' : Z' \subset Z\}$. Then,
  each $\calE_Z$ is at most countable, since it is an almost
  disjointed family in the ccc MSN $(Z, \calA_Z, \calN_Z)$. As $\calE
  = \bigcup_{Z \in \calC} \calE_Z$, we conclude that $\rmcard \calE
  \leq \rmcard \calC$.
\end{proof}

\begin{Corollary}
  \label{existsccccv}
  Every saturated locally ccc MSN admits a cccc version.
\end{Corollary}

\begin{proof}
  Apply Lemma~\ref{lemma72} to the family $\calC \defeq \calA \cap \{Z
  : (Z, \calA_Z, \calN_Z) \text{ is ccc}\}$ and then
  Theorem~\ref{thm61}.
\end{proof}

\begin{Empty}
  \label{slversion}
  It is worth noticing that all the arguments contained in
  Theorem~\ref{thm61} and Corollary~\ref{existsccccv} remain valid
  provided we replace ``ccc'' by ``strictly localizable'', ``locally
  ccc'' by ``locally strictly localizable'', and ``cccc'' by ``strictly
  localizable''. Summing up, a saturated locally strictly localizable
  MSN $(X, \calA, \calN)$ has a strictly localizable version, which is
  constructed as the coproduct of subMSNs whose underlying sets
  belongs to a family $\calE$ that satisfies hypotheses (2), (3) of
  Theorem~\ref{thm61} and
  \begin{enumerate}
  \item[(1')] $(Z, \calA_Z, \calN_Z)$ is strictly localizable for
    every $Z \in \calC$.
  \end{enumerate}
  Since (1') implies (1) we can apply Theorem~\ref{thm61} again to
  conclude that the cccc and strictly localizable versions of $(X,
  \calA, \calN)$ are the same.

  As for the existence of lld versions, we have a partial result,
  which applies for most locally ccc MSNs that one is likely to
  encounter in analysis.
\end{Empty}

\begin{Theorem}
  \label{existslld}
  Let $(X, \calA, \calN)$ be a saturated MSN with a collection $\calC
  \subset \calA$ such that
  \begin{enumerate}
  \item[(1)] $(Z, \calA_Z, \calN_Z)$ is ccc for all $Z \in \calC$;
  \item[(2)] $\calC$ is $\calN$-generating;
  \item[(3)] $\rmcard \calC \leq \mathfrak{c}$.
  \end{enumerate}
  The following hold.
  \begin{enumerate}
  \item[(A)] If $(X, \calA, \calN)$ has an lld version, then it
    coincides with the cccc version.
  \item[(B)] If moreover $(Z, \calA_Z)$ is countably separated for all
    $Z \in \calC$, then the lld version exists.
  \end{enumerate}
\end{Theorem}

\begin{proof}
  (A) Recall $(X,\calA,\calN)$ has a cccc version, according to Corollary 
\ref{existsccccv}. Suppose $(X, \calA, \calN)$ has an lld version $[(\hat{X},
  \hat{\calA}, \hat{\calN}), \bp]$. In view of Proposition \ref{elemld}(E), conclusion (A) will be established if we prove that $(\hat{X},
  \hat{\calA}, \hat{\calN})$ is cccc. To this end, we need to find a
  suitable decomposition in $\hat{X}$. Apply Lemma~\ref{lemma72} to
  get an almost disjointed $\calN$-generating family $\calE$ such that
  $\rmcard \calE \leq \mathfrak{c}$ and $(Z, \calA_Z, \calN_Z)$ is ccc
  for all $Z \in \calE$. Choose an injection $c : \calE \to \R : Z \mapsto c_Z$ and $p \in \bp$. Let $f_Z
  \colon p^{-1}(Z) \to \R$ be the constant map equal to $c_Z$; it is readily $(\hat{\calA}_{p^{-1}(Z)},\calB(\R))$-measurable. The family $\la f_Z 
\ra_{Z \in \calE}$ is obviously compatible, since $\calE$ is almost disjointed. As $(\hat{X}, \hat{\calA},
  \hat{\calN})$ is localizable, this family has an $(\hat{\calA},
  \calB(\R))$-measurable gluing $f \colon \hat{X} \to \R$.

  We now show that $\la f^{-1}\{c_Z\} \ra_{Z \in \calE}$ is a
  partition of $\hat{X}$ into ccc measurable pieces.  Since $c$ is injective, the family $\la f^{-1}\{c_Z\} \ra_{Z \in \calE}$ is, indeed, a partition of $\hat{X}$ and, since $f$ is $(\hat{\calA},\calB(\R))$-measurable, 
$f^{-1}\{c_Z\} \in \hat{\calA}$ for all $Z \in \calE$. As $f$ is a gluing 
of $\la f_Z \ra_{Z \in \calE}$, we
  have $p^{-1}(Z) \setminus
  f^{-1}\{c_Z\} = p^{-1}(Z) \cap \{f \neq f_Z\}  \in \hat{\calN}$ for all $Z \in \calE$. Moreover, since $c$ is injective, for
  all $Z' \in \calE$ distinct from $Z$, one has $p^{-1}(Z') \cap
  (f^{-1}\{c_Z\} \setminus p^{-1}(Z)) \subset p^{-1}(Z') \cap \{f \neq
  f_{Z'}\} \in \hat{\calN}$. Also, $p^{-1}(Z) \cap (f^{-1}\{c_Z\}
  \setminus p^{-1}(Z)) = \emptyset$ is clearly negligible. Recalling that $(\hat{X},\hat{\calA},\hat{\calN})$ is saturated and that $p^{-1}(\calE)$ is $\hat{\calN}$-generating (because $\calE$ is $\calN$-generating and $p$ is supremum preserving), one infers from the
  Distributivity Lemma \ref{supDistrib} that $f^{-1}\{c_Z\} \setminus
  p^{-1}(Z) \in \hat{\calN}$. Thus $p^{-1}(Z) \ominus f^{-1}\{c_Z\} \in
  \hat{\calN}$. As $\bp$ is a local isomorphism, according to Propositions \ref{localiso}(B) and \ref{elemld}(A), $(p^{-1}(Z),
  \hat{\calA}_{p^{-1}(Z)}, \hat{\calN}_{p^{-1}(Z)})$ is ccc. By what
  precedes, so is $(f^{-1}\{c_Z\}, \hat{\calA}_{f^{-1}\{c_Z\}},
  \hat{\calN}_{f^{-1}\{c_Z\}})$. Therefore, the MSN
  \[
  (Y,\calB,\calM) \defeq \coprod_{Z \in \calE}
  (f^{-1}\{c_Z\}, \hat{\calA}_{f^{-1}\{c_Z\}}, \hat{\calN}_{f^{-1}\{c_Z\}})
  \]
  is cccc, by definition. It remains to establish that $(\hat{X},\hat{\calA},\hat{\calN})$ and $(Y,\calB,\calM)$ are isomorphic in $\sfMSNsp$. This is a consequence of Proposition \ref{disjointed.coproduct} applied to the measurable partition $\calF = \{ f^{-1}\{c_Z\} : Z \in \calE \}$. Recalling that $p^{-1}(\calE)$ is $\hat{\calN}$-generating, it ensues from the preceding paragraph that so is $\calF$. Since $(\hat{X},\hat{\calA},\hat{\calN})$ is locally determined (whence, has locally determined negligible sets, recall \ref{def.ld}), $\calF$ satisfies hypotheses (1) and (2) 
of Proposition \ref{disjointed.coproduct}.
  
  (B) Apply Lemma~\ref{lemma72} to $\calC$ and let
  $\calE$ be the family thus obtained. By Theorem~\ref{thm61}, the MSN $(\hat{X}, \hat{\calA},
  \hat{\calN}) \defeq \coprod_{Z \in \calE} (Z, \calA_Z, \calN_Z)$ and
  the morphism $\bp$ induced by $p = \coprod_{Z \in \calE} \iota_Z$
  constitute the cccc version of $(X, \calA, \calN)$. Furthermore, $(\hat{X},
  \hat{\calA})$ is countably separated, by Proposition~\ref{prop68}. In
  order to prove that $[(\hat{X}, \hat{A}, \hat{\calN}), \bp]$ is an
  lld version, we need to adapt the end of the proof of~\ref{thm61}.

  Let $(Y, \calB, \calM)$ be an lld MSN and $\bq \colon (Y, \calB,
  \calM) \to (X, \calA, \calN)$ a supremum preserving morphism
  represented by $q \in \bq$. As before, we let $\hat{\iota}_Z \colon Z
  \to \hat{X}$ be the inclusion map and $q_Z \defeq \hat{\iota}_Z
  \circ (q \hel q^{-1}(Z))$ for all $Z \in \calE$. The family $\la q_Z
  \ra_{Z \in \calE}$ subordinated to $\la q^{-1}(Z) \ra_{Z \in \calE}$
  is compatible. This time we use the gluing result~\ref{gluelld}
  instead, that provides a gluing $r \colon Y \to X$ of $\la q_Z
  \ra_{Z \in \calE}$. We argue as before to show that $r$ induces the
  unique supremum preserving morphism $\br \colon (Y, \calB, \calM)
  \to (X, \calA, \calN)$ such that $\bp \circ \br = \bq$.
\end{proof}

\section{Strictly localizable version of a measure space}

\begin{Lemma}
  \label{lemma72-clash!!!}
  Let $(X, \calA, \mu)$ be a measure space and $\calE \subset \calA$
  an $\calN_\mu$-generating collection that is closed under finite
  union. Then, for every $A \in \calA$, we have $\mu(A) = \sup \left\{
  \mu(A \cap Z) : Z \in \calE \right\}$.
\end{Lemma}

\begin{proof}
  If $\alpha \defeq \sup \{ \mu(A \cap Z) : Z \in \calE\}$ is
  infinite, there is nothing to prove. Otherwise, select an increasing
  sequence $\la Z_n \ra_{n \in \N}$ such that $\lim_n \mu(A \cap Z_n)
  = \alpha$. Set $A' \defeq A \cap \bigcup_{n \in \N} Z_n$. Suppose
  that $\mu((A \setminus A') \cap Z) > 0$ for some $Z \in \calE$. Then
  \[
  \alpha \geq \mu(A \cap (Z_n \cup Z)) \geq \mu(A \cap Z_n) +
  \mu((A \setminus A') \cap Z)
  \]
  Letting $n \to \infty$ gives a contradiction. So we conclude that
  $(A \setminus A') \cap Z$ is negligible for all $Z \in \calE$. With
  the help of Lemma~\ref{supDistrib}, we obtain that $\mu(A \setminus
  A') = 0$. Consequently, $\mu(A) = \mu(A') = \lim_{n \to \infty}
  \mu(A \cap Z_n) = \alpha$.
\end{proof}

\begin{Empty}[Pushforward of a measure by a morphism]
  Let $(X, \calA, \calN)$ be a saturated MSN. A measure $\mu \colon
  \calA \to [0, \infty]$ is {\em absolutely continuous with respect to
    $\calN$} whenever $N \in \calN$ implies $\mu(N) = 0$. Let $\bq
  \colon (X, \calA, \calN) \to (Y, \calB, \calM)$ a morphism of
  saturated MSNs. We define the {\em pushforward} measure $\bq_\# \mu
  \defeq q_\# \mu$, where $q$ is any representative of $\bq$. This
  definition makes sense, because, for all $q' \in \bq$ and $A \in
  \calA$, we have $\mu(q^{-1}(A) \ominus (q')^{-1}(A)) = 0$, owing to
  the absolute continuity of $\mu$. Trivially, $\bq_\#\mu$ is
  absolutely continuous with respect to $\calM$.
\end{Empty}

\begin{Empty}[Pre-image Measure]
  \label{para72}
  Let $(X, \calA, \mu)$ be a complete semi-finite measure space. To
  simplify the notations, we abbreviate $\calN_\mu$ to
  $\calN$. Following the discussion in Paragraph~\ref{locally}, $(X,
  \calA, \calN)$ is locally ccc. Recording
  Corollary~\ref{existsccccv}, $(X, \calA, \calN)$ has a cccc version
  $[(\hat{X}, \hat{\calA}, \hat{\calN}), \bp]$ and we shall show that
  there is a unique measure $\hat{\mu}$ on $(\hat{X}, \hat{\calA})$
  such that
  \begin{enumerate}
  \item[(1)] $\calN_{\hat{\mu}} = \hat{\calN}$;
  \item[(2)] $\bp_{\#} \hat{\mu} = \mu$.
  \end{enumerate}
  Such a measure $\hat{\mu}$ is referred to as the {\em pre-image measure} of
  $\mu$. Moreover, we will show that the measure space $(\hat{X},
  \hat{\calA}, \hat{\mu})$ is strictly localizable; we say that
  $[(\hat{X}, \hat{\calA}, \hat{\mu}), \bp]$ is the {\em strictly
    localizable version} of the measure space $(X,\calA, \mu)$.
  
  We start to prove the uniqueness of $\hat{\mu}$. Fix a
  representative $p \in \bp$. For any $F \in \calA$ we define $\hat{F}
  \defeq p^{-1}(F)$ and call $p_F\colon \hat{F} \to F$ the restriction
  of $p$, which induces, as usual, a morphism $\bp_F \colon (\hat{F},
  \hat{\calA}_{\hat{F}}, \hat{\calN}_{\hat{F}}) \to (F, \calA_F,
  \calN_F)$. Call $\calA^f \defeq \calA \cap \{F : \mu(F) <
  \infty\}$. A pre-image measure $\hat{\mu}$ must satisfy $\bp_{F_\#}(\hat{\mu}
  \hel \hat{F}) = \mu \hel F$ for every $F \in \calA^f$. But $(F,
  \calA_F, \calN_F)$ is ccc, so by Proposition~\ref{localiso}, $\bp_F$
  is an isomorphism, forcing $\hat{\mu} \hel \hat{F} =
  (\bp_{F}^{-1})_\# (\mu \hel F)$ to hold. Since $p$ is supremum
  preserving, the collection $\{\hat{F}: F \in \calA^f\}$ admits
  $\hat{X}$ as an $\hat{\calN}$ essential supremum. By~(1) and
  Lemma~\ref{lemma72-clash!!!} we infer that
  \[
  \hat{\mu}(A) = \sup \left\{ \hat{\mu}(A \cap \hat{F}) : F \in
  \calA^f\right\} = \sup \left\{ (\bp_{F}^{-1})_\#(\mu \hel F)(A \cap
  \hat{F}) : F \in \calA^f\right\}
  \]
  for all $A \in \hat{\calA}$, from which the uniqueness of the
  pre-image measure follows straightforwardly.
\end{Empty}

\begin{Empty}
  \label{para74}
  To deal with the existence of pre-image measures, we will fix a cccc version,
  obtained by an application Theorem~\ref{thm61} to the family
  $\calA^f$ defined above. As all cccc versions of $(X, \calA, \calN)$
  are isomorphic, there is no restriction in considering this special
  case.

  Henceforth we suppose that $(\hat{X}, \hat{\calA}, \hat{\calN}) =
  \coprod_{Z \in \calE} (Z, \calA_Z,\calN_Z)$, where $\calE \subset
  \calA^f \setminus \calN$ is a collection such that (A), (B) and
  (C) of Theorem~\ref{thm61} hold. We now define $\hat{\mu}$ on
  $(\hat{X}, \hat{\calA})$ by
  \[
  \hat{\mu}\left( \coprod_{Z \in \calE} A_Z\right) \defeq \sum_{Z \in
    \calE} \mu(A_Z)
  \]
  each $A_Z$ being an arbitrary measurable subset of $Z$. We choose
  the representative $p = \coprod_{Z \in \calE} \iota_Z$ of $\bp$,
  each $\iota_Z \colon Z \to X$ being the inclusion map.

  Obviously, $(\hat{X}, \hat{\calA}, \hat{\mu})$ is a strictly
  localizable measure space and $\calN_{\hat{\mu}} = \hat{\calN}$,
  which is condition~(1) of Paragraph~\ref{para72}. The next result
  gathers some facts about the measure $\hat{\mu}$. In particular,
  condition~(2) of Paragraph~\ref{para72} is proven in
  Proposition~\ref{prophatmu}(B).
\end{Empty}

\begin{Proposition}
  \label{prophatmu}
  With the notations of paragraph~\ref{para74}:
  \begin{enumerate}
  \item[(A)] For all $A \in \calA$, one has $\mu(A) = \sum_{Z \in
    \calE} \mu(A \cap Z)$.
  \item[(B)] $\bp_\# \hat{\mu} = \mu$;
  \item[(C)] For every set $A \in \hat{\calA}$ with $\sigma$-finite
    $\hat{\mu}$-measure, there is $B \in \calA$ with $\sigma$-finite
    $\mu$-measure such that $\hat{\mu}(A \ominus \hat{B}) = 0$, where
    $\hat{B} \defeq p^{-1}(B)$.
  \end{enumerate}
\end{Proposition}

\begin{proof}
  (A) When $\calE \cap \{Z : \mu(A \cap Z) > 0\}$ is uncountable, the
  result follows easily, for there is $\alpha > 0$ such that
  $\calE_\alpha \defeq \calE \cap \{Z : \mu(A \cap Z) > \alpha\}$ is
  infinite. Taking a countable subset $\calE_\alpha' \subset
  \calE_\alpha$, then
  \[
  \mu(A) \geq \mu\left(A \cap \bigcup \calE_\alpha'\right) = \sum_{Z
    \in \calE_\alpha'} \mu(A \cap Z) = \infty
  \]
  because $\calE_\alpha'$ is almost disjointed.

  On the other hand, suppose $\calE' \defeq \calE \cap \{Z : \mu(A
  \cap Z) > 0\}$ is countable and set $A' \defeq A \setminus \bigcup
  \calE'$. Then $\mu(A' \cap Z) = 0$ for every $Z \in \calE$. By
  Lemma~\ref{supDistrib}, $A'$ is an $\calN$ essential supremum of
  $\{A' \cap Z : Z \in \calE\}$, which forces $A'$ to be
  negligible. Consequently,
  \[
  \mu(A) = \mu\left(A \cap \bigcup \calE'\right) = \sum_{Z \in \calE'}
  \mu(A \cap Z) = \sum_{Z \in \calE} \mu(A \cap Z)
  \]

  (B) For any $A \in \calA$, one has
  \[
  \bp_\#\hat{\mu}(A) = \hat{\mu}(p^{-1}(A)) =
  \hat{\mu}\left(\coprod_{Z \in \calE} A \cap Z \right) = \sum_{Z \in
    \calE} \mu(A \cap Z)
  \]
  We conclude by means of (A).
  
  (C) Let $A \in \hat{\calA}$ a set of $\sigma$-finite $\hat{\mu}$
  measure. Writing $A = \coprod_{Z \in \calE} A_Z$, each $A_Z$ being a
  measurable subset of $Z$, the set $\calE' \defeq \calE \cap \{Z :
  \mu(A_Z) > 0\}$ must be countable. Define $B \defeq \bigcup \{A_Z :
  Z \in \calE' \}$. We claim that $A \ominus \hat{B} = \coprod_{Z
    \in \calE} \left( A_Z \ominus (B \cap Z)\right)$ is negligible,
  or, equivalently, all $A_Z \ominus (B \cap Z)$ are
  negligible. Indeed, for $Z \in \calE'$, one has $A_Z \ominus (B \cap
  Z) \subset \bigcup \{ A_{Z'} \cap Z :Z' \in \calE' \text{ and } Z'
  \neq Z\} \in \calN$. If $Z \not\in \calE'$, then both $A_Z$ and $B
  \cap Z$ are negligible.
\end{proof}

\begin{Proposition}
  \label{prop74}
  The Banach spaces $\bL_1(X, \calA, \mu)$ and $\bL_1(\hat{X},
  \hat{\calA}, \hat{\mu})$ are isometrically isomorphic.
\end{Proposition}

\begin{proof}
  For any $f,f' \in \mathbf{f} \in \bL_1(X, \calA, \mu)$ we check that $f\circ p$
  and $f' \circ p$ coincide almost everywhere. Thus, the linear map $\varphi \colon
  \bL_1(X, \calA, \mu) \to \bL_1(\hat{X}, \hat{\calA}, \hat{\mu})$ which
  assigns to $\mathbf{f}$ the equivalence class of $f \circ p$ is
  well-defined. Furthermore, we have
  \begin{align*}
  \int_X |f| d\mu & = \int_0^\infty \mu(\{ |f| > t\}) dt =
  \int_0^\infty \bp_\# \hat{\mu}(\{|f| > t\}) dt \\ & = \int_0^\infty
  \hat{\mu}(\{|f \circ p| > t\}) dt = \int_{\hat{X}} |f \circ p| d
  \hat{\mu},
  \end{align*}
  showing that $\varphi$ is an isometry.

  Let us show that $\varphi$ is onto. Let $\hat{f}$ be an integrable
  function on $(\hat{X}, \hat{\calA}, \hat{\mu})$. As $\{\hat{f}\neq
  0\}$ has $\sigma$-finite $\hat{\mu}$ measure,
  Proposition~\ref{prophatmu}(B) provides a set $B \in \calA$ of
  $\sigma$-finite $\mu$ measure such that $\hat{\mu}(\{\hat{f} \neq
  0\} \ominus \hat{B}) = 0$. But $(B, \calA_B, \calN_B)$ is strictly
  localizable, and by Proposition~\ref{localiso}, the morphism $\bp_B
  \colon (\hat{B}, \hat{\calA}_{\hat{B}}, \hat{\calN}_{\hat{B}}) \to
  (B, \calA_B, \calN_B)$ induced by the restriction $p_B \colon
  \hat{B} \to B$ of $p$ is an isomorphism. We choose $q_B \colon B \to
  \hat{B}$ a representative of $\bp_B^{-1}$ and define the map $f
  \colon X \to \R$ by $f(x) \defeq \hat{f}(q_B(x))$ for $x \in B$ and
  $f(x) \defeq 0$ otherwise. Finally, because $\{\hat{f} \neq f \circ
  p\} \subset (\hat{B} \cap \{x : q_B(p(x)) \neq x\}) \cup (\{\hat{f}
  \neq 0\} \setminus \hat{B})$, the maps $\hat{f}$ and $f \circ p$
  coincide almost everywhere.
\end{proof}

\begin{Corollary}[Dual of $\bL_1$]
  \label{dualL1}
  The dual of $\bL_1(X, \calA, \mu)$ is $\bL_\infty(\hat{X}, \hat{\calA},
  \hat{\mu})$.
\end{Corollary}

\begin{proof}
  It follows from Proposition~\ref{prop74} and the (strict)
  localizability of $(\hat{X}, \hat{\calA}, \hat{\mu})$, see {\it e.g.} \cite[243G]{FREMLIN.II}.
\end{proof}

\begin{Empty}[Semi-finite version]
\label{sfv}
We report on \cite[213X(c)]{FREMLIN.II}.
Let $(X,\calA,\mu)$ be a measure space. We define a measure $\check{\mu}$ 
on $\calA$ by the formula
\[
\check{\mu}(A) = \sup \{ \mu(A  \cap F) : F \in \calA^f \} ,
\]
$A \in \calA$. As usual, $\calA^f = \calA \cap \{ A : \mu(A) < \infty\}$. The following hold.
\begin{enumerate}
\item[(1)] $(X,\calA,\check{\mu})$ is semi-finite.
\item[(2)] If $A \in \calA$ and $\mu \hel A$ is $\sigma$-finite, then $\mu \hel A = \check{\mu} \hel A$.
\item[(3)] If $A \in \calA$ and $\check{\mu}(A) < \infty$, then there are 
$F \in \calA^f$ and $N \in \calN_{\check{\mu}}$ such that $A = F \cup N$. 
\item[(4)] The Banach space $\bL_1(X,\calA,\mu)$ and $\bL_1(X,\calA,\check{\mu})$ are isometrically isomorphic.
\end{enumerate}
These all straightforwardly follow from the definition. 
\par 
If we let $(X,\tilde{A},\tilde{\mu})$ be the completion of $(X,\calA,\check{\mu})$, it follows from (4) that $\bL_1(X,\calA,\mu)$ is isometrically 
isomorphic to $\bL_1(X,\tilde{A},\tilde{\mu})$ and, in turn, to $\bL_1(\hat{X},\hat{\calA},\hat{\mu})$, according to Proposition \ref{prop74}. In other words, we have associated with {\em each} measure space $(X,\calA,\mu)$ a strictly localizable ``version'', and we have identified the dual of $\bL_1(X,\calA,\mu)$. However, reference to Zorn's Lemma in Section \ref{sec:slv} (by means of Lemma \ref{zorn}) makes it difficult to understand the corresponding space $\hat{X}$. This is why we determine $\hat{X}$ explicitly in Sections \ref{sec:pu} and \ref{sec:appl}, in some special cases of interest. 
\end{Empty}

\section{A directional Radon-Nikod\'ym theorem}
\label{sec:RN}

In this section, we prove an extension of the Radon-Nikod\'ym theorem
for measure spaces that are not necessarily localizable, in connection
with the duality outlined in Corollary~\ref{dualL1}. So to speak, it
involves a generalized Radon-Nikod\'ym density that also depends on
the direction: as a function, it is defined on the strictly
localizable version.

This result is a slight extension of the Radon-Nikod\'ym theorem that
was discovered independently by McShane \cite[Theorem 7.1]{MCS.62} and
Zaanen \cite{ZAA.61}. Using Fremlin's version of the Radon-Nikod\'ym
theorem \cite[232E]{FREMLIN.II} in the proof below instead of the
standard one (between measure spaces of finite measure), we are able
to weaken one of the hypotheses in~\cite{MCS.62} and ask~(2)
instead. But the main difference with~\cite{MCS.62} and~\cite{ZAA.61}
is in terms of formulation. In their work, the Radon-Nikod\'ym density
takes the form of a ``quasi-function'' or a ``cross-section'', a
notion that is very close to that of a compatible family of measurable
functions.

\begin{Theorem}
  Let $(X, \calA, \mu)$ be a complete semi-finite measure space and
  $\nu$ a semi-finite measure on $(X, \calA)$. We let $[(\hat{X},
    \hat{\calA}, \hat{\mu}),\bp]$ be the strictly localizable version
  of $(X, \calA, \mu)$. Suppose that
  \begin{enumerate}
  \item[(1)] $\nu$ is absolutely continuous with respect to $\mu$.
  \item[(2)] For all $A \in \calA$ such that $\nu(A) > 0$, there is an
    $\calA$-measurable subset $F \subset A$ such that $\mu(F) <
    \infty$ and $\nu(F) > 0$.
  \end{enumerate}
  Then there is a $\hat{\calA}$-measurable function $f \colon \hat{X}
  \to \R^+$, unique up to equality $\hat{\mu}$-almost everywhere, such
  that $\nu= \bp_\#(f \hat\mu)$.
\end{Theorem}

\begin{proof}
  We let $\calA^f \defeq \{F : \nu(F) < \infty\}$. We claim that this
  family is $\calN_\mu$-generating. Let $U \in \calA$ be an
  $\calN_\mu$-essential upper bound of $\calA^f$. Then $\mu(F
  \setminus U) = 0$ for all $F \in \calA^f$. By absolute continuity,
  it follows that $\nu((X \setminus U) \cap F) = \nu(F \setminus U) =
  0$ for all $F \in \calA^f$. However, $\calA^f$ is
  $\calN_\nu$-generating, by semi-finiteness of $\nu$, and a routine
  application of the distributivity lemma~\ref{supDistrib} shows that
  $\nu(X \setminus U) = 0$. Hence $X \setminus U \in \calA^f$ and
  $\mu(X \setminus U) = \mu((X \setminus U) \setminus U) = 0$.

  Now, the measure $\nu \hel F$ is truly continuous with respect to
  $\mu \hel F$, for $F \in \calA^f$. Indeed, the hypotheses
  of~\cite[232B(b)]{FREMLIN.II} are all satisfied. Thus we can apply
  Fremlin's version of the Radon-Nikod\'ym theorem. It says that $\nu
  \hel F$ has a Radon-Nikod\'ym density $g_F \colon F \to \R^+$ with
  respect to $\mu \hel F$. It is easy to show that, for any $F, F' \in
  \calA^f$, one has $\mu(F \cap F' \cap \{g_F \neq g_F'\}) = 0$. Hence
  $\la g_F \ra_{F \in \calA^f}$ is a compatible family subordinated to
  $\calA^f$.

  Fix a representative $p \in \bp$ and set $\hat{F} \defeq p^{-1}(F)$
  and $f_F \defeq g_F \circ p_F$ for each $F \in \calA^f$.  where $p_F
  \colon\hat{F} \to F$ is the restriction of $p$. We claim that $\la
  f_F \ra_{F \in \calA^f}$ is a compatible family subordinated to $\la
  \hat{F} \ra_{F \in \calA^f}$. Indeed, for distinct $F, F' \in
  \calA^f$, we have $\hat{F} \cap \hat{F}' \cap \{f_F \neq f_{F'}\} =
  p^{-1}(F \cap F' \cap \{g_F \neq g_{F'}\})$. Since $p$ is
  $[(\hat{\calA}, \calN_{\hat\mu}), (\calA, \calN_\mu)]$-measurable,
  $\hat{F} \cap \hat{F}' \cap \{f_F \neq f_{F'}\} \in
  \calN_{\hat\mu}$. Owing to the supremum preserving character of $p$,
  the family $\{\hat{F} : F \in \calA^f\}$ is
  $\calN_{\hat\mu}$-generating. By Proposition~\ref{Gluingcccc}, the
  family $\la f_F \ra_{F \in \calA^f}$ has a gluing $f \colon \hat{X}
  \to \R^+$. For every $A \in \calA$ and $F \in \calA^f$, we have
  \begin{align*}
    \nu(A \cap F) & = \int \ind_{A \cap F} g_F d\mu & \text{Radon-Nikod\'ym Theorem} \\
    & = \int \ind_{A \cap F} g_F  d\bp_\# \hat\mu & \hat\mu \text{ is a 
pre-image measure}\\
    & = \int \ind_{p^{-1}(A \cap F)} f_F d \hat\mu \\
    & = \int_{p^{-1}(A) \cap \hat{F}} f d \hat\mu & f = f_F \text{ a.e on } \hat{F} \\
    & = (f \hat\mu)(p^{-1}(A) \cap \hat{F})
  \end{align*}
  Applying Lemma~\ref{lemma72-clash!!!}, we obtain $\nu(A) = \sup
  \{\nu(A \cap F) : F \in \calA^f\}$. Also, if we set $Z \defeq
  \hat{X} \cap \{x : f(x) > 0\}$, then in the subMSN $(Z,
  \hat{\calA}_Z, (\calN_{\hat\mu})_Z)$ the family $\{Z \cap \hat{F} :
  F \in \calA^f\}$ admits $Z$ as an $(\calN_{\hat\mu})_Z$-essential
  supremum, because $p \circ \iota_Z$ is supremum preserving ($\iota_Z
  \colon Z \to \hat{X}$, being the inclusion map, is supremum
  preserving, and the composition of supremum preserving maps is
  supremum preserving). Since $(\calN_{\hat\mu})_Z = \calN_{f\hat\mu
    \shel Z}$, we can apply Lemma~\ref{lemma72-clash!!!} again and
  deduce
  \begin{align*}
  (f\hat\mu)(p^{-1}(A)) & = (f\hat\mu \hel Z)(p^{-1}(A) \cap Z) \\ & = 
\sup
    \left\{ (f\hat\mu \hel Z)(p^{-1}(A) \cap Z \cap \hat{F}) : F \in
    \calA^f\right\} \\ & = \sup \left\{ (f\hat\mu)(p^{-1}(A) \cap
    \hat{F}) : F \in \calA^f \right\}
  \end{align*}
  Hence $\nu(A) = \bp_\#(f \hat\mu)(A)$.

  Now we prove the uniqueness of $f$. Let $f'$ be another density, and
  suppose $\hat{\mu}(\{f' > f\}) > 0$. By semi-finiteness of $\hat\mu$
  there is a set $A \in \hat{\calA}$ such that $A \subset \{f' > f\}$
  and $0 < \hat{\mu}(A) < \infty$. By Proposition~\ref{prophatmu}(C),
  there is $B \in \calA$ such that $\hat{\mu}(A \ominus p^{-1}(B)) =
  0$. However, we have $\bp_\#(f' \hat\mu)(B) = (f' \hat\mu)(A) > (f
  \hat\mu)(A) = \bp_\#(f \hat\mu)(B)$, which is a contradiction. It
  follows that $f' \leq f$ almost everywhere. Similarly, we prove the
  reverse inequality.
\end{proof}

\section{Cccc version deduced from a compatible family of lower densities}
\label{sec:pu}

We devote this section to an explicit construction of the cccc and lld
version under some extra assumptions. It will be applied in the next
section.

\begin{Empty}
  \label{ldens}
  Let $(X, \calA, \calN)$ be an MSN. A {\em lower density} for $(X,
  \calA, \calN)$ is a function $\Theta \colon \calA \to \calA$ such
  that
  \begin{enumerate}
  \item[(1)] $\Theta(A) = \Theta(B)$ for all $A, B \in \calA$ such
    that $A \ominus B \in \calN$;
  \item[(2)] $A \ominus \Theta(A) \in \calN$ for all $A \in \calA$;
  \item[(3)] $\Theta(\emptyset) = \emptyset$;
  \item[(4)] $\Theta(A \cap B) = \Theta(A) \cap \Theta(B)$ for all $A,
    B \in \calA$.
  \end{enumerate}
\end{Empty}

\begin{Proposition}
  \label{ldens=>cccc}
  Let $(X, \calA,\calN)$ be a saturated MSN, $\calE \subset \calA$, and $\Theta \colon \calA \to \calA$ a lower
  density. Assume that
  \begin{enumerate}
  \item[(A)] for all $Z \in \calE$, the subMSN $(Z, \calA_Z, \calN_Z)$ is 

ccc;
  \item[(B)] $\calE$ is $\calN$-generating;
  \item[(C)] One has
  \begin{enumerate}
  \item[(i)] $\forall A \subset X : \big[ \forall Z \in \calE : A \cap Z \in \calA \big] \Rightarrow A \in \calA$;
  \item[(ii)] $\forall N \subset X : \big[ \forall Z \in \calE : N \cap Z 
\in \calN \big] \Rightarrow N \in \calN$.
  \end{enumerate}
  \end{enumerate}
  Then $(X, \calA, \calN)$ is cccc.
\end{Proposition}

\begin{proof}
 Let $\calE_1$ be associated with $\calE$ in Lemma \ref{zorn}. Thus, $\calE_1$ is almost disjointed and $\calN$-generating, and $(Z,\calA_Z,\calN_Z)$ is ccc for all $Z \in \calE_1$.
 
 We claim that $\calE$ may be replaced by $\calE_1$ in hypothesis (C).
  Let $A \in \calP(X)$ be such that $A \cap Z \in \calA$ for every $Z \in
  \calE_1$. Let $Z' \in \calE$. Define $\calZ \defeq \{Z \cap Z' : Z \in
  \calE_1 \text{ and } Z \cap Z' \not\in \calN\}$. Notice that $(Z', \calA_{Z'},
  \calN_{Z'})$ is ccc and $\calZ$ is almost disjointed. Thus $\calZ$
  is countable and $\bigcup \calZ$ is an $\calN$-essential supremum of
  $\calZ$.  Besides, by the Distributivity Lemma~\ref{supDistrib}, the family $\{Z \cap
  Z' : Z \in \calE_1\}$ admits $Z'$ as an $\calN$-essential
  supremum. This family differs from $\calZ$ only by negligible
  sets. Therefore, $Z' \ominus \bigcup \calZ \in \calN$. Since $(X,\calA,\calN)$ is saturated, we deduce that
  $(A \cap Z') \ominus \left(A \cap \bigcup \calZ\right) \in
  \calN$. Thus, one needs to
  establish that $A \cap \bigcup \calZ \in \calA$ in order to show that $A \cap
  Z' \in \calA$. This is readily done by observing that $A \cap
  \bigcup \calZ = \bigcup \{A \cap Z \cap Z' : Z \in \calE_1 \text{ and
  } Z \cap Z' \not\in \calN\}$ is a countable union of measurable
  sets. We just proved that $A \cap Z' \in \calA$ for all $Z' \in
  \calE$. By hypothesis~(C)(i), $A \in \calA$. Now assume that $A \cap Z \in \calN$, for each $Z \in \calE_1$, and let $Z'$ and $\calZ$ be as above. Since $Z'$ is an $\calN$-essential supremum of $\calZ$, it follows from Lemma \ref{supDistrib} that $A \cap Z'$ is an $\calN$-essential supremum of $\{ A \cap Z \cap Z' : Z \in \calE_1 \text{ and } Z \cap Z' \not \in 
\calN \}$. Therefore, $A \cap Z' \in \calN$. Since $Z'$ is arbitrary, it follows that $A \in \calN$, by hypothesis (C)(ii)

Next we define $\calE_2 = \{ \Theta(Z) : Z \in \calE_1 \}$. The family $\calE_2$ is disjointed, for $\Theta(Z) \cap \Theta(Z') = \Theta(Z \cap 
Z') =
  \emptyset$ for any distinct $Z, Z' \in \calE_1$, since $Z \cap Z' \in
  \calN$. We next claim that $\calE$ (or, for that matter, $\calE_1$) may 
be replaced by $\calE_2$ in hypothesis (C). Indeed, for every $A \subset X$ and every $Z \in \calE_1$, $(A \cap Z) \ominus (A \cap \Theta(Z)) \subset Z \ominus \Theta(Z) \in \calN$, therefore (i) $A \cap Z \in \calA$ if 
and only if $A \cap \Theta(Z) \in \calA$, and (ii) $A \cap Z \in \calN$ if and only if $A \cap \Theta(Z) \in \calN$, since $(X,\calA,\calN)$ is saturated. In particular, letting $N \defeq X \setminus \cup \calE_2$, we infer from (C)(ii) with $\calE$ replaced by $\calE_2$ that $N \in \calN$. Finally, the conclusion follows from Proposition \ref{disjointed.coproduct} applied to $\calF = \calE_2 \cup \{N\}$.
\end{proof}

\begin{Empty}
  \label{compatFamily}
  Let $(X, \calA, \calN)$ be a saturated MSN and $\calE \subset \calA$. A 
{\em compatible family of lower densities} is a
  family $\la \Theta_Z \ra_{Z \in \calE}$ such that
  \begin{enumerate}
  \item[(1)] For all $Z \in \calE$, the map $\Theta_Z \colon \calA_Z
    \to \calA_Z$ is a lower density for $(Z, \calA_Z, \calN_Z)$;
  \item[(2)] For all $Z, Z' \in \calE$ and $A \subset Z \cap Z'$ a
    measurable set, $\Theta_Z(A) = \Theta_{Z'}(A)$;
  \item[(3)] $\Theta_Z(Z) = Z$ for all $Z \in \calE$.
  \end{enumerate} 
  Condition~(3) is merely of technical nature. If a family satisfies only 

(1) and
  (2), we can enforce (3) by replacing $\calE$ with $\{\Theta_Z(Z) : Z
  \in \calE\}$ and observing that $\Theta_Z \colon \calA_Z \to
  \calA_Z$ restricts to $\calA_{\Theta_Z(Z)} \to \calA_{\Theta_Z(Z)}$. This, indeed, follows from the fact that $\Theta_Z(A) \subset \Theta_Z(B)$, 

whenever $A,B \in \calA_Z$ and $A \subset B$, and $\Theta_Z \circ \Theta_Z = \Theta_Z$, as one easily checks from the definition of lower density.
\end{Empty}

\begin{Empty}[Germ space]
  \label{para95}
  In the sequel, we consider a saturated MSN $(X, \calA, \calN)$ that
  has a compatible family of lower densities $\la \Theta_Z \ra_{Z \in
    \calE}$, where $\calE$ is a family such that
  \begin{enumerate}
  \item[(1)] $\calE$ is $\calN$-generating;
  \item[(2)] $(Z, \calA_Z, \calN_Z)$ is ccc for each $Z \in \calE$.
  \end{enumerate}
  Under these assumptions, we will now construct a new MSN $(\hat{X}, \hat{\calA}, \hat{\calN})$ that we call the {\em germ space of $(X,\calA,\calN)$ associated with $\calE$ and $\la \Theta_Z \ra_{Z \in \calE}$}.

  For every $x \in X$, we set $\calE_x \defeq \calE \cap \{Z : x \in
  Z\}$ and we define the relation $\sim_x$ on $\calE_x$ by $Z \sim_x
  Z' \iff x \in \Theta_Z(Z \cap Z')$. We claim that it is an
  equivalence relation. Indeed, it is reflexive because
  of~\ref{compatFamily}(3); it is symmetric because of the set
  equality $\Theta_Z(Z \cap Z') = \Theta_{Z'}(Z \cap Z')$ implied
  by~\ref{compatFamily}(2). Let us check that it is transitive. For
  $Z, Z', Z'' \in \calE_x$ such that $Z \sim_x Z' \sim_x Z''$, we have
  \begin{align*}
  x \in \Theta_{Z'}(Z \cap Z') \cap \Theta_{Z'}(Z' \cap Z'') & =
  \Theta_{Z'}(Z \cap Z' \cap Z'') & \text{\ref{ldens}(4)} \\
  & = \Theta_{Z}(Z \cap Z' \cap Z'') & \text{\ref{compatFamily}(2)} \\
  & \subset \Theta_{Z}(Z \cap  Z'')
  \end{align*}
  hence $Z \sim_x Z''$.

  We define the quotient set $\Gamma_x \defeq \calE_x /
  \sim_x$. The equivalence class of $Z \in \calE_x$ is
  denoted $[Z]_x \in \Gamma_x$.  Next we define the set $\hat{X}
  \defeq \{(x, [Z]_x) : x \in X \text{ and } [Z]_x \in \Gamma_x\}$ and
  the projection map $p \colon \hat{X} \to X$ which assigns $(x,
  \bZ)$ to $x$. For each $Z \in \calE$, we define the map $\gamma_Z
  \colon Z \to \hat{X}$ by $\gamma_Z(x) = (x, [Z]_x)$ for $x \in
  Z$. We define a $\sigma$-algebra $\hat{\calA}$ and a $\sigma$-ideal
  $\hat{\calN}$ on $\hat{X}$ by
  \begin{gather*}
    \hat{\calA} \defeq \calP(\hat{X}) \cap \left\{ A :  \gamma_Z^{-1}(A) \in \calA_Z \,, \forall Z \in \calE  \right\}\\
    \hat{\calN} \defeq \calP(\hat{X}) \cap \left\{ N :  \gamma_Z^{-1}(N) \in \calN_Z \,, \forall Z \in \calE  \right\}.
  \end{gather*}
  Actually, $\hat{\calA}$ and $\hat{\calN}$ are the finest
  $\sigma$-algebra and $\sigma$-ideal such that the maps $\gamma_Z$
  become $[(\calA_Z, \calN_Z), (\hat{\calA},
    \hat{\calN})]$-measurable.  Clearly, $(\hat{X}, \hat{\calA},
  \hat{\calN})$ is a saturated. Let us check that the projection map
  $p \colon \hat{X} \to X$ is $[(\hat{\calA}, \hat{\calN}), (\calA,
    \calN)]$-measurable. If $A \in \calA$ then for any $Z \in \calE$
  we have $\gamma_Z^{-1}(p^{-1}(A)) = (p \circ \gamma_Z)^{-1}(A) = Z
  \cap A \in \calA_Z$, so by definition $p^{-1}(A) \in
  \hat{\calA}$. One proves similarly that $p^{-1}(N) \in \hat{\calN}$
  for all $N \in \calN$.
\end{Empty}

%

\begin{Theorem}
\label{germ.theorem}
  Let $(X, \calA, \calN)$ be a saturated MSN that has a compatible
  family of lower density $\la \Theta_Z \ra_{Z \in \calE}$, where
  $\calE \subset \calA$ is a family such that conditions~(1) and~(2)
  of~\ref{para95} hold. Then the germ space $(\hat{X}, \hat{\calA},
  \hat{\calN})$ constructed in~\ref{para95} together with $\bp$ is the
  cccc version of $(X, \calA, \calN)$. It is also its lld version in case 
$\rmcard \calE \leq \mathfrak{c}$ and $(Z,\calA_Z)$ is countably separated for all $Z \in \calE$.
\end{Theorem}

\begin{proof}
The second conclusion is a consequence of the first and of Theorem \ref{existslld}.

  {\em Step 1: we prove that $(\hat{X}, \hat{\calA}, \hat{\calN})$
    possesses a lower density $\Theta$, obtained by ``patching
    together'' the lower densities $\Theta_Z$ for $Z \in \calE$}.
  For every $A \in \hat{\calA}$, we set
  \[
  \Theta(A) \defeq \hat{X} \cap \{(x, [Z]_x) : x \in
  \Theta_Z(\gamma_Z^{-1}(A))\}
  \]
  The condition $x \in \Theta_Z(\gamma_Z^{-1}(A))$ does not depend on the
  representative $Z$ of $[Z]_x$. Indeed, if $Z' \sim_x Z$ for some $Z'
  \in \calE_x$, then $x \in \Theta_Z(\gamma_Z^{-1}(A)) \cap \Theta_Z(Z \cap
  Z') = \Theta_Z(\gamma_Z^{-1}(A) \cap Z')$. Note that the sets
  $\gamma_Z^{-1}(A) \cap Z'$ and $\gamma_{Z'}^{-1}(A) \cap Z$ coincide $\calN$-almost everywhere, as
  \begin{align*}
  (\gamma_Z^{-1}(A) \cap Z') \ominus (\gamma_{Z'}^{-1}(A) \cap Z) & \subset Z
    \cap Z' \cap \{y : [Z]_y \neq [Z']_y\} \\ & \subset Z \cap Z'
    \setminus \Theta_Z(Z \cap Z')
  \end{align*}
  is negligible by~\ref{ldens}(2). Consequently,
  \begin{align*}
  \Theta_Z(\gamma_{Z}^{-1}(A) \cap Z') & = \Theta_{Z}(\gamma_{Z'}^{-1}(A) \cap
  Z) & \text{\ref{ldens}(2)} \\
  & = \Theta_{Z'}(\gamma_{Z'}^{-1}(A) \cap Z) & \text{\ref{compatFamily}(2)} \\
  & \subset \Theta_{Z'}(\gamma_{Z'}^{-1}(A))
  \end{align*}
  and in turn $x \in \Theta_{Z'}(\gamma_{Z'}^{-1}(A))$, as expected.
  
  Next we show that $\Theta$ satisfies the four properties required to be
  a lower density:
  \begin{itemize}
  \item Let $A, B \in \hat{\calA}$ such that $A \ominus B \in
    \hat{\calN}$. Then $\gamma_Z^{-1}(A) \ominus \gamma_Z^{-1}(B) \in \calN$ for
    all $Z \in \calE$, which implies
    \begin{align*}
      (x, [Z]_x) \in \Theta(A) & \iff x \in \Theta_Z(\gamma_Z^{-1}(A)) \\
      & \iff x \in \Theta_Z(\gamma_Z^{-1}(B)) & \text{\ref{ldens}(1)}\\
      & \iff (x, [Z]_x) \in \Theta(B)
    \end{align*}
    and we conclude that $\Theta(A) = \Theta(B)$.
  \item Let $A \in \hat{\calA}$. By construction, $\gamma_Z^{-1}(\Theta(A))
    = \Theta_Z(\gamma_Z^{-1}(A))$, for all $Z \in \calE$. This gives that
    $\gamma_Z^{-1}(A \ominus \Theta(A)) = \gamma_Z^{-1}(A) \ominus
    \gamma_Z^{-1}(\Theta(A)) \in \calN$. By definition of the
    $\sigma$-ideal $\hat{\calN}$, we infer that $A \ominus \Theta(A)
    \in \hat{\calN}$.
  \item That $\Theta(\emptyset) = \emptyset$ is straightforward.
  \item Let $A, B \in \hat{\calA}$. We have
    \begin{align*}
      (x, [Z]_x) \in \Theta(A \cap B) & \iff x \in \Theta_Z(\gamma_Z^{-1}(A \cap B)) \\
      & \iff x \in \Theta_Z(\gamma_Z^{-1}(A) \cap \gamma_Z^{-1}(B)) \\
      & \iff x \in \Theta_Z(\gamma_Z^{-1}(A)) \cap \Theta_Z(\gamma_Z^{-1}(B)) \\
      & \iff (x, [Z]_x) \in \Theta(A) \cap \Theta(B).
    \end{align*}
  \end{itemize}

  {\em Step 2: we establish that $\bp$ is a ``local isomorphism''}.
  Set $\hat{Z} \defeq p^{-1}(Z)$ for all $Z \in \calE$. We also call
  $p_Z$ and $s_Z$ the respective restrictions of $p$ and $\gamma_Z$ to
  $\hat{Z} \to Z$ and $Z \to \hat{Z}$. First, we remark that $p_Z
  \circ s_Z = \rmid_Z$.

  Let us show that $\hat{Z} \setminus s_Z(Z) \in \hat{\calN}$. For $Z'
  \in \calE$, we find that
  \begin{align*}
    x \in \gamma_{Z'}^{-1}(\hat{Z} \setminus s_Z(Z)) & \iff x \in
    Z' \text{ and } (x, [Z']_x) \in \hat{Z} \setminus s_Z(Z) \\ &
    \iff x \in Z \cap Z' \text{ and } [Z']_x \neq [Z]_x \\ & \iff x
    \in Z \cap Z' \setminus \Theta_{Z'}(Z \cap Z').
  \end{align*}
  So $\gamma_{Z'}^{-1}(\hat{Z} \setminus s_Z(Z)) \in \calN_{Z'}$. As
  this holds for all $Z' \in \calE$, we deduce that $\hat{Z} \setminus
  s_Z(Z)$ is $\hat{\calN}$-negligible. 

  Since $\hat{Z} \cap \{\xi : (s_Z \circ p_Z)(\xi) \neq \xi\} = \hat{Z}
  \setminus s_Z(Z)$, this shows that $s_Z \circ p_Z$ and
  $\rmid_{\hat{Z}}$ coincide $\hat{\calN}$-almost everywhere. As a
  consequence, the morphisms $\bp_Z$ and $\bs_Z$ induced by $p_Z$ and
  $s_Z$ are reciprocal isomorphisms of $\sfMSN$ between $(Z, \calA_Z, \calN_Z)$
  and $(\hat{Z}, \hat{\calA}_{\hat{Z}}, \hat{\calN}_{\hat{Z}})$. They are 
supremum preserving, according to Proposition \ref{propsp}(A).

  {\em Step 3: $(\hat{X}, \hat{\calA}, \hat{\calN})$ is ``locally
    determined'' (in the sense of Proposition \ref{ldens=>cccc}(C)) by the family $\hat{\calE} \defeq \{\hat{Z} : Z \in
    \calE\}$}. Let $A$ a subset of $\hat{X}$. By definition of
  $\hat{\calA}$, we have
  \begin{align*}
    A \in \hat{\calA} & \iff \forall Z \in \calE : \gamma_Z^{-1}(A)
    \in \calA_Z \\
    & \iff \forall Z \in \calE : s_Z^{-1}(A \cap \hat{Z}) \in \calA_Z \\
    & \iff \forall Z \in \calE : A \cap \hat{Z} \in \hat{\calA}_{\hat{Z}}.
  \end{align*}
  The direct implication of the last equivalence is justified as
  follows: if $s_Z^{-1}(A \cap \hat{Z})$ is measurable, then so is
  $p_Z^{-1}(s_Z^{-1}(A \cap \hat{Z}))$, from which $A \cap \hat{Z}$
  differs only by an $\hat{\calN}$ negligible set.
  We prove analogously that a set $N \subset \hat{X}$ is negligible if
  and only if $N \cap \hat{Z} \in \hat{\calN}$ for all $Z \in\calE$.

  {\em Step 4: $p$ is supremum preserving}. Let $\calF \subset \calA$
  be a collection which has an $\calN$-essential supremum denoted
  $S$. Clearly, $p^{-1}(S)$ is an $\hat{\calN}$-essential upper bound
  of $p^{-1}(\calF) \defeq \{p^{-1}(F) : F \in \calF\}$.  Let $U$ be
  an arbitrary $\hat{\calN}$-essential upper bound of $p^{-1}(\calF)$. We
  need to prove that $p^{-1}(S) \setminus U \in \hat{\calN}$, that is,
  $\gamma_Z^{-1}(p^{-1}(S) \setminus U) \in \hat{\calN}$ for all $Z
  \in \calE$. But $\gamma_Z^{-1}(p^{-1}(S) \setminus U) = (p \circ
  \gamma_Z)^{-1}(S) \setminus \gamma_Z^{-1}(U) = Z \cap S \setminus
  \gamma_Z^{-1}(U)$. By Lemma~\ref{supDistrib} we recognize $Z \cap S$
  as an $\calN$-essential supremum of $\{Z \cap F : F \in \calF\}$.
  This last collection can be also written
  $\{\gamma_Z^{-1}(p^{-1}(F)) : F \in \calF\}$, of which $\gamma_Z^{-1}(U)$ is an $\calN$-essential upper bound,
  leading to $Z \cap S \setminus \gamma_Z^{-1}(U) \in \calN$.

  {\em Step 5: $(\hat{X}, \hat{\calA}, \hat{\calN})$ is cccc}. This is
  an application of Proposition~\ref{ldens=>cccc} with the collection
  $\hat{\calE}$. We check that all the hypotheses are satisfied. For $Z
  \in \calE$, the subMSN $(\hat{Z}, \hat{\calA}_{\hat{Z}},
  \hat{\calN}_{\hat{Z}})$ is ccc because of the isomorphism $\bp_Z
  \colon (\hat{Z}, \hat{\calA}_{\hat{Z}}, \hat{\calN}_{\hat{Z}}) \to
  (Z, \calA_Z, \calN_Z)$. Since $p$ is supremum preserving, $\hat{X}$
  is an $\hat{\calN}$-essential supremum of $\hat{\calE}$. The ``local
  determination'' property was established in step 3.

  {\em Step 6: The pair $((\hat{X}, \hat{\calA}, \hat{\calN}), \bp)$
    satisfies the universal property of Definition~\ref{pversion}}. We
  finish the proof in a way similar to the proof of Theorem~\ref{thm61}. Let
  $(Y, \calB, \calM)$ a cccc MSN and $\bq \colon (Y, \calB, \calM) \to
  (X, \calA,\calN)$ a supremum preserving morphism, represented by a
  map $q \in \bq$. For every $Z \in \calE$, we define $q_Z = \gamma_Z
  \circ (q \hel q^{-1}(Z)) \colon q^{-1}(Z) \to \hat{X}$. We claim
  that $\la q_Z \ra_{Z \in \calE}$ is a compatible family subordinated
  to $\la q^{-1}(Z) \ra_{Z \in \calE}$. Indeed, for distinct $Z, Z' \in \calE$,
  \begin{align*}
    q^{-1}(Z) \cap q^{-1}(Z') \cap \{q_Z \neq q_{Z'}\} & = q^{-1}(Z
    \cap Z' \cap \{\gamma_Z \neq \gamma_{Z'}\}) \\ & = q^{-1}(Z \cap
    Z' \setminus \Theta_Z(Z \cap Z'))
  \end{align*}
  is negligible, using that $\Theta_Z$ is a lower density and $q$ is
  $[(\calB, \calM), (\calA, \calN)]$-measurable. Then, by
  Proposition~\ref{Gluingcccc}, the family $\la q_Z \ra_{Z \in \calE}$ has
  a gluing that we denote $r \colon Y \to \hat{X}$. That $r$ is
  $[(\calB, \calM), (\hat{\calA}, \hat{\calN})]$-measurable and
  supremum preserving follows from Lemma~\ref{propGluing}. Indeed, each $\gamma_Z$ is supremum preserving. This follows from the same property of $s_Z$, proved in Step 2, and the Distributivity Lemma \ref{supDistrib}.

  We need to show that $\{p \circ r \neq q\}$ is
  $\calM$-negligible. In fact, for any $Z \in \calE$ and $y \in
  q^{-1}(Z)$, we note that $p(q_Z(y)) = p(\gamma_Z(q(y)) = q(y)$, so
  $q^{-1}(Z) \cap \{p \circ r \neq q\} \subset \left(q^{-1}(Z) \cap
  \{r \neq q_Z\}\right)$ is $\calM$-negligible. We then use that $(Y,
  \calB, \calM)$ has locally determined negligible sets (see
  Proposition~\ref{elemld}(E) and the preceding
  Paragraph~\ref{def.ld}) to conclude that $p \circ r = q$ almost
  everywhere. We have found a supremum preserving morphism $\br \colon
  (Y, \calB, \calM) \to (\hat{X}, \hat{\calA}, \hat{\calN})$, namely
  the one induced from $r$, such that $\bp \circ \br = \bq$.

  We now prove that this factorization is unique. Let $\br$ be a
  supremum preserving morphism such that $\bp \circ \br = \bq$ and $r
  \in \br$. For $Z \in\calE$ and almost every $y \in q^{-1}(Z)$, we
  have $p(r(y)) = q(y) \in Z$. Therefore $r(y) \in \hat{Z}$ for almost
  all $y \in q^{-1}(Z)$. For such a $y$, we have $q(y) = p(r(y)) =
  p_Z(r(y))$, which implies that $q_Z(y) = \gamma_Z(q(y)) = s_Z(q(y))
  = s_Z(p_Z(r(y))$. But $s_Z \circ p_Z$ and $\rmid_{\hat{Z}}$ coincide
  almost everywhere on $\hat{Z}$ as we saw in Step~2. This implies
  that $r(y) = q_Z(y)$ for almost all $y \in q^{-1}(Z)$. The map $r$
  must be a gluing of $\la q_Z \ra_{Z \in \calE}$, so it is unique up
  to equality almost everywhere according to Proposition~\ref{Gluingcccc}.
\end{proof}

\section{Applications}
\label{sec:appl}

Here, we apply Theorem \ref{germ.theorem} to two different
situations. For this result to apply to an MSN $(X,\calA,\calN)$, the
following conditions need to be met.

\begin{enumerate}
\item[(i)] $(X,\calA,\calN)$ is saturated.
\item[(ii)] An $\calN$-generating family $\calE \subset \calA$ is
  given.
\item[(iii)] For every $Z \in \calE$, the MSN $(Z,\calA_Z,\calN_Z)$ is
  ccc.
\item[(iv)] For every $Z \in \calE$, the measurable space
  $(Z,\calA_Z)$ is countably separated.
\item[(v)] $\rmcard \calE \leq \mathfrak{c}$.
\item[(vi)] For every $Z \in \calE$, a lower density $\Theta_Z$ is
  given for $(Z,\calA_Z,\calN_Z)$, so that $\Theta_Z(Z)=Z$.
\item[(vii)] For every $Z,Z' \in \calE$ and $A \in \calA$ such that $A
  \subset Z \cap Z'$, one has $\Theta_Z(A) = \Theta_{Z'}(A)$.
\end{enumerate}

In that case, the corresponding germ space
$(\hat{X},\hat{\calA},\hat{\calN})$ constructed in \ref{para95} is the
cccc version and the lld version of $(X,\calA,\calN)$.

\begin{Empty}[Purely unrectifiable negligibles]
  \label{unrectifiablenull}
  Fix integers $1 \leq k \leq m-1$. Recall \cite[3.2.14]{GMT} that a
  subset $N \subset \R^m$ is called {\em purely $(\calH^k,
    k)$-unrectifiable} whenever $\calH^k(N \cap M) = 0$ for every
  $k$-rectifiable set $M \subset \R^m$. This is equivalent to
  $\calH^k(N \cap M)=0$ for every $k$-dimensional embedded submanifold
  $M \subset \Rm$ of class $C^1$ with $\calH^k(M) < \infty$, by
  \cite[3.1.15]{GMT}. We denote by $\calN_{\rmpu,k}$ the collection of
  purely $(\calH^k, k)$-unrectifiable subsets of $\R^m$. It is a
  $\sigma$-ideal of $\calP(\R^m)$. We also introduce the Borel
  $\sigma$-algebra $\calB(\R^m)$ of $\R^m$ and its completion
  $\overline{\calB(\R^m)} \defeq \{B \ominus N : B \in \calB(\R^m), N
  \in \calN_{\rmpu,k}\}$. We shall show that the MSN $(\R^m,
  \overline{\calB(\R^m)}, \calN_{\rmpu,k})$ can be associated with a
  germ space, as in \ref{para95}. We notice that, by definition, this
  MSN is saturated.  We let $\calE$ be the collection of all
  $k$-dimensional (embedded) submanifolds $M \subset \R^m$ of class
  $C^1$, \cite[3.1.19]{GMT}, such that $\calH^k \hel M$ is locally
  finite (that is $\calH^k(M \cap B) < \infty$ for every bounded Borel
  set $B \subset \Rm$). Clearly, each member of $\calE$ is Borel.
  \par (ii) We now show that $\calE$ is
  $\calN_{\rmpu,k}$-generating. Let $U \in \overline{\calB(\R^m)}$ be
  such that $\R^m \setminus U \not \in \calN_{\rmpu,k}$. By definition
  of this $\sigma$-ideal, there exists $M \in \calE$ such that
  $\calH^k((\R^m \setminus U)\cap M) > 0$. In other words, $M \setminus U
  \not \in \calN_{\rmpu,k}$, i.e. $U$ is not an
  $\calN_{\rmpu,k}$-essential upper bound of $\calE$.
  \par (iii) We next claim that
  $(M,\overline{\calB(\R^m)}_M,(\calN_{\rmpu,k})_M)$ is ccc, for every
  $M \in \calE$. To this end, we notice that for every $M \in \calE$
  the following holds:
\begin{equation}
\label{eq.pu.zero}
\text{For every }S \subset M : S \in \calN_{\rmpu,k} \text{ if and only if }\calH^k(S)=0.\tag{$\bigstar$}
\end{equation}
 In other words, $(M,\overline{\calB(\Rm)}_M,(\calN_{\rmpu,k})_M)$ is the 
saturation of the MSN associated with the measure space $(M,\calB(M),\calH^k \hel M)$. Since the latter is $\sigma$-finite, the claim follows from 
Proposition \ref{finite=>loc}. 
\par  
 We also record the following useful consequence of \eqref{eq.pu.zero}, for $M \in \calE$: 
\begin{equation}
\label{eq.pu.measurable}
\text{If } S \in \overline{\calB(\R^m)}_M \text{ then } S = B \ominus N 
\text{ for some } B  \in \calB(M) \text{ and } N \subset M \text{ with } \calH^k(N)=0 .\tag{$\blacklozenge$}
\end{equation} 
 Indeed, $S = B' \ominus N'$, $B' \in \calB(\R^m)$, $N' \in \calN_{\rmpu,k}$. Thus $S = M \cap S = (M \cap B') \ominus (M \cap N')$, which proves \eqref{eq.pu.measurable}. In particular, $S$ is $\calH^k$-measurable, even though some $S \in \overline{\calB(\R^m)}$ may not be $\calH^k$-measurable.
\par 
(iv) Let $M \in \calE$. We observe that the canonical embedding $(M,\overline{\calB(\R^m)}_M) \to (\Rm,\calB(\Rm))$ is, indeed, injective and measurable. Therefore, $(M,\overline{\calB(\Rm)}_M)$ is countably separated, according to Proposition \ref{caracCS}.
\par 
(v) Since $\calE \subset \calB(\Rm)$ we infer that $\rmcard \calE \leq \mathfrak{c}$, according to \cite[3.3.18]{SRIVASTAVA}.
\par 
(vi) In order to define lower densities, we recall \cite[2.10.19]{GMT} the density numbers $\Theta^k_*(\phi,x)$ and $\Theta^{*\,k}(\phi,x)$, defined by means of closed Euclidean balls, associated with an outer measure $\phi$ on $\Rm$ and $x \in \Rm$. Given $M \in \calE$ we abbreviate $\phi_M 
= \calH^k \hel M$ and we define 
\[ 
\Theta_M(A) = M \cap \{ x : \Theta^k_*(\phi_M,x) = 1 \},
\]
whenever $A \in \overline{\calB(\Rm)}_M$. Given $x \in \Rm$, the function 
$r \mapsto \phi_M(\bB(x,r))$ is right continuous, since $\phi_M$ is locally finite. It easily follows that $x \mapsto \Theta^k_*(\phi_M,x)$ is Borel measurable and, in turn, that $\Theta_M(A) \in \calB(\Rm)$. In particular, $\Theta_M$ maps $\overline{\calB(\Rm)}_M$ to itself. The following is the main point of the construction:
\begin{equation}
\label{eq.pu.density.1}
\text{For every } x \in M :   \Theta^m_*(\phi_M,x) = 1.\tag{$\clubsuit$}
\end{equation}
See for instance the proof of \cite[3.6.1]{DEP.05c}. For instance, it follows that $\Theta_M(M)=M$. We now turn to checking that $\Theta_M$ is a 
lower density. If $A,B \subset M$ are such that $A \ominus B \in \calN_{\rmpu,k}$ then $\calH^k(A \ominus B)=0$, recall (iii). Consequently, $\phi_M(A \cap \bB(x,r)) = \phi_M(B \cap \bB(x,r))$ for all $x \in \Rm$ and $r > 0$. Thus, $\Theta^k_*(A,x) = \Theta^k_*(B,x)$. Since $x$ is arbitrary, $\Theta_M(A) = \Theta_M(B)$. This proves condition of \ref{ldens}. Condition (3) of \ref{ldens} is trivial. In view of proving \ref{ldens}(3) we let $A \in \overline{\calB(\Rm)}_M$. According to condition (1) just proved and \eqref{eq.pu.measurable}, there is no restriction to assume that $A$ is Borel. We ought to show that the equation $\Theta^k_*(\phi_M \hel A ,x) = \ind_A(x)$ holds for $\calH^k$-almost every $x \in M$. Letting $\psi = \phi_M \hel A$, we infer from the Besicovitch Covering Theorem as in \cite[2.12]{MATTILA} that $\lim_{r  \to 0^+} \frac{\psi(\bB(x,r))}{\phi_M(\bB(x,r))} = \ind_A(x)$ for $\phi_M$-almost every every $x \in \Rn$. In view of \eqref{eq.pu.density.1}, it ensues that the sought 
for equation holds $\calH^k$-almost everywhere on $M$. To establish that $\Theta_M$ is a lower density, it remains to proves \ref{ldens}(4). Let $A, B  \in \overline{\calB(\Rm)}_M$. We observe that $\Theta_*^k(\phi_M \hel (A \cap B),x) \geq \Theta^k_*(\phi_M,x) - \Theta^{*\,k}(\phi_M \hel (M 
\setminus A),x) - \Theta^{*\,k}(\phi_M \hel (M \setminus B),x)$, for all $x \in M$. Now, as $A$ and $B$ are $\phi_M$-measurable, according to \eqref{eq.pu.measurable}, if $x \in \Theta_M(A) \cap \Theta_M(B)$, then it follows from \eqref{eq.pu.density.1} that $ \Theta^{*\,k}(\phi_M \hel (M \setminus A),x) =  \Theta^{*\,k}(\phi_M \hel (M \setminus B),x) = 0$ and, in turn, referring to \eqref{eq.pu.density.1} again, that $\Theta_*^k(\phi_M \hel (A \cap B),x)=1$. Thus, $x \in \Theta_M(A \cap B)$. We have 
shown that $\Theta_M(A) \cap \Theta_M(B) \subset \Theta_M(A \cap B)$. The 
other inclusion is trivial, so that $\Theta_M$ is, indeed, a lower density. 
\par 
(vii) Let $M, M' \in \calE$ and $A \in \overline{\calB(\Rm)}$ be such that $A \subset M \cap M'$. Notice that $A = A \cap M' = A \cap M$ and $\phi_M \hel A \cap M' = \phi_{M'} \hel A \cap M$. Therefore, if $x \in \Theta_M(A)$, then $1 = \Theta^k_*(\phi_M \hel A ,x) = \Theta^k_*(\phi_M \hel A \cap M,x) = \Theta^k_*(\phi_{M'} \hel A \cap M,x) = \Theta^k_*(\phi_{M'} \hel A,x)$. Since also $x \in M'$, we conclude that $x \in 
\Theta_{M'}(A)$. Switching the r\^oles of $M$ and $M'$ we conclude that $\Theta_M(A) = \Theta_{M'}(A)$.
\par 
It is interesting to try to understand the corresponding germ space. Each 
$(x,\bM) \in \widehat{\Rm}$ consists of a pair where $x \in \Rm$ belongs to the base space $\Rm$ and $\bM$ is an equivalence class of a $k$-dimensional submanifolds passing through $x$. If $M \ni x \in M'$ are two such submanifolds, then $M \sim_x M'$ if and only if 
\[
\lim_{r \to 0^+} \frac{\calH^k(M \cap M' \cap \bB(x,r))}{\balpha(k)r^k}= 
1 .
\]
This relation is finer than the usual notion of a germ of a $k$-dimensional submanifold passing through $x$. Of course if $M$ and $M'$ belong to the same, classically defined, germ, i.e. if there exists a neighborhood $V$ of $x$ in $\Rm$ such that $M \cap V = M' \cap V$, then $M \sim_x M'$. Notwithstanding, the following example illustrates the difference. Let $x \in \Rm$, let $W \subset \Rm$ be a $k$-dimensional affine subspace containing $x$, and let $C \subset W$ be closed with empty interior and such that $\Theta^k_*(\phi_W \hel C,x) =1$. Choose a $k$-dimensional submanifold $M \subset \Rm$ of class $C^1$ ``that sticks to $W$ exactly along 
$C$'', that is $W \cap M = C$. It follows that $W \sim_x M$, yet $M \cap V \neq W \cap V$, for every neighborhood $V$ of $x$. We note, however, that if $M \sim_x M'$, then $\rmTan(M,x) = \rmTan(M',x)$.
\par 
The construction here could be repeated by replacing $\calE$ by $\calE'$, 
the collection of all Borel measurable, countably $(\calH^k,k)$-rectifiable subsets $M$ of $X$ such that $\calH^k \hel M$ is locally finite, {\em and} $\Theta^k_*(\phi_M,x)=1$ for every $x \in M$. The latter does not hold in general for rectifiable sets, unlike the case of (embedded) submanifolds. It is critical when establishing that condition \ref{ldens}(4) holds. 
\end{Empty}

\begin{Empty}[Integral geometric measure]
\label{igm}
Here, we show that the methods of \ref{unrectifiablenull} apply, in
fact, to a special measure space. We keep the same notations as in
\ref{unrectifiablenull} and we let $\calI^k_\infty$ be the integral
geometric outer measure on $\Rm$ defined in \cite[2.10.5(1)]{GMT} or
\cite[5.14]{MATTILA}. The measure space $(\Rm,\calB(\Rm),\calI^k_\infty)$
is not semi-finite (for the case $1=k=m-1$, see
\cite[3.3.20]{GMT}). Thus, recalling \ref{sfv}, we introduce the
following:
\[
\check{\calI}^k_\infty(A) = \sup \left\{ \calI^k_\infty(A \cap B) : B \in \calB(\Rm), B \subset A \text{ and } \calI^k_\infty(B) < \infty \right\},
\]
for $A \in \calB(\Rm)$. The measure space $(\Rm,\calB(\Rm),\check{\calI}^k_\infty)$ is semi-finite, and $\check{\calI}^k_\infty(A)=0$ whenever $A \in \calB(\Rm)$ is purely $\calI^k_\infty$-infinite, i.e. $A$ itself and all its Borel subsets of nonzero measure have infinite measure. We denote by $(\Rm,\widetilde{\calB(\Rm)},\tilde{\calI}^k_\infty)$ its completion. Our goal is to describe its cccc, lld, and strictly localizable version. The corresponding MSN $(\Rm,\widetilde{\calB(\Rm)},\calN_{\tilde{\calI}^k_\infty})$ is readily saturated. We will check conditions (ii) through 
(vii) at the beginning of this section.
\par 
(ii) We claim that $\calE$ is
$\calN_{\tilde{\calI}^k_\infty}$-generating in the MSN
$(\Rm,\widetilde{\calB(\Rm)},\calN_{\tilde{\calI}^k_\infty})$, where
$\calE$ is as in \ref{unrectifiablenull}. We know that the collection
$\calA \defeq \widetilde{\calB(\Rm)} \cap \{ A :
\tilde{\calI}^k_\infty(A) < \infty \}$ is
$\calN_{\tilde{\calI}^k_\infty}$-generating, by \ref{semi-finite}. It
is easy to check that it suffices to establish the following: For
every $A \in \calA$ there is a sequence $\la M_n \ra_{n \in \N}$ in
$\calE$ such that $A \setminus \bigcup_{n \in \N} M_n \in
\calN_{\tilde{\calI}^k_\infty}$. Let $A \in \calA$. By definition of
completion of a measure space, there are $B \in \calB(\Rm)$, $N \in
\calN_{\check{\calI}^k_\infty}$, and $N' \subset N$ such that $A = B
\ominus N'$. Since $\tilde{\calI}^k_\infty(N')=0$, it suffices to
prove the existence of a sequence $\la M_n \ra_{n \in \N}$ in $\calE$
such that $B \setminus \bigcup_{n \in \N} M_n \in
\calN_{\tilde{\calI}^k_\infty}$. Since $\check{\calI}^k_\infty(B) =
\tilde{\calI}^k_\infty(B) = \tilde{\calI}^k_\infty(A) < \infty$, there
are Borel sets $F$ and $N$ such that $B = F \cup N$,
$\calI^k_\infty(F) < \infty$, and $\check{\calI}^k_\infty(N)=0$, by
\ref{sfv}(3). It follows from the Besicovitch Structure Theorem
\cite[3.3.14]{GMT} that $F$ is $(\calI^k_\infty,k)$-rectifiable. In
particular, there is a sequence $\la M_n \ra_{n \in \N}$ in $\calE$
such that $F \setminus \bigcup_{n \in \N} M_n \in \calN_{\calH^k}
\subset \calN_{\calI^k_\infty}$. Since $B \setminus F \in
\calN_{\check{\calI}^k_\infty} \subset
\calN_{\tilde{\calI}^k_\infty}$, the proof is complete.  \par In order
to establish (iii) through (vii), it suffices to observe that for each
$M \in \calE$ the MSNs
$(M,\overline{\calB(\Rm)}_M,(\calN_{\rmpu,k})_M)$ and $(M,
\widetilde{\calB(\Rm)}_M,(\calN_{\tilde{\calI}^k_\infty})_M)$ are the
same. We recall from \ref{unrectifiablenull}(iii) that the former is
the saturation of $(M,\calB(\Rm)_M,\calH^k \hel M)$. Let us prove that
the latter has the same property. Let $S \in
\widetilde{\calB(\Rm)}_M$. There are $B \in \calB(\Rm)$ and $N \in
\calN_{\tilde{\calI}^k_\infty}$ such that $S = B \ominus N$. Since $S
= S \cap M = (B \cap M) \ominus (M \cap N)$, there is no restriction
to assume that both $B$ and $N$ are contained in $M$. Therefore, we
ought to show that $\calH^k(N)=0$. There exists a Borel set $N'
\subset M$ containing $N$ and such that
$\check{\calI}^k_\infty(N')=0$. We observe that
$\check{\calI}^k_\infty \hel M = \calI^k_\infty \hel M = \calH^k \hel
M$, where the second equality follows from \cite[3.2.26]{GMT} and
first equality follows from \ref{sfv}(2) and the fact that $M$ has
$\sigma$-finite $\calI^k_\infty$ measure. Thus, $\calH^k(N')=0$ and we
are done.  \par It follows that the germ space
$(\widehat{\Rm},\hat{\calA},\hat{\calN})$ constructed in
\ref{unrectifiablenull} is, in fact, also the cccc and lld version of
the MSN
$(\Rm,\widetilde{\calB(\Rm)},\calN_{\tilde{\calI}^k_\infty})$. Furthermore,
if $\hat{\calI}^k_\infty$ denotes the pre-image measure of
$\tilde{\calI}^k_\infty$ along the projection map $p \colon
\widehat{\Rm} \to \Rm$, then
$(\widehat{\Rm},\widehat{\widetilde{\calB(\Rm)}},{\hat{\calI}^k_\infty})$
is the strictly localizable version of
$(\Rm,\widetilde{\calB(\Rm)},\tilde{\calI}^k_\infty)$.
\end{Empty}

\begin{Empty}[Hausdorff measures]
Here, we briefly comment on why the lower densities set up so far in
this section do not help to describe explicitly the cccc and lld
version of the saturation of the MSN
$(\Rm,\calB(\Rm),\calN_{\calH^k})$. The main reason is that we would
need to enlarge the collection $\calE$ for it to be generating, since
there are (much) less $\calH^k$-negligible sets than there are purely
$k$-unrectifiable sets. In doing so we loose \eqref{eq.pu.density.1},
which was critical for implementing the techniques of the previous
section. In fact, if $M \subset \Rm$ is Borel, $\phi_M = \calH^k \hel M$
is locally finite, and $\Theta^k(\phi_M,x)=1$ for $\calH^k$-almost
every $x \in M$, then $M$ is countably $(\calH^k,k)$-rectifiable, see
{\it e.g.} \cite[17.6(1)]{MATTILA}. Since we ought to include non
$\calH^k$-negligible, purely $k$-unrectifiable sets in an
$\calN_{\calH^k}$-generating family, our only choice is, if possible,
to change the definition of the lower densities $\Theta_M$. So far, we
do not know how to construct, in this case, a compatible family of
lower densities.
\end{Empty}


\bibliographystyle{amsplain}
\bibliography{../../../Bibliography/thdp.bib}


\end{document}